\documentclass[11pt]{article}
\usepackage[T1]{fontenc}
\usepackage{lmodern}
\usepackage[margin=1in]{geometry}
\usepackage{amsmath,amssymb,amsthm,mathtools,mathrsfs}
\usepackage{microtype,booktabs,needspace,graphicx}
\usepackage{cite,xcolor}
\definecolor{refblue}{RGB}{38,96,146}
\usepackage[colorlinks=true,allcolors=refblue]{hyperref}
\usepackage{bookmark,xurl}
\newtheorem{theorem}{Theorem}
\newtheorem{conjecture}{Conjecture}
\newtheorem{lemma}{Lemma}[section]
\newtheorem{proposition}[lemma]{Proposition}
\newtheorem{corollary}{Corollary}
\theoremstyle{definition}
\theoremstyle{remark}
\newtheorem{example}[lemma]{Example}
\numberwithin{equation}{section}
\newcommand{\E}{\mathbb E}\newcommand{\PP}{\mathbb P}
\newcommand{\R}{\mathbb R}
\newcommand{\T}{\mathbb T^2}\newcommand{\dd}{\,\mathrm d}
\newcommand{\Id}{\operatorname{Id}}\newcommand{\Dir}{\operatorname{Dir}}
\newcommand{\Bin}{\operatorname{Bin}}\newcommand{\Pois}{\operatorname{Pois}}

\newcommand{\diver}{\operatorname{div}}
\newcommand{\dist}{\operatorname{dist}}\newcommand{\diag}{\operatorname{Diag}}
\newcommand{\OT}{\operatorname{OT}}
\newcommand{\HS}{\mathrm{HS}}\newcommand{\1}{\mathbf1}
\newcommand{\cL}{\mathcal L}
\newcommand{\cB}{\mathcal B}

\newcommand{\Z}{\mathcal Z}

\allowdisplaybreaks[2]
\title{Random matching between distinct distributions in two dimensions}
\author{Yunzong Xu\thanks{University of Illinois Urbana--Champaign. Email: \href{mailto:xyz@illinois.edu}{\texttt{xyz@illinois.edu}}.}}
\date{}
\hypersetup{pdftitle={Random matching between distinct distributions in two dimensions},pdfauthor={Yunzong Xu}}

\begin{document}
\maketitle
\begin{abstract}
Given two independent $n$-point samples drawn from possibly distinct
probability distributions on a compact 2-dimensional surface, we minimize
the total squared distance over all perfect bipartite matchings. Under
regularity assumptions, the expected matching cost minus $n$ times the
squared 2-Wasserstein distance between the sampling distributions grows
logarithmically with $n$. We determine this difference up to an additive
$o(1)$ error, including the dependence of the logarithmic coefficient and
constant term on geometry and sampling distributions, and establish the
existence of a universal constant shared by these formulas.

Our analysis expresses the expected matching cost in terms of the cost
of transporting a heat-smoothed sample back to its original points. The
reductions giving this relation hold in every dimension and for general
costs under suitable assumptions, with consequences beyond the setting
studied here. In two dimensions, the relation is exact up to an additive
$o(1)$ error, so determining the matching cost amounts to evaluating the
smoothing cost with the same accuracy. We obtain this evaluation by
subdividing the domain and exploiting cancellations between the sample
and its own smoothing.
\end{abstract}

\section{Introduction}\label{sec:introduction}
Draw $n$ points independently from each of two probability distributions
$\mu$ and $\nu$ on a 2-dimensional surface. We call the points sampled
from $\mu$ the sources and those sampled from $\nu$ the targets. The
\emph{bipartite matching problem} pairs each source with a distinct target
to minimize the total cost. Both distance and squared distance are widely
studied as matching costs. Unless otherwise specified, we use squared
distance as the cost (quadratic cost) and write $M_n(\mu,\nu)$ for the
minimum total cost.

The corresponding \emph{population problem} minimizes the average squared
distance for transporting one unit of mass from $\mu$ to $\nu$. Its value
is $W_2^2(\mu,\nu)$, the square of the 2-Wasserstein distance. The
baseline for $n$ pairs is therefore $nW_2^2(\mu,\nu)$. The
\emph{excess matching cost}
$\E M_n(\mu,\nu)-nW_2^2(\mu,\nu)$ measures the additional expected cost
caused by sampling. When $\mu=\nu$, the population cost vanishes and this
is the entire expected matching cost. The goal of this paper is to
determine the excess matching cost up to an error that tends to zero as
$n\to\infty$.

Random matching has been studied through probability and empirical
process theory~\cite{AKT,BobkovLedoux}, and through optimal
transport~\cite{AST,Villani}. The per-pair matching cost $M_n/n$ equals
the squared 2-Wasserstein distance between the two empirical distributions,
each placing mass $1/n$ at every point of its sample.
Thus $M_n/n$ estimates $W_2^2(\mu,\nu)$, and the excess matching cost is
$n$ times its bias.

\subsection{The statistical-physics prediction}\label{subsec:physics}
The statistical-physics approach of Caracciolo, Lucibello, Parisi, and
Sicuro~\cite{CLPS}, developed further by Caracciolo and Sicuro~\cite{CS},
approximates random matching by an electrostatic problem. Source and
target points play the roles of charges of opposite sign. After smoothing
these point sets, each with total mass one, their density difference
becomes the right-hand side of a Poisson equation. The
gradient of its solution describes movement from regions with more
source mass to regions with more target mass. For uniform sampling on an
area-one space, multiplying the integral of this squared movement by $n$
gives the predicted total matching cost.

These statistical-physics predictions concern samples drawn from a common
distribution. A closely related problem transports a continuous
distribution to $n$ independent target points drawn from that same
distribution, each receiving mass $1/n$. This
\emph{semidiscrete} problem has also been extensively
studied~\cite{AST,AG}. Let $\mu_\Omega$ be uniform probability measure on an
area-one geometry $\Omega$, and let $Y_1,\ldots,Y_n$ be drawn independently
from $\mu_\Omega$. We write
\[
 C_n^\Omega=\E M_n(\mu_\Omega,\mu_\Omega),\qquad
 S_n^\Omega=n\E W_2^2\!\left(
       \mu_\Omega,\frac1n\sum_{i=1}^n\delta_{Y_i}\right)
\]
for the \emph{bipartite mean cost} and the \emph{semidiscrete mean cost},
respectively. The factor $n$ makes the two expected costs comparable.

The Poisson equation approach predicts
$C_n^\Omega\sim(\log n)/(2\pi)$.\footnote{For comparison, with unsquared distance
cost, Ajtai, Koml\'os, and Tusn\'ady~\cite{AKT}
established the $\sqrt{n\log n}$ order for the expected minimum total cost
between two independent $n$-point samples drawn uniformly from the unit square.}
Ambrosio, Stra, and
Trevisan~\cite{AST} proved this leading law. Ambrosio and
Glaudo~\cite{AG} sharpened the estimates and extended the bipartite law to
closed surfaces; Ambrosio, Goldman, and Trevisan~\cite{Domain} established
the leading law for general 2-dimensional domains and densities. For
uniform sampling on spaces of the same area, the leading coefficient is
independent of geometry. The effect of geometry on the expected cost
therefore lies in the lower-order terms.

Benedetto, Caglioti, Caracciolo, D'Achille, Sicuro, and
Sportiello~\cite{BCCDS} predict that this geometric contribution is
determined by the Laplace spectrum. Let
$(\lambda_j^\Omega)_{j\ge1}$ be the positive eigenvalues of
$-\Delta_\Omega$. Write
\begin{equation}\label{eq:spectral-finite-part}
 F_\Omega(t)=\sum_{j\ge1}\frac{e^{-t\lambda_j^\Omega}}{\lambda_j^\Omega},
 \qquad
 h_\Omega=\lim_{t\downarrow0}\left\{
       F_\Omega(t)-\frac1{4\pi}\log(1/t)\right\}.
\end{equation}
The \emph{spectral finite part} $h_\Omega$ is the constant left after
subtracting the logarithmic divergence of $F_\Omega(t)$.\footnote{The spectral constant used in~\cite{BCCDS} differs from
$h_\Omega$ by a constant independent of geometry. This shift is absorbed
into the universal remainder and does not affect comparisons between
geometries.}
On the square, the Laplacian has Neumann boundary conditions.

After subtracting the logarithmic terms and the appropriate multiples of
$h_\Omega$, the conjecture
in~\cite[Conjecture~(47)--(48) and Remark~1]{BCCDS} leaves remainders that
are independent of geometry and may grow with $n$; the reported numerical
results suggest finite limits. Constant limits were predicted for the
torus in~\cite{CLPS} and the square in~\cite{CS}. We summarize these
predictions below for the geometries considered here.

\begin{conjecture}[Spectral universality {\cite{CLPS,CS,BCCDS}}]\label{conj:physics}
For uniform sampling on a compact connected surface $\Omega$ of unit area, allowing a smooth boundary or convex corners, the expected costs $C_n^\Omega,S_n^\Omega$ satisfy the following assertions.
\begin{enumerate}
\item \emph{Dependence on geometry.} There are functions $r,r_{\mathrm{sd}}=o(\log n)$, independent of $\Omega$, such that, for each fixed $\Omega$,
\begin{align*}
 C_n^\Omega&=\frac{\log n}{2\pi}+r(n)+2h_\Omega+o(1),\\
 S_n^\Omega&=\frac{\log n}{4\pi}+r_{\mathrm{sd}}(n)+h_\Omega+o(1).
\end{align*}
\item \emph{Constant limits.} These functions converge to finite constants:
\[
 r(n)\longrightarrow\beta,
 \qquad r_{\mathrm{sd}}(n)\longrightarrow\beta_{\mathrm{sd}}.
\]
\end{enumerate}
\end{conjecture}

In the first assertion, the common remainder cancels when two geometries
are compared: $C_n^\Omega-C_n^{\Omega'}$ tends to
$2(h_\Omega-h_{\Omega'})$, even if $r(n)$ grows. The second assertion says
that the common remainder itself has a finite limit.

\subsection{Main results}\label{subsec:main}
We prove the mean-cost formula predicted in Conjecture~\ref{conj:physics}
and extend it to distinct source and target distributions. The population
optimal transport map
determines a new metric governing the excess matching cost. Its area can
differ from that of the original surface, changing even the logarithmic
coefficient.

We first take $\Omega$ to be the closure of a bounded convex domain in
$\R^2$, with Euclidean distance, or a flat torus $\R^2/\Gamma$, with
geodesic distance, where $\Gamma$ is a full-rank lattice. We allow arbitrary
area $|\Omega|$. The unit square $Q=[0,1]^2$ and unit flat torus
$\T=\R^2/\mathbb Z^2$ are particular examples.
Let $\mu=p\,dx$ and $\nu=q\,dx$ be the sampling distributions on $\Omega$.
Let $T$ be the population optimal transport map from $\mu$ to $\nu$,
and write $H=DT$ for its Jacobian matrix. When $T$ is smooth and $H$ is positive definite,
this symmetric matrix defines a metric. Its area
measure, total area, and sampling density are
\[
 dA_H=\sqrt{\det H}\,dx,\qquad
 A_H=\int_{\Omega}dA_H,\qquad
 \rho_H=\frac{p}{\sqrt{\det H}},
\]
so that $\mu=\rho_H\,dA_H$. Write $h_H$ for the spectral finite part
for this metric and density, defined by \eqref{eq:h-general} for
$(\Omega,H)$ with sampling law $\mu$.

\begin{theorem}[Mean-cost formula for matching between distinct distributions]\label{thm:source-target}
Let $\Omega$ be a flat torus or the closure of a bounded convex domain
in $\R^2$ with piecewise smooth boundary and corner angles in $(0,\pi)$.
Let $\mu=p\,dx$ and $\nu=q\,dx$ have strictly positive probability
densities, and suppose their population optimal transport map $T$
is a diffeomorphism. Assume $p,q,T$ are smooth up to the boundary
when it is present, and $H=DT$ is uniformly positive definite.
Then, for a universal constant $\beta\in\R$,
\begin{equation*}
 \begin{split}
 \E M_n(\mu,\nu)-nW_2^2(\mu,\nu)
       ={}&\frac{A_H}{2\pi}\log n+A_H\beta
       +2h_H+\frac1{2\pi}\int_{\Omega}\log\rho_H\,dA_H\\
       &+O_{\Omega,p,q}(n^{-1/180}\log(en)).
 \end{split}
\end{equation*}
\end{theorem}

On flat tori, smooth positive densities imply the required regularity of $T$.
The term $A_H\beta$ is the area multiplied by the same universal constant
that appears for matching between samples drawn uniformly from the square and torus
(Corollary~\ref{cor:flat-mean}). We call $2h_H$ the \emph{spectral
correction} and the integral term the \emph{density correction}.
The spectral correction depends on both $H$ and $\rho_H$.

\begin{corollary}[The logarithmic coefficient]\label{cor:effective-area}
Under the assumptions of Theorem~\ref{thm:source-target},
$0<A_H\le|\Omega|$, where $|\Omega|$ is the original area, with
equality if and only if $\mu=\nu$.
\end{corollary}
Distinct distributions therefore reduce the logarithmic coefficient of
the excess matching cost, after subtracting the population transport
cost $nW_2^2(\mu,\nu)$.

Under the corresponding population transport assumptions,
Theorem~\ref{thm:source-target} extends to compact surfaces with smooth
boundary, convex corners, or no boundary.
Corollary~\ref{cor:surface-mean} gives the mean-cost formula for samples
from a common distribution on these surfaces.
The distinct-to-same-distribution reduction (Theorem~\ref{thm:distinct-to-same})
then transfers this formula to the metric induced by population optimal
transport. Both results also give semidiscrete counterparts.
Here the continuous source distribution $\mu$ is transported to the
empirical target distribution $\nu_n=n^{-1}\sum_j\delta_{Y_j}$, where
$Y_1,\ldots,Y_n$ are drawn independently from $\nu$. The excess
semidiscrete mean cost is
\[
 n\{\E W_2^2(\mu,\nu_n)-W_2^2(\mu,\nu)\}.
\]
It satisfies the corresponding bipartite formula with $\beta$ replaced by
$\beta_{\mathrm{sd}}$ and with the logarithmic term, spectral correction,
and density correction halved.

For a fixed pair of distributions, the population optimal transport map
and its geometric and spectral quantities are computed once. These inputs
enter our formula for the expected excess cost and do not
depend on the sample size or realization. The universal constant $\beta$
can be estimated once from uniform sampling on the unit square and then
used for any pair of sampling distributions satisfying the assumptions
above, including nonuniform ones.

\subsection{Relation to prior work}\label{subsec:prior}
The statistical optimal transport literature studies the bias and
fluctuations of empirical transport costs. For a common smooth positive density
on the square or a flat torus, the semidiscrete mean cost is
$O(\log n)$~\cite{Domain}. Combined with the stability bounds of
Manole, Balakrishnan, Niles-Weed, and
Wasserman~\cite{MBNW}, this gives an $O(\log n)$
bound on the excess bipartite matching cost on the square or a flat torus,
under the regularity assumptions of Theorem~\ref{thm:source-target}.
Theorem~\ref{thm:source-target} sharpens this bound to a precise asymptotic formula,
identifying both the logarithmic coefficient and the constant term with
an $o(1)$ error. Equivalently, it determines the bias of
$W_2^2(\widehat\mu_n,\widehat\nu_n)$ as an estimator of
$W_2^2(\mu,\nu)$ up to an $o(1/n)$ error, where
$\widehat\mu_n,\widehat\nu_n$ are the two empirical distributions.\footnote{For this
estimator,~\cite{MBNW} establishes an $n^{-1/2}$-scale central limit
theorem for distinct distributions in dimension two under regularity
assumptions. Under squared Euclidean cost,~\cite{MNW} gives an
$O(n^{-1/2})$ expected absolute-error bound for arbitrary compactly
supported distributions in $\R^2$.}

For samples from a common distribution, our mean-cost formula proves both
assertions of Conjecture~\ref{conj:physics} with a polynomial convergence
rate. In particular, it proves the flat-torus constant-limit prediction
of~\cite{CLPS}, recalled as an open problem
in~\cite{AST}. For semidiscrete transport on
the flat torus, Goldman, Huesmann, and Otto~\cite{GHO} bounded the
remainder after subtracting the logarithmic term by $O(\log\log n)$;
we prove that this remainder converges and identify its geometric part.
Our semidiscrete-to-smoothing reduction expresses the semidiscrete mean cost
as $n$ times the expected smoothing cost studied in~\cite{AG}, plus an explicit
spectral term and an $o(1)$ error. Subdivision then evaluates the smoothing cost through
its constant term.

When two unit-area geometries are compared, the universal constant
cancels. For the square and torus, this confirms the prediction
in~\cite{BCCDS}:
\[
 C_n^Q-C_n^{\T}=\frac13+\frac{\log2}{\pi}
                    +o(1).
\]
Thus the expected-cost difference is explicit even though $\beta$
remains unevaluated.

Theorem~\ref{thm:distinct-to-same} extends the formula to general
costs satisfying its population regularity and nondegeneracy assumptions
(Corollary~\ref{cor:general-cost-mean}).
For example, with cost $|x-y|^r$, every $r>1$ gives logarithmic excess
cost when one smooth distribution is a fixed nonzero translate of the
other (Example~\ref{ex:power-translation}). Armegioiu, Goldman, Grotto,
and Trevisan~\cite{AGGT} study leading asymptotics for power costs with
a common sampling distribution. For unsquared distance cost,
distinct distributions can
even change the order of the excess cost. Proposition~\ref{prop:linear-product}
gives an example of two distinct smooth positive densities on the square
with excess cost at most $O(\log n)$, below the classical $\sqrt{n\log n}$
order for two samples drawn uniformly from the square~\cite{AKT}.
An analogue of our formula for unsquared distance cost will require new ideas
to explain how distinct sampling distributions determine the leading
order of the excess cost.

\subsection{Key ideas: smoothing cost as a proxy for matching cost}\label{subsec:ideas}
Our key observation is that transporting a heat-smoothed sample back to its
original points gives a different optimal transport problem that is easier
to analyze, yet whose expected cost determines the expected matching cost
to additive $o(1)$ accuracy.
Here \emph{heat smoothing} replaces a sample of point masses by a smooth
density, making the Poisson equation approach of
Section~\ref{subsec:physics} applicable~\cite{AST,AG}.
The \emph{heat time} $t$ controls the spreading distance, of order $\sqrt t$.
The minimum cost of transporting this density back to the original points
is the \emph{smoothing cost}. Prior work~\cite{AST,AG} already studies this
cost, choosing heat times that make it negligible in the leading-order
analysis.
Our approach reveals a different possibility: in two dimensions, the
smoothing cost can contribute at leading order while still determining
the expected total matching cost to additive $o(1)$ accuracy after adding
an explicit spectral term.
We can therefore use stronger smoothing to obtain accurate Poisson
equation estimates, while retaining the nonlinear transport contribution
and evaluating it through the smoothing cost alone.

This would be of limited use if the smoothing cost were just as difficult
to analyze as the original matching cost. We find, however, that the sample
and its own smoothing share local excesses of mass on scales much larger
than the smoothing radius. Their cancellation makes \emph{subdivision}
particularly effective: we divide the domain into smaller regions and
compare the expected total transport costs on the whole domain and within
those regions, with additive $o(1)$ error. Together with comparisons across
heat times, these relations determine the mean-cost formula without
explicitly identifying the optimal allocations. The contributions below
explain these reductions and comparisons and their extension to distinct
sampling distributions.

\paragraph{The bipartite-to-semidiscrete reduction.}
Our first contribution is a general averaging identity
(Theorem~\ref{thm:averaging}) that puts bipartite matching and semidiscrete
transport on the same footing: the bipartite mean cost equals $n$
times the expected minimum cost of transporting one unit of mass from
the source distribution to the target points, with random Dirichlet
shares of mass assigned to those targets. The identity holds in every
dimension and for general integrable costs.

For quadratic cost, let $Y_1,\ldots,Y_n$ be independent targets with
law $\nu$, and let $W$ be independent of the targets and uniform on the
simplex $\{w_i\ge0:\sum_iw_i=1\}$, the $\Dir(1,\ldots,1)$ law.
The averaging identity gives
\[
 \E M_n(\mu,\nu)
 =n\E W_2^2\!\left(\mu,\sum_iW_i\delta_{Y_i}\right).
\]
We call the right-hand side a \emph{weighted semidiscrete mean cost}.
Equal target masses
$W_i=1/n$ give the usual (unweighted) semidiscrete mean cost. Our subsequent
2-dimensional analysis therefore treats a continuous source distribution and either
Dirichlet or equal target masses.

\paragraph{The semidiscrete-to-smoothing reduction.}
Our second contribution gives a precise relation between the (weighted) semidiscrete
mean cost and the expected smoothing cost, with an explicit spectral
term and a controlled error (Theorem~\ref{thm:smoothing-reduction}). This reduction
holds in every dimension on manifolds without boundary for regular costs
with a positive quadratic part as specified in~\eqref{eq:smoothing-cost-assumptions},
with boundary extensions in Section~\ref{subsec:surface-reduction}.
To explain its use in two dimensions, take
$\mu=\nu=\mu_\Omega$ on the unit square or unit torus and use the
Dirichlet target measure $\nu_n=\sum_iW_i\delta_{Y_i}$ above.
Let $P_t$ denote heat smoothing with heat time $t$.
We therefore evaluate
$n\E W_{2,\Omega}^2(P_t\nu_n,\nu_n)$.

The classical approach uses the solution to a Poisson equation, called
the \emph{Poisson potential}, to approximate
$W_{2,\Omega}^2(\mu_\Omega,P_t\nu_n)$, the transport cost between
two continuous distributions~\cite{AST,AG}.
We use the same Poisson potential to evaluate the \emph{source-change effect}
$W_{2,\Omega}^2(\mu_\Omega,\nu_n)-W_{2,\Omega}^2(P_t\nu_n,\nu_n)$:
although both costs involve nonlinear transport to the unsmoothed target
$\nu_n$, we show that their nonlinear contributions cancel in the
difference to the required accuracy. Taking expectations gives
\[
 C_n^\Omega=n\E W_{2,\Omega}^2(P_t\nu_n,\nu_n)
            +2\{2F_\Omega(t)-F_\Omega(2t)\}+o(1),
 \qquad t\to0,\quad \frac{nt}{(\log n)^3}\to\infty.
\]
The explicit spectral formula~\eqref{eq:spectral-finite-part} gives
$2\{2F_\Omega(t)-F_\Omega(2t)\}
=(2\pi)^{-1}\log(2/t)+2h_\Omega+o(1)$.

\paragraph{Subdivision of the smoothing cost.}
Our third contribution evaluates the smoothing cost to additive $o(1)$
accuracy in two dimensions (Theorem~\ref{thm:smoothing-cost}).
To compare different sample sizes, we partition the domain into smaller
squares, called \emph{cells}, and smooth and transport within each cell
separately. We prove that this restriction changes the expected total
cost by only $o(1)$ at a suitable heat time
(Lemma~\ref{g:lem:grid}).

Goldman and Trevisan~\cite{GT} and Ambrosio, Goldman, and
Trevisan~\cite{Domain} use subdivision to evaluate expected matching
costs and obtain their leading-order asymptotics. We apply subdivision
directly to the smoothing cost, exploiting the dependence between $\nu_n$
and its \emph{own} smoothing $P_t\nu_n$: their shared excesses of mass over
$\mu_\Omega$ largely cancel in $\nu_n-P_t\nu_n$, allowing us to bound
the additional cost of confining transport to cells by $o(1)$.

Rescaling a cell gives the same smoothing problem with a smaller sample
size and a different heat time. The
change-of-time formula~\eqref{f:eq:time-change}, obtained by subtracting
the semidiscrete-to-smoothing reduction at two times, removes the latter
change. The resulting relation involves only sample size and yields
\[
 n\E W_{2,\Omega}^2(P_t\nu_n,\nu_n)
 =\frac{\log(nt)}{2\pi}+\kappa_\star+o(1)
\]
throughout the range above, with the same constant on the square and
torus. This also characterizes the effect of heat
smoothing in Wasserstein distance independently of random matching.
Equal target masses admit the same analysis. Shrinking cells extend
this evaluation to smooth positive densities on surfaces, allowing
smooth boundaries and convex corners (Theorem~\ref{thm:surface-smoothing}).
The semidiscrete-to-smoothing reduction then gives the mean-cost
formulas, with $\beta=\kappa_\star+(\log2)/(2\pi)$ in the Dirichlet case,
and proves both assertions of Conjecture~\ref{conj:physics}.

\paragraph{The distinct-to-same-distribution reduction.}
Our fourth contribution reduces matching between distinct distributions
in two dimensions to a same-distribution problem in an induced metric, after subtracting
the population cost, with additive $o(1)$ error
(Theorem~\ref{thm:distinct-to-same}).
Suppose the population optimal transport map $T$ from $\mu$ to $\nu$
is smooth and invertible. Pulling each target back to $Z_j=T^{-1}Y_j$
gives two independent samples from $\mu$. Writing $d(x,y)$ for distance
on the surface, the population dual potentials $a,b$ satisfy
$a(x)+b(y)\le d(x,y)^2$, with equality at $y=T(x)$. Subtracting these
potentials from every matching removes $nW_2^2(\mu,\nu)$ exactly in
expectation. The remaining cost of matching $x$ to the pulled-back target
$z$ is the nonnegative \emph{population dual gap}
\[
 k(x,z)=d(x,Tz)^2-a(x)-b(Tz).
\]
Under the assumptions of Theorem~\ref{thm:distinct-to-same}, this cost
vanishes only at $x=z$ and its quadratic part defines a metric $H$.
In Theorem~\ref{thm:source-target}, $H=DT$.
Replacing $k$ by squared distance in this metric changes the expected
total cost by $o(1)$, completing the reduction. Applying the
same-distribution formula of Corollary~\ref{cor:surface-mean} in this
induced metric gives
Corollary~\ref{cor:general-cost-mean}.

\Needspace{16\baselineskip}
\subsection{Resolution of the Barthe--Bordenave conjecture for quadratic cost}\label{subsec:highdim}
The averaging identity (Theorem~\ref{thm:averaging}) holds in every
dimension and for general integrable costs. The following application illustrates its usefulness
independently of the 2-dimensional mean-cost formula.

For quadratic cost in dimensions $d\ge5$, Barthe and Bordenave~\cite{BB}
obtained upper and lower bounds on the leading-order matching cost for
samples from a common distribution. These bounds involve two constants, defined by matching in a
cube and by a relaxation that permits transport to its boundary. They
conjectured that these constants coincide, giving a single leading-order
law for general densities. Our averaging identity, combined with the
semidiscrete boundary comparison of Huesmann, Goldman, and
Trevisan~\cite[Theorem~4.3]{GHT}, proves this equality. Consequently,
\[
 n^{2/d-1}\min_{\sigma\in S_n}\sum_i|X_i-Y_{\sigma(i)}|^2
 \longrightarrow\beta_d\int_{\R^d}f^{1-2/d}\,dx
\]
almost surely and in every finite $L^q$, for two independent $n$-point samples with any
compactly supported probability density $f$. The constant $\beta_d>0$
depends only on dimension. The proof is given in Appendix~\ref{app:highdim}.

\subsection{Organization and notation}
Section~\ref{sec:averaging} proves the bipartite-to-semidiscrete reduction
for integrable costs and arbitrary source distributions.
Section~\ref{sec:smoothing-reduction} establishes the
semidiscrete-to-smoothing reduction for regular costs in every dimension,
then specializes it to two-dimensional surfaces, allowing boundaries
and convex corners. Section~\ref{sec:smoothing-evaluation} evaluates
the smoothing cost by subdivision in these settings, for equal or
Dirichlet masses. Section~\ref{sec:source-target} derives the same-distribution
mean-cost formulas and proves the distinct-to-same-distribution reduction.
Its example for unsquared distance cost shows
why that case requires a different argument.

For an integrable real-valued function $f$, we write $\nu(f)=\int f\,d\nu$.
We write $\operatorname{Id}$ for the identity map (acting on points) and $I$
for the identity matrix or operator (acting on vectors or functions).
In local coordinates $x=(x_1,\ldots,x_d)$, we write
$DT=(\partial T_i/\partial x_j)_{i,j}$ for the Jacobian of a map $T$
and $D^2f=(\partial^2f/\partial x_i\partial x_j)_{i,j}$ for the Hessian
matrix of $f$. A subscript specifies the argument
being differentiated: $D_x^2f(x,y)$ differentiates in $x$, holding $y$ fixed.
On a Riemannian manifold $(M,g)$, we write $\nabla f$, $\nabla^2f$, and
$\Delta f=\operatorname{div}(\nabla f)$ for the gradient, Riemannian Hessian, and
Laplacian defined by $g$.
In limits as $n\to\infty$, the geometry and sampling distributions are fixed.
For real sequences $f(n),g(n)$ with $g(n)>0$, $f=O(g)$ means
$|f(n)|\le Cg(n)$ for all $n\ge1$, with $C$ independent of $n$.
For positive $f,g$, we write $f\asymp g$ if $cg(n)\le f(n)\le Cg(n)$
for all $n\ge1$ and constants $0<c\le C$ independent of $n$.
Subscripts on $O$ specify allowed dependence: for example, in
$O_{\Omega,p,q}(g)$, $C$ may depend on $\Omega,p,q$.
We write $f=o(g)$, equivalently $f\ll g$ or $g\gg f$, if
$f(n)/g(n)\to0$ as $n\to\infty$, and $f\sim g$ if $f(n)/g(n)\to1$.

\section{Bipartite-to-semidiscrete reduction}\label{sec:averaging}
This section concerns the problem of matching $n$ independent source points with law $\mu$ to $n$ fixed targets.
We prove that their expected optimal matching cost equals $n$ times the
expected minimum cost of allocating $\mu$ among those targets with
Dirichlet masses. The argument allows arbitrary source distributions
and integrable costs. We first define the allocation cost for prescribed
target masses.

Label the targets $1,\ldots,n$ and fix measurable cost functions
$c_1,\ldots,c_n:\mathcal X\to\R$, where $c_j(x)$ is the cost of sending
source $x$ to target $j$. Let $\mu$ be a probability measure on $\mathcal X$ with
$\E_\mu\max_j|c_j|<\infty$. For an $n$-dimensional target mass vector
$w=(w_1,\ldots,w_n)$ with $w_j\ge0$ and $\sum_jw_j=1$, define
\begin{equation}\label{eq:allocation-cost}
 \mathsf T(\mu,w)=
 \inf_{\substack{\alpha_j\ge0,\ \sum_j\alpha_j=1\ \mu\text{-a.e.}\\
                  \mu(\alpha_j)=w_j}}
       \int\sum_{j=1}^n c_j(x)\alpha_j(x)\,d\mu(x).
\end{equation}
Here $\alpha_j(x)$ is the fraction of source mass at $x$ sent to label $j$,
and $\mu(\alpha_j)=\int\alpha_j\,d\mu$ is the total mass that label receives.
Allowing fractions makes the definition valid even when $\mu$ has atoms.
For fixed targets $y_1,\ldots,y_n$ and quadratic costs
$c_j(x)=d(x,y_j)^2$, $\mathsf T(\mu,w)$ is the squared Wasserstein transport
cost from $\mu$ to the weighted target measure $\sum_jw_j\delta_{y_j}$.
Both measures have total mass one, so we compare $n\mathsf T$ with the
total cost of $n$ matched pairs.

We choose the mass vector uniformly from the simplex
$\{w_j\ge0:\sum_jw_j=1\}$. This is the Dirichlet distribution with all
parameters equal to one, denoted by $\Dir(1,\ldots,1)$.

\begin{theorem}[Exact averaging identity]\label{thm:averaging}
Let $\mu$ be a probability measure and $c_1,\ldots,c_n\in L^1(\mu)$.
Let $X_1,\ldots,X_n$ be independent with law $\mu$, and let
$W\sim\Dir(1,\ldots,1)$. Then
\begin{equation}\label{eq:averaging}
 \E\min_{\sigma\in S_n}\sum_i c_{\sigma(i)}(X_i)
       =n\E_W\mathsf T(\mu,W).
\end{equation}
In particular, for fixed targets $y_1,\ldots,y_n$ and integrable quadratic costs,
\[
 \E_X\min_{\sigma\in S_n}\sum_i d(X_i,y_{\sigma(i)})^2
 =n\E_W W_2^2\!\left(\mu,\sum_jW_j\delta_{y_j}\right).
\]
\end{theorem}
The costs $(c_1(X_i),\ldots,c_n(X_i))$ for one source may be dependent;
atoms and ties are allowed. When the targets are random, sample the
Dirichlet weights independently and then average over the target sites.
This gives the bipartite mean cost.

\subsection{Proof of the averaging identity}
Fix $n-1$ source points. For each possible location $x$ of the final source,
record the target matched to $x$ in an optimal matching. These choices
define a rule for allocating $\mu$ among the targets. We show that this
rule solves the semidiscrete transport problem with the resulting target
masses, relate its expected cost to the expected total matching cost,
and identify the Dirichlet law of the masses.

\begin{proof}[Proof of Theorem~\ref{thm:averaging}]
The case $n=1$ is immediate. For $n\ge2$, first assume that each difference
$c_j(X)-c_k(X)$, for $X\sim \mu$ and $j\ne k$, has an atomless law.
We remove this assumption by a small random perturbation at the end.

Take $X_1,\ldots,X_{n-1}$ as a base. For each target $j$, let $D_j$
be the minimum cost of assigning the base sources to all targets except
$j$. If the added source $x$ uses target $j$, the full assignment has
minimum cost $c_j(x)+D_j$. Its optimal target is therefore given by the
map $j:\mathcal X\to\{1,\ldots,n\}$,
\begin{equation}\label{eq:insertion}
 j(x)=\operatorname*{argmin}_{1\le r\le n}\{c_r(x)+D_r\}.
\end{equation}
The minimizer is unique $\mu$-almost everywhere, almost surely over the base.
The regions $\{x:j(x)=j\}$ are generalized \emph{Laguerre cells}~\cite{KMT};
write $A_j=\mu\{x:j(x)=j\}$ for their masses and $A=(A_1,\ldots,A_n)$.

We now show that the rule $x\mapsto j(x)$ achieves the minimum transport
cost $\mathsf T(\mu,A)$. For any feasible allocation
$(\alpha_1,\ldots,\alpha_n)$ in~\eqref{eq:allocation-cost} with $w=A$,
\[
\begin{aligned}
 \int c_{j(x)}(x)\,d\mu(x)
 &=\int\min_j\{c_j(x)+D_j\}\,d\mu(x)-\sum_jD_jA_j\\
 &\le\int\sum_j[c_j(x)+D_j]\alpha_j(x)\,d\mu(x)-\sum_jD_jA_j\\
 &=\int\sum_jc_j(x)\alpha_j(x)\,d\mu(x).
\end{aligned}
\]
The cell allocation $\alpha_j(x)=\mathbf1_{\{j(x)=j\}}$ is itself feasible,
so $\int c_{j(x)}(x)\,d\mu(x)=\mathsf T(\mu,A)$.
For the independent final source $X_n$, this integral is its expected
matched-pair cost conditional on the base. Since the sources are i.i.d.,
their joint law is unchanged by relabeling them (a property called \emph{exchangeability}).
Each source therefore has the same expected matched-pair cost, giving
\begin{equation}\label{eq:averaging-before-law}
 \E\min_{\sigma\in S_n}\sum_i c_{\sigma(i)}(X_i)
 =n\E[c_{j(X_n)}(X_n)]
 =n\E\mathsf T(\mu,A).
\end{equation}

It remains to identify the law of $A$. Surprisingly, under the
atomless-difference assumption, this law is always $\Dir(1,\ldots,1)$,
regardless of the cost functions and source distribution.
For two targets, it follows from
the uniform rank of a random observation. Put $H(x)=c_1(x)-c_2(x)$ and
let $X$ be the one base source.
Write $F_H$ for the distribution function of $H(X)$.
The first target is chosen when $H(x)\le H(X)$, so
$A_1=F_H(H(X))$ is uniform on $[0,1]$ and $A_2=1-A_1$.
Thus $A\sim\Dir(1,1)$.

For general $n$, take any integer $k\ge1$ and sample additional sources
$X'_1,\ldots,X'_k$ independently from $\mu$, independently of the base.
Each added source chooses its target using $j(\cdot)$, without competing
with other added sources (so several may choose the same target).
Let $K_j$ count the added sources assigned to target $j$.
Conditional on the base, each added source independently chooses target
$j$ with probability $A_j$, so
\[
 (K_1,\ldots,K_n)\mid(X_1,\ldots,X_{n-1})
 \sim\operatorname{Multinomial}(k;A_1,\ldots,A_n).
\]
On the other hand,
Lemma~\ref{lem:counts} in Appendix~\ref{app:averaging-proof} shows that,
among the $n-1+k$ sources, each
nonnegative count vector with total $k$ occurs for exactly one choice of
the base subset. By exchangeability, we may choose the base uniformly
among these subsets without changing the count distribution.
The bijection therefore makes all count vectors equally likely.
For $r_1+\cdots+r_n=k$,
\[
 \frac{k!}{\prod_jr_j!}\E\prod_jA_j^{r_j}
       =\binom{n+k-1}{n-1}^{-1},
 \qquad
 \E\prod_jA_j^{r_j}
       =\frac{(n-1)!\prod_jr_j!}{(n+k-1)!}.
\]
These are the mixed moments of $\Dir(1,\ldots,1)$.
Since probability laws on the compact simplex are determined by their
mixed moments, $A\sim\Dir(1,\ldots,1)$.

Substituting this law into~\eqref{eq:averaging-before-law}
proves~\eqref{eq:averaging} under the atomless-difference assumption.

To allow atoms and ties, augment each source independently by a vector
$Z$ uniform on $[0,1]^n$ and use costs
$c_j^\epsilon(x,z)=c_j(x)+\epsilon z_j$. Their pairwise differences are
atomless. The assignment cost increases by at most $n\epsilon$. Write
$\mathsf T_\epsilon(w)$ for the allocation cost with the augmented source distribution
and perturbed costs. Then $0\le\mathsf T_\epsilon(w)-\mathsf T(\mu,w)\le\epsilon$
for every $w$ (use an original allocation independent of $z$ for the upper
bound, and average an augmented allocation over $z$ for the lower bound).
Apply the identity
to the perturbed costs and let $\epsilon\downarrow0$.
\end{proof}

\section{Semidiscrete-to-smoothing reduction}\label{sec:smoothing-reduction}\label{sec:interpretation}
The averaging identity in Section~\ref{sec:averaging} puts bipartite
matching and semidiscrete transport on the same footing. We therefore
treat weighted and unweighted semidiscrete mean costs together, writing
$\nu_n=\sum_iw_i\delta_{Y_i}$ with either Dirichlet masses
$w\sim\Dir(1,\ldots,1)$ independent of the positions or equal masses
$w_i=1/n$. This section reduces the evaluation of either mean cost to the
evaluation of the smoothing cost $W_2^2(P_t\nu_n,\nu_n)$, where $P_t$
denotes heat smoothing at time $t$.

We first consider quadratic cost, a uniform source distribution
$\mu=\mu_\Omega$ on the square or torus, and independent target
positions $Y_i\sim\mu$. Write $P_t\nu_n=q_t\mu$,
so $q_t$ is the smoothed density relative to $\mu$. Let the
\emph{Poisson potential} $f_t$ solve the \emph{Poisson equation}
\[
 -\Delta f_t=q_t-1,\qquad \mu(f_t)=0,
\]
with Neumann or periodic boundary conditions, respectively. The
right-hand side $q_t-1$ is the difference between the smoothed and uniform
densities. The classical approach uses this potential to approximate
transport between the two continuous distributions~\cite{AST,AG}:
\[
 W_2^2(\mu,P_t\nu_n)
 \approx \int|\nabla f_t|^2\,d\mu.
\]
To reach the semidiscrete cost $W_2^2(\mu,\nu_n)$, this approach then
uses the triangle inequality. Together with Cauchy--Schwarz, it gives
\[
 n\E W_2^2(\mu,\nu_n)
 \le n\E W_2^2(\mu,P_t\nu_n)+n\E W_2^2(P_t\nu_n,\nu_n)
       +2n\sqrt{\E W_2^2(\mu,P_t\nu_n)\,
                 \E W_2^2(P_t\nu_n,\nu_n)}.
\]
The cross term restricts the choice of heat time. When
$n\E W_2^2(\mu,P_t\nu_n)\asymp\log n$, recovering the leading logarithmic
coefficient from this bound requires
$n\E W_2^2(P_t\nu_n,\nu_n)=o(\log n)$.
Obtaining the additive $o(1)$ accuracy needed for
Conjecture~\ref{conj:physics} requires the stronger bound
$n\E W_2^2(P_t\nu_n,\nu_n)=o(1/\log n)$.
This condition is possible only if $t=o(1/n)$
(Appendix~\ref{app:smoothing-small-time}), below the range where
the classical estimates justify the preceding Poisson equation
approximation~\cite[Remark~6.1]{AG}.

We instead use the same Poisson potential to evaluate the
\emph{source-change effect}, the difference between the costs of transport
from $\mu$ and from $P_t\nu_n$ to the common discrete target $\nu_n$:
\[
 W_2^2(\mu,\nu_n)-W_2^2(P_t\nu_n,\nu_n)
 \approx 2\nu_n(f_t)-\int|\nabla f_t|^2\,d\mu.
\]
This approximation has expected absolute error $o(1/n)$ throughout
the range $0<t\le1/2$, $nt/(\log n)^3\to\infty$.
This range allows large heat times such as $t=n^{-5/6}$ used for
subdivision in Section~\ref{sec:smoothing-evaluation}, at which
$n\E W_2^2(P_t\nu_n,\nu_n)\asymp\log n$ is of the \emph{same order}
as $n\E W_2^2(\mu,\nu_n)$.\footnote{At
this heat time, the bound from the triangle inequality has an inflated leading logarithmic
coefficient even if every term in that bound is evaluated exactly.}
The expectation of the right-hand side is explicitly computable from
the spectrum. Thus evaluating $n\E W_2^2(\mu,\nu_n)$ reduces to
evaluating $n\E W_2^2(P_t\nu_n,\nu_n)$ with the same additive $o(1)$
accuracy.

\subsection{Reduction for general costs and dimensions}\label{subsec:general-smoothing}
We now establish a semidiscrete-to-smoothing reduction in every dimension
for costs with a positive quadratic part. This quadratic part and the
sampling distribution determine the heat smoothing and Poisson equation.
Let $(M,g)$ be a compact, connected, smooth Riemannian manifold of dimension
$d$ without boundary, and let $\mu=\rho\,d\mathrm{vol}_g$ be a probability measure,
with $\rho>0$ smooth and $d\mathrm{vol}_g$ the Riemannian volume measure.
Suppose $k:M\times M\to[0,\infty)$ is continuous, uniformly Lipschitz in
its first variable, smooth near the diagonal, and satisfies
\begin{equation}\label{eq:smoothing-cost-assumptions}
 k(y,y)=0,\qquad k(x,y)>0\ (x\ne y),\qquad
 D_xk(y,y)=0,\qquad \tfrac12D_x^2k(y,y)=g_y.
\end{equation}
Here $g_y(u,v)$ is the inner product of tangent vectors $u,v$ at $y$;
for the Euclidean metric it is $u\cdot v$. The last condition gives
$k(x,y)=d_g(x,y)^2+O(d_g(x,y)^3)$ near the diagonal.
Squared geodesic distance satisfies these assumptions; symmetry of $k$
is not required.

Write $\OT_k(\eta,\zeta)=\inf_{\pi\in\Pi(\eta,\zeta)}\int k\,d\pi$ and
let $\nu_n=\sum_iw_i\delta_{Y_i}$, with independent $Y_i\sim\mu$
and equal or Dirichlet masses as above.
To preserve $\mu$, use heat smoothing $P_t=e^{-t\cL}$ with weighted
Laplacian $\cL=-\rho^{-1}\diver_g(\rho\nabla_g)$.
Write $P_t\nu_n=q_t\mu$.
The Poisson potential $f_t$ now solves
\[
 \cL f_t=q_t-1,\qquad \mu(f_t)=0.
\]
For uniform sampling, $\cL=-\Delta_g$, recovering the Poisson equation
used above. For nonuniform sampling, the factor $\rho$ in $\cL$ accounts
for the sampling density.

Let $\lambda_j>0$ be the positive eigenvalues of $\cL$ in $L^2(\mu)$, and put
\[
 F(t)=\sum_{j\ge1}\frac{e^{-t\lambda_j}}{\lambda_j},\qquad
 \chi_n=n\E\sum_iw_i^2=
 \begin{cases}1,&\text{equal masses},\\2n/(n+1),&\text{Dirichlet masses}.
 \end{cases}
\]

\begin{theorem}[Semidiscrete-to-smoothing reduction]\label{thm:smoothing-reduction}
Under the assumptions of Section~\ref{subsec:general-smoothing}, there is a constant $C>0$ such that,
for every $n\ge2$, if $0<t\le1/2$ and
$nt^{d/2}\ge C\log n$, then
\begin{equation}\label{eq:general-smoothing-reduction}
 \begin{split}
 n\E\bigl[\OT_k(\mu,\nu_n)-\OT_k(P_t\nu_n,\nu_n)\bigr]
 &=\chi_n\{2F(t)-F(2t)\}+e_{n,k}(t),\\
 |e_{n,k}(t)|&\le C R_d(n)\sqrt{\frac{\log n}{nt^{d/2}}}+Cn^{-10}.
 \end{split}
\end{equation}
Here $R_1(n)=1$, $R_2(n)=\log(en)$, and $R_d(n)=n^{1-2/d}$ for $d\ge3$.
The constant $C$ depends only on the fixed geometry, density, and cost.
In particular, $e_{n,k}(t)=o(1)$ whenever
$nt^{d/2}/[R_d(n)^2\log n]\to\infty$ and $t\le1/2$.
\end{theorem}

\begin{proof}
We first estimate the source-change effect for a fixed realization of $\nu_n$.
Put $E_{f_t}=\int|\nabla_g f_t|^2\,d\mu$.
Lemma~\ref{s:lem:split} compares corresponding transport plans from $\mu$
and from $q_t\mu$ to the same $\nu_n$. The leading change in cost,
$2\nu_n(f_t)-E_{f_t}$, is independent of the chosen plan, so taking
infima gives the following bound, with error controlled by the Hessian:
\[
 \bigl|\OT_k(\mu,\nu_n)-\OT_k(P_t\nu_n,\nu_n)
             -2\nu_n(f_t)+E_{f_t}\bigr|
 \le C\|\nabla_g^2f_t\|_\infty
             \{\OT_k(\mu,\nu_n)+E_{f_t}\}
\]
when the Hessian norm is sufficiently small.

We next compute the expected leading term. For an orthonormal mean-zero
eigenbasis $\psi_j$ of $\cL$,
\[
 f_t=\sum_{j\ge1}\frac{e^{-t\lambda_j}\nu_n(\psi_j)}{\lambda_j}\psi_j,
 \qquad
 n\E[\nu_n(\psi_j)\nu_n(\psi_\ell)]=\chi_n\mathbf1_{j=\ell}.
\]
Consequently,
$n\E\nu_n(f_t)=\chi_nF(t)$ and
$n\E E_{f_t}=\chi_nF(2t)$, which give the spectral term
$\chi_n\{2F(t)-F(2t)\}$.

To control the remainder, Appendix~\ref{app:general-heat} bounds
$\|\nabla_g^2f_t\|_\infty$ by $C\sqrt{\log n/(nt^{d/2})}$
outside an event whose contribution to $n$ times the expected absolute
error is at most $Cn^{-10}$. The condition on $t$ makes this Hessian
bound small enough for the source comparison. Combining it with
$n\E[\OT_k(\mu,\nu_n)+E_{f_t}]\le CR_d(n)$, also proved there,
gives~\eqref{eq:general-smoothing-reduction}.
\end{proof}

\subsection{Reduction for quadratic cost on surfaces}\label{subsec:surface-reduction}
We specialize Theorem~\ref{thm:smoothing-reduction} to quadratic cost
in two dimensions and extend it to surfaces with boundaries and
convex corners.
Let $(M,g)$ be a compact connected Riemannian surface, with no
boundary, a smooth boundary, or finitely many smooth boundary arcs
meeting at angles in $(0,\pi)$. These vertices are \emph{convex corners}.
Let $\mu=\rho\,dA$ be a probability measure, where $dA$ is area measure,
$A=\int_MdA$, and $g,\rho$ are smooth up to the boundary with $\rho>0$.
Write $d_M$ for distance along the surface. Use $\cL,P_t,\nu_n,\chi_n$
as in Section~\ref{subsec:general-smoothing}, with Neumann conditions
on any boundary, and write
$F_{M,\mu}(t)=\sum_{j\ge1}e^{-t\lambda_j}/\lambda_j$ for the spectral
sum of $\cL$.

\begin{proposition}[Semidiscrete-to-smoothing reduction on surfaces]\label{prop:surface-reduction}
Under the assumptions of Section~\ref{subsec:surface-reduction}, there is a constant $C>0$ such that,
for either equal or Dirichlet target masses, $n\ge2$, $0<t\le1/2$,
and $nt\ge C\log n$,
\begin{equation}\label{eq:surface-decomposition}
 \begin{aligned}
 n\E\bigl[W_{2,M}^2(\mu,\nu_n)-W_{2,M}^2(P_t\nu_n,\nu_n)\bigr]
 &=\chi_n\{2F_{M,\mu}(t)-F_{M,\mu}(2t)\}\\
 &\quad+O_{M,g,\rho}\!\left(\sqrt{\frac{(\log n)^3}{nt}}
                      +n^{-1/180}\log(en)\right).
 \end{aligned}
\end{equation}
The constant $C$ depends only on $M,g,\rho$. For uniform sampling on
the square or torus, the term $n^{-1/180}\log(en)$ can be omitted.
In particular, the error is $o(1)$ when $nt/(\log n)^3\to\infty$.
\end{proposition}

The proposition determines the (weighted) semidiscrete mean cost from
$n\E W_{2,M}^2(P_t\nu_n,\nu_n)$ and the explicitly computable spectral
term, with additive $o(1)$ error in the stated regime.
Its proof is given in Appendix~\ref{app:surface-reduction}.
For the mean-cost formula in Corollary~\ref{cor:surface-mean}, we
separate the spectral term into its logarithmic and constant parts:
\begin{equation}\label{eq:h-general}
 \begin{aligned}
 F_{M,\mu}(t)&=\frac A{4\pi}\log(1/t)+h_{M,\mu}+O(\sqrt t),\\
 2F_{M,\mu}(t)-F_{M,\mu}(2t)
 &=\frac A{4\pi}\log(2/t)+h_{M,\mu}+O(\sqrt t),
 \qquad t\downarrow0.
 \end{aligned}
\end{equation}
The first expansion defines the \emph{spectral finite part}
$h_{M,\mu}$, a constant determined by the geometry $(M,g)$ and sampling
distribution $\mu$. We write $h_M$ under uniform sampling.
Appendix~\ref{app:surface-trace} proves these expansions.

\subsection{The change-of-time formula}\label{subsec:poisson-smoothing}\label{subsec:flat-heat}
For Section~\ref{sec:smoothing-evaluation}, we need to compare expected
smoothing costs at different heat times to additive $o(1)$ accuracy.
Take quadratic cost and uniform $\mu=\mu_\Omega$ on $\Omega\in\{Q,\T\}$.
Subtracting~\eqref{eq:surface-decomposition} at two times and
using~\eqref{eq:h-general} with $A=1$ cancels both
$n\E W_{2,\Omega}^2(\mu,\nu_n)$ and $h_\Omega$.
For $0<t_1,t_2\le1/2$ with $nt_i\ge C\log n$ $(i=1,2)$, this gives the
\emph{change-of-time formula}:
\begin{equation}\label{f:eq:time-change}
 \begin{aligned}
 n\E\bigl[W_{2,\Omega}^2(P_{t_2}\nu_n,\nu_n)
           -W_{2,\Omega}^2(P_{t_1}\nu_n,\nu_n)\bigr]
 &=\frac{\chi_n}{4\pi}\log(t_2/t_1)\\
 &\quad+O\!\left(\sqrt{\frac{(\log n)^3}{nt_1}}
             +\sqrt{\frac{(\log n)^3}{nt_2}}
             +\sqrt{t_1}+\sqrt{t_2}\right).
 \end{aligned}
\end{equation}
Its error is $o(1)$ when $t_i\to0$ and $nt_i/(\log n)^3\to\infty$ $(i=1,2)$.

\section{Evaluating the smoothing cost by subdivision}\label{sec:smoothing-evaluation}\label{sec:flat}
Sections~\ref{sec:averaging} and~\ref{sec:smoothing-reduction} show that
the bipartite and semidiscrete mean costs can be determined to additive
$o(1)$ accuracy by evaluating the expected smoothing cost.
The change-of-time formula in Section~\ref{subsec:flat-heat} lets us
choose one convenient heat time for this evaluation. In this section,
we evaluate the smoothing cost by subdivision, comparing transport on the
whole domain with transport within smaller cells. We first determine its
logarithmic term and a limiting constant shared by the square and torus,
then extend the formula to surfaces and nonuniform densities.

We begin with uniform sampling on the unit square or unit flat torus,
$\mu=\mu_\Omega$, $\Omega\in\{Q,\T\}$, and equal or Dirichlet masses,
as in Section~\ref{subsec:flat-heat}. To compare sample sizes and heat times, write
\[
 d_n^\Omega(t)=n\E W_{2,\Omega}^2(P_t\nu_n,\nu_n),
 \qquad d_0^\Omega(t)=0.
\]

\begin{theorem}[The smoothing cost on the square and torus]\label{thm:smoothing-cost}
For uniform sampling on $\Omega\in\{Q,\T\}$ with Dirichlet target
masses, there is a universal constant $\kappa_\star\in\R$ such that
\[
 d_n^\Omega(t)=\frac{\log(nt)}{2\pi}+\kappa_\star+o(1)
\]
whenever $t=t(n)\to0$ and $nt/(\log n)^3\to\infty$.
More precisely,
\begin{equation}\label{g:eq:smoothing-asymptotic}
 \left|d_n^\Omega(t)-\frac{\log(nt)}{2\pi}-\kappa_\star\right|
 \le C\left\{n^{-1/120}\log(en)
       +\sqrt{\frac{(\log(en))^3}{nt}}+\sqrt t\right\}.
\end{equation}
The bound holds for $n\ge2$, $0<t\le1/2$, and $nt\ge C\log(en)$,
with a universal constant $C$.
For equal target masses, replace $1/(2\pi)$ by $1/(4\pi)$,
with a possibly different constant $\kappa_\star$.
\end{theorem}

For equal masses, Ambrosio and Glaudo~\cite[Theorem~5.2]{AG} proved the
upper bound $d_n^\Omega(t)=O(\log(nt))$ for $nt\ge C\log n$.
Theorem~\ref{thm:smoothing-cost} identifies the logarithmic coefficient
and proves convergence of the remainder to a constant shared by the
square and torus, for each mass law. This gives precise asymptotics for
heat smoothing in Wasserstein distance independently of matching.

\subsection{Subdivision with additive \texorpdfstring{$o(1)$}{o(1)} error}\label{subsec:subdivision}
Divide the domain into an $m\times m$ grid of square cells of side length
$1/m$. We compare the original problem with one in which both smoothing
and transport take place separately in each cell. Reflection at the cell
boundary keeps the smoothed mass inside, so the resulting density and the
sampled points in that cell have equal total mass.

For $m=2$, consider a quarter containing $K\sim\Bin(n,1/4)$ sampled
points (Figure~\ref{fig:subdivision}). Given $K=k>0$, normalize its mass
to one and double distances to obtain the original $k$-point model on
$Q$ at heat time $4t$, with expected smoothing cost $d_k^Q(4t)/k$.
Scaling back multiplies this cost by $1/4$ and by the quarter's mass.
That mass is independent of the normalized cell problem and has
conditional mean $k/n$, so the expected transport cost within the
quarter is $d_k^Q(4t)/(4n)$. Averaging over $K$, summing over the four
quarters, and multiplying by $n$ gives $\E d_K^Q(4t)$, to be compared
with $d_n^\Omega(t)$.
Appendix~\ref{app:flat-local} verifies the same calculation for general $m$.
Earlier subdivision arguments~\cite{GT,Domain} address expected matching
costs directly to obtain leading-order asymptotics. Here we apply
subdivision to $d_n^\Omega(t)$ at $t_n=n^{-5/6}$, obtaining additive $o(1)$ error.

\begin{lemma}[Subdivision estimate]\label{g:lem:grid}
Let $t_n=n^{-5/6}$ and $K\sim\Bin(n,m^{-2})$, where $m\in\{1,2,3\}$.
For either equal or Dirichlet target masses, every $n\ge2$,
and $\Omega\in\{Q,\T\}$,
\begin{equation}\label{g:eq:subdivision}
 \left|d_n^\Omega(t_n)-\E d_K^Q(m^2t_n)\right|
 \le Cn^{-1/120}\log n,
\end{equation}
\end{lemma}
The cases $m=2,3$ compare sample sizes; $m=1$ compares the square and torus.

\begin{figure}[!t]
 \centering
 \includegraphics[width=.98\textwidth]{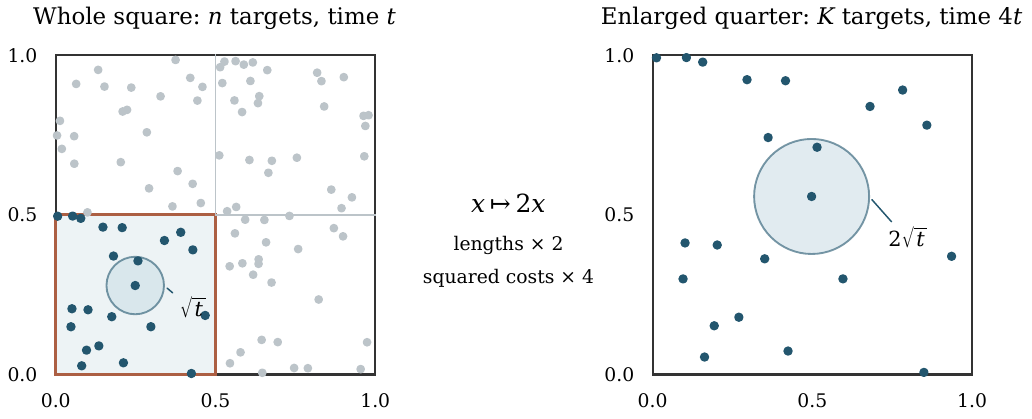}
 \caption{Rescaling a quarter to the unit square doubles distances and
 changes the heat time from $t$ to $4t$. Its sample size is
 $K\sim\Bin(n,1/4)$. Scaling back divides squared-distance costs by four.
 Lemma~\ref{g:lem:grid} compares $d_n^\Omega(t_n)$ with
 $\E d_K^Q(4t_n)$ to additive $o(1)$ accuracy.}
 \label{fig:subdivision}
\end{figure}

\paragraph{Proof sketch of Lemma~\ref{g:lem:grid}.}
The rescaling above identifies the expected cost of transport within cells. It remains to
control the two changes from the original problem. First, a coupling of
global heat smoothing and heat smoothing reflected within cells bounds
the change in expected transport cost by $o(1)/n$.
The main challenge is the second change: keeping the cellwise-smoothed
source distribution fixed, we must reassign transport that crosses cell
boundaries at an additional expected cost of $o(1)/n$.
We start from transport on the torus, whose cost is a lower bound for
both the square cost and the cost of transport confined to cells.
Write its optimal map as $x\mapsto x+\nabla\phi(x)$, with $\int\phi=0$.
The decisive estimate controls the optimal transport potential $\phi$:
\[
 n\E\|\phi\|_{L^2(\mu)}^2=o(1)\qquad(t=t_n).
\]
This estimate exploits the crucial cancellation: a local excess of mass
in $\nu_n$ is also present in its own smoothing on scales much larger than
$\sqrt t$. We modify the allocation within a narrow neighborhood of the
cell boundaries using a smooth cutoff.
Integration by parts bounds the resulting error through $\|\phi\|_2$.
The quantitative bound lets us shrink this neighborhood while keeping the
expected added cost $o(1)/n$. Appendix~\ref{app:flat-local} gives the full proof.

\subsection{Convergence of the constant term}
Subdivision changes both the sample size and the heat time. We use
the change-of-time formula to express every cell cost at the same
chosen time for its sample size, leaving a relation in the sample size
alone. Set
\[
 t_n=n^{-5/6},\qquad u_n=d_n^Q(t_n)\quad(n\ge1),\qquad u_0=0.
\]
Write $\chi=2$ for Dirichlet masses and $\chi=1$ for equal masses.
For $m=2,3$ and
$K\sim\Bin(n,m^{-2})$, the cell heat time $m^2t_n$ differs from $t_K$.
Combining the subdivision estimate~\eqref{g:eq:subdivision} with the
change-of-time formula~\eqref{f:eq:time-change} gives
\begin{equation}\label{g:eq:twosided}
 \left|u_n-\E u_K-\frac{\chi}{24\pi}\log m^2\right|
 \le Cn^{-1/120}\log n,\qquad m=2,3.
\end{equation}
The logarithmic correction comes from changing $m^2t_n$ to $t_K$:
at $K=n/m^2$, their ratio is $m^{1/3}$. Appendix~\ref{app:flat-counts}
checks the admissible times and controls the fluctuations of $K$.

To obtain a deterministic recursion for the averaged cost,
let $J\sim\Pois(N)$ and define $\bar u(N)=\E u_J$.
The cell count $K\mid J\sim\Bin(J,m^{-2})$ then has law
$K\sim\Pois(N/m^2)$ (Poisson thinning), so
$\E u_K=\bar u(N/m^2)$. Averaging~\eqref{g:eq:twosided} over $J$
therefore relates $\bar u(N)$ to $\bar u(N/m^2)$.
Subtracting the logarithmic term, with
$b(N)=\bar u(N)-(\chi/(24\pi))\log N$, gives
\begin{equation}\label{g:eq:two-scales}
 |b(4N)-b(N)|+|b(9N)-b(N)|
 \le CN^{-1/120}\log(eN).
\end{equation}
The rough bound $u_n\le C\log(en)$ makes $b$ continuous.
The fourfold comparison alone allows $b(N)=\sin(2\pi\log N/\log4)$.
The ninefold comparison eliminates such oscillations. Summability of
the errors then gives convergence.

\begin{lemma}[Two-scale convergence]\label{lem:two-scale}
Let $b:(0,\infty)\to\R$ be continuous. If, for $N\ge N_0\ge1$,
\[
 |b(4N)-b(N)|+|b(9N)-b(N)|
 \le CN^{-\alpha}\log(eN),\qquad\alpha>0,
\]
then some $c\in\R$ satisfies
$|b(N)-c|\le C'N^{-\alpha}\log(eN)$ for $N\ge N_0$.
\end{lemma}
\begin{proof}
Summing the fourfold errors gives a locally uniform limit
$L(N)=\lim_{k\to\infty}b(4^kN)$ and
\[
 |b(N)-L(N)|\le C\sum_{k\ge0}(4^kN)^{-\alpha}\log(e4^kN)
                \le C'N^{-\alpha}\log(eN).
\]
Thus $L$ is continuous and $L(4N)=L(N)$. Applying the ninefold comparison
at $4^kN$ and letting $k\to\infty$ gives $L(9N)=L(N)$.
The continuous function $s\mapsto L(e^s)$ therefore has periods
$\log4$ and $\log9$. Their ratio is irrational, since $4^p=9^q$ is
impossible for positive integers $p,q$. The subgroup they generate is
dense, so $L$ is constant.
\end{proof}

Applying the lemma to~\eqref{g:eq:two-scales} gives
\begin{equation}\label{g:eq:poissonrate}
 \bar u(N)=\frac{\chi}{24\pi}\log N+\kappa_\star
                    +O(N^{-1/120}\log(eN)).
\end{equation}

To return to fixed sample sizes, we use a monotonicity property:
for each fixed $t$, $d_n^Q(t)/n$ is nonincreasing in $n$.
This follows from joint convexity of squared transport cost by deleting
one sampled point and renormalizing the remaining masses. Standard Poisson
concentration bounds control deviations of the random sample size from
its mean. Together with the change-of-time formula, they transfer the
formula for $\bar u(N)$ to $u_n$.
Appendix~\ref{app:flat-counts} gives the details and proves
\begin{equation}\label{g:eq:canonical-smoothing}
 u_n=\frac{\chi}{24\pi}\log n+\kappa_\star
                    +O(n^{-1/120}\log(en)).
\end{equation}
The case $m=1$ of Lemma~\ref{g:lem:grid} gives the same formula for
$d_n^{\T}(t_n)$, with the same constant.

Finally, the change-of-time formula~\eqref{f:eq:time-change} gives,
for $0<t\le1/2$ with $nt\ge C\log(en)$,
\[
 d_n^\Omega(t)
 =u_n+\frac{\chi}{4\pi}\log(t/t_n)
       +O\!\left(n^{-1/120}\log(en)
          +\sqrt{\frac{(\log(en))^3}{nt}}+\sqrt t\right).
\]
Together with~\eqref{g:eq:canonical-smoothing}, this proves
Theorem~\ref{thm:smoothing-cost}.

\subsection{Surfaces and nonuniform densities}\label{subsec:surface-smoothing}
We extend the smoothing-cost evaluation to the surfaces and densities of
Section~\ref{subsec:surface-reduction}. Within a small cell, the metric
and density are nearly constant. Rescaling to the square therefore
identifies its smoothing cost; the main task is to keep the total error
$o(1)$ as the number of cells grows. Write
$d_n^{M,\mu}(t)=n\E W_{2,M}^2(P_t\nu_n,\nu_n)$.

\begin{theorem}[The smoothing cost on surfaces]\label{thm:surface-smoothing}
Let $(M,g)$ and $\mu=\rho\,dA$ satisfy the assumptions of
Section~\ref{subsec:surface-reduction}. Write $A=\int_MdA$ for the total
area. For Dirichlet target masses, $n\ge2$, $0<t\le1/2$, and
$nt\ge C\log(en)$,
\begin{equation}\label{eq:surface-smoothing}
 \begin{aligned}
 d_n^{M,\mu}(t)
 &=\frac A{2\pi}\log(nt)+A\kappa_\star
       +\frac1{2\pi}\int_M\log\rho\,dA\\
 &\quad+O_{M,g,\rho}\!\left(n^{-1/180}\log(en)
       +\sqrt{\frac{(\log(en))^3}{nt}}+\sqrt t\right),
 \end{aligned}
\end{equation}
where $C$ depends only on $M,g,\rho$ and $\kappa_\star$ is the constant
in Theorem~\ref{thm:smoothing-cost}. The error is $o(1)$ when
$t\to0$ and $nt/(\log n)^3\to\infty$.
For equal target masses, replace both occurrences of $1/(2\pi)$ by
$1/(4\pi)$ and use the equal-mass constant from that theorem.
\end{theorem}

The constant from the square is multiplied by the area; the density
contributes through the integral of $\log\rho$. In what follows, we
explain how both terms arise from subdivision at $t=t_n=n^{-5/6}$.
Subtracting Proposition~\ref{prop:surface-reduction} at $t_n$ and $t$
then extends the formula to the full heat-time range above.
Appendix~\ref{app:surface-proof} gives the details.

\paragraph{Confining transport to small cells.}
Choose $\ell=n^{-1/180}$ and cells $B$ of diameter $O(\ell)$, area
$A_B=\int_BdA\asymp\ell^2$, and probability $p_B=\mu(B)$.
The cells shrink but remain larger than the heat length and contain many
points: $\sqrt t\ll\ell$ and $np_B\to\infty$.
Let $d_k^B(t)$ be $k$ times the expected smoothing cost for $k$ points
sampled from $\mu(\cdot\mid B)$, with the same mass law and reflection
at the cell boundary; set $d_0^B(t)=0$.
For $K_B\sim\Bin(n,p_B)$, the same mass calculation as in
Section~\ref{subsec:subdivision} gives $\E d_{K_B}^B(t)$ as the
cell's expected contribution. The subdivision estimate is
\[
 d_n^{M,\mu}(t)=\sum_B\E d_{K_B}^B(t)+O(\ell\log(en)).
\]
Appendix~\ref{app:surface-confinement} controls the total error from
confining transport and replacing global smoothing by cellwise smoothing.

\paragraph{Comparing each cell with the square.}
Except for cells of total area $O(\ell)$, each cell admits coordinates
from $Q$ that send uniform probability exactly to $\mu(\cdot\mid B)$.
Smoothness makes squared distances measured within $B$ agree with
$A_B|x-y|^2$ in these coordinates up to relative error $O(\ell)$.
Comparing both distance and heat smoothing in these coordinates gives
(Appendix~\ref{app:surface-cell-comparison})
\begin{equation}\label{eq:surface-cell-comparison}
 d_k^B(t)=A_Bd_k^Q(t/A_B)+O(\ell A_B\log n),
 \qquad np_B/2\le k\le2np_B.
\end{equation}
The factor $A_B$ makes the summed error $O(\ell\log n)$, since
$\sum_BA_B=A$. The remaining cells satisfy the rough bound
$d_k^B(t)\le CA_B\log(en)$, so their total contribution is
$O(\ell\log n)$ as well.

\paragraph{Extracting the density correction.}
Applying Theorem~\ref{thm:smoothing-cost} to
$A_B\E d_{K_B}^Q(t/A_B)$ and using concentration of $K_B$ near $np_B$
gives the cell contribution
\[
 A_B\left\{\frac\chi{4\pi}\log\!\left(nt\frac{p_B}{A_B}\right)
                   +\kappa_\star\right\},
\]
up to errors summable at the stated rate. Here $\chi=2$ for Dirichlet
masses and $\chi=1$ for equal masses, and $\kappa_\star$ is the
corresponding constant from Theorem~\ref{thm:smoothing-cost}.
Since $\sum_BA_B=A$ and $p_B/A_B$ is the average density in $B$,
\[
 \sum_BA_B\log(p_B/A_B)=\int_M\log\rho\,dA+O(\ell).
\]
This proves the formula at $t_n$ with error $O(n^{-1/180}\log(en))$;
the change-of-time comparison completes Theorem~\ref{thm:surface-smoothing}.
The factor $A_B$ explains why the density correction is integrated against area.
\section{Mean-cost formulas and distinct distributions}\label{sec:source-target}
We combine the reductions and smoothing-cost estimates of
Sections~\ref{sec:averaging}--\ref{sec:smoothing-evaluation} to derive the
two-dimensional mean-cost formulas for quadratic cost, first in the
same-distribution case (Section~\ref{subsec:same-distribution}). A further
reduction extends these formulas to distinct distributions. We formulate
it for general costs (Section~\ref{subsec:distinct-to-same}), specialize it
to prove Theorem~\ref{thm:source-target} for quadratic cost
(Section~\ref{subsec:torus}), and then show that distinct distributions
can change the leading order for unsquared distance cost
(Section~\ref{subsec:linear}).

\subsection{Mean-cost formulas for the same-distribution case}\label{subsec:same-distribution}
For two independent $n$-point samples $X_i,Y_j$ from a common law $\mu$
on $M$, write
\[
 C_n^M(\mu)=\E\min_{\sigma\in S_n}\sum_i d_M(X_i,Y_{\sigma(i)})^2,
 \qquad
 S_n^M(\mu)=n\E W_{2,M}^2\!\left(\mu,\frac1n\sum_j\delta_{Y_j}\right).
\]
The spectral finite part $h_{M,\mu}$ is defined in~\eqref{eq:h-general}.
Under uniform sampling we write $h_M,C_n^M,S_n^M$.

\begin{corollary}[Same-distribution mean-cost formulas]\label{cor:surface-mean}\label{cor:flat-mean}
Let $(M,g)$ be a compact connected 2-dimensional Riemannian surface
with smooth boundary, convex corners, or no boundary. Let $\mu=\rho\,dA$
be a probability measure with strictly positive density, with $g$ and
$\rho$ smooth up to the boundary. Write $A=\int_MdA$.
There are universal constants $\beta,\beta_{\mathrm{sd}}\in\R$, independent of
$(M,g)$ and $\rho$, such that
\begin{align*}
 C_n^M(\mu)={}&\frac A{2\pi}\log n+A\beta
       +2h_{M,\mu}+\frac1{2\pi}\int_M\log\rho\,dA
       +O_{M,g,\rho}(n^{-1/180}\log(en)),\\
 S_n^M(\mu)={}&\frac A{4\pi}\log n+A\beta_{\mathrm{sd}}
       +h_{M,\mu}+\frac1{4\pi}\int_M\log\rho\,dA
       +O_{M,g,\rho}(n^{-1/180}\log(en)).
\end{align*}
For uniform sampling on $Q$ or $\T$, $A=1$ and the density integral
vanishes; both remainders improve to $O(n^{-1/120}\log(en))$.
\end{corollary}
\begin{proof}
Substitute the formula for the smoothing term $d_n^{M,\mu}(t_n)$ from
Theorem~\ref{thm:surface-smoothing} into
Proposition~\ref{prop:surface-reduction}, with $t_n=n^{-5/6}$.
Using the spectral expansion~\eqref{eq:h-general}, the heat-time terms
cancel because
\[
 \log(2/t_n)+\log(nt_n)=\log n+\log2.
\]
Adding $(\log2)/(2\pi)$ and $(\log2)/(4\pi)$ to the corresponding
smoothing constants defines $\beta$ and $\beta_{\mathrm{sd}}$.
The averaging identity (Theorem~\ref{thm:averaging}) identifies the
Dirichlet case with $C_n^M(\mu)$; equal masses give $S_n^M(\mu)$.
Replacing the Dirichlet factor $2n/(n+1)$ by $2$ contributes only
$O(\log(en)/n)$ to the error. For $Q$ and $\T$, using
Theorem~\ref{thm:smoothing-cost} in the same calculation gives the
sharper remainder.
\end{proof}

Density enters globally through the weighted spectrum in $h_{M,\mu}$ and
locally through the integral of $\log\rho$. For uniform sampling on
unit-area surfaces, this integral vanishes, proving both assertions of
Conjecture~\ref{conj:physics}. Comparing two such fixed geometries gives
\begin{align*}
 C_n^M-C_n^N&=2(h_M-h_N)+O(n^{-1/180}\log(en)),\\
 S_n^M-S_n^N&=h_M-h_N+O(n^{-1/180}\log(en)).
\end{align*}
This proves the spectral comparison predicted in~\cite{BCCDS}, including
comparisons between domains with boundary and closed surfaces.
For the square and torus, Appendix~\ref{app:flat-spectrum} computes
$h_Q-h_{\T}=1/6+(\log2)/(2\pi)$, giving
$C_n^Q-C_n^{\T}=1/3+(\log2)/\pi+o(1)$.

\subsection{Distinct-to-same-distribution reduction}\label{subsec:distinct-to-same}
We now turn to the distinct-distribution case. We show that the excess
matching cost equals the expected cost of a same-distribution problem
with a new cost. This new cost need not be squared distance even when
the original cost is, so we state the reduction for general costs.
In two dimensions, Theorem~\ref{thm:distinct-to-same} replaces the new
cost by squared-distance cost in an induced metric, with an $o(1)$
error in expected total cost.
Corollary~\ref{cor:surface-mean} then gives the mean-cost formula.

\paragraph{An exact same-distribution representation.}
Let $\mu$ and $\nu$ be probability measures on compact spaces $M$ and $N$,
and let the cost function $c:M\times N\to\R$ be continuous.
Suppose a population optimal transport map $T$ is invertible and has continuous dual potentials
$a,b$ satisfying
\begin{equation}\label{eq:population-dual}
 a(x)+b(y)\le c(x,y),\qquad
 a(x)+b(Tx)=c(x,Tx),\qquad T_\#\mu=\nu.
\end{equation}
Here $T_\#\mu$ denotes the distribution of $T(X)$ when $X\sim\mu$.
To express the remaining cost using two points in the source space,
write a target as $Tz$ and define the nonnegative \emph{population dual gap}
\begin{equation}\label{eq:population-gap}
 k(x,z)=c(x,Tz)-a(x)-b(Tz).
\end{equation}
Thus $k(z,z)=0$.

For two independent samples $X_1,\ldots,X_n$ from $\mu$ and
$Y_1,\ldots,Y_n$ from $\nu$, set
$Z_j=T^{-1}Y_j$. The $Z_j$ form an independent sample from $\mu$, and
the following identity holds for every realization, in any dimension:
\begin{equation}\label{eq:population-decomposition}
 \min_\sigma\sum_i c(X_i,Y_{\sigma(i)})
 =\sum_i a(X_i)+\sum_j b(Y_j)
     +\min_\sigma\sum_i k(X_i,Z_{\sigma(i)}).
\end{equation}
The first two sums do not depend on the matching and have total
expectation $n\OT_c(\mu,\nu)$, where
$\OT_c(\mu,\nu)=\int c(x,Tx)\,d\mu(x)$ is the population cost.
Thus the excess mean cost equals the expected matching cost for $k$
between two samples from $\mu$.\footnote{Whenever $k$ satisfies the
assumptions of Theorem~\ref{thm:smoothing-reduction}, the
bipartite-to-semidiscrete and semidiscrete-to-smoothing reductions lead
to its smoothing cost, with the quantitative error stated in that theorem.}

\paragraph{Identifying the metric.}
In two dimensions, we replace $k$ by squared-distance cost in the metric
determined by its quadratic part, then apply Corollary~\ref{cor:surface-mean}.
When $k$ is smooth near the diagonal, define
\[
 g_T(z)=\tfrac12D_x^2k(z,z)
\]
provided this quadratic form is positive definite.\footnote{In the
terminology of Wong and Yang~\cite{WongYang},
$k$ is the $c$-divergence in source coordinates. Their associated metric
is $2g_T$; our $g_T$ is their metric for $k/2$.}
Since $k(\cdot,z)$ has its minimum at $z$, $D_xk(z,z)=0$, so its
Hessian defines a quadratic form independently of coordinates.
With this normalization,
$k(x,z)=d_{g_T}(x,z)^2+O(d_{g_T}(x,z)^3)$ near the diagonal.

Write $C_n^{(M,g_T)}(\mu)$ and $S_n^{(M,g_T)}(\mu)$ for the
bipartite and semidiscrete mean costs defined above, now
using the induced metric $g_T$.

\begin{theorem}[Distinct-to-same-distribution reduction]\label{thm:distinct-to-same}\label{thm:regular-source-target}
Let $M,N$ be compact connected 2-dimensional surfaces, let $\mu,\nu$ be
probability measures on them, and let $c:M\times N\to\R$ be continuous.
Suppose a smooth diffeomorphism $T:M\to N$, smooth up to the boundary,
and continuous dual potentials $a,b$ satisfy \eqref{eq:population-dual}.
Assume the gap $k$ in \eqref{eq:population-gap} is smooth near the
diagonal, vanishes only on the diagonal, and has positive-definite
quadratic part $g_T$. Suppose $(M,g_T)$ has smooth boundary, convex
corners, or no boundary, and $\mu$ has a strictly positive density
relative to induced area, with the metric and density smooth up to
the boundary.
Then two independent $n$-point samples $X_i\sim \mu$ and $Y_j\sim \nu$
satisfy
\[
 \E\min_\sigma\sum_i c(X_i,Y_{\sigma(i)})-n\OT_c(\mu,\nu)
 =C_n^{(M,g_T)}(\mu)+O(n^{-1/4}\log^{5/4}(en)).
\]
For $\nu_n=n^{-1}\sum_j\delta_{Y_j}$,
\[
 n\{\E\OT_c(\mu,\nu_n)-\OT_c(\mu,\nu)\}
 =S_n^{(M,g_T)}(\mu)+O(n^{-1/4}\log^{5/4}(en)).
\]
The constants hidden in the error terms depend only on the fixed data
$M,g_T,\mu,k$.
\end{theorem}

\begin{proof}
The dual terms in~\eqref{eq:population-decomposition} satisfy
$\E[\sum_i a(X_i)+\sum_j b(Y_j)]=n\OT_c(\mu,\nu)$.
Thus the excess matching cost equals the expected matching cost for
$k$ between two samples from $\mu$.
Lemma~\ref{lem:cost-comparison} in Appendix~\ref{app:cost-comparison}
bounds the change in this mean cost when $k$ is replaced by $d_{g_T}^2$
by $O(n^{-1/4}\log^{5/4}(en))$, proving the formula for bipartite matching.
For semidiscrete transport, the corresponding identity after taking
expectations is
\[
 n\{\E\OT_c(\mu,\nu_n)-\OT_c(\mu,\nu)\}
 =n\E\OT_k\!\left(\mu,\frac1n\sum_j\delta_{T^{-1}Y_j}\right).
\]
The points $T^{-1}Y_j$ are independent with law $\mu$, so the equal-mass
part of the same lemma gives the formula for semidiscrete transport.
\end{proof}

Theorem~\ref{thm:distinct-to-same} lets us apply the same-distribution
formula of Corollary~\ref{cor:surface-mean} in the induced metric.
Write $dA_T$ for the area measure of $g_T$,
$A_T=\int_MdA_T$, and $\mu=\rho_T\,dA_T$. Let $h_{g_T,\mu}$ be the
spectral finite part in~\eqref{eq:h-general} for this metric and density.

\begin{corollary}[Mean-cost formulas for general costs]\label{cor:general-cost-mean}
Under the assumptions of Theorem~\ref{thm:distinct-to-same},
\[
 \begin{split}
 \E\min_\sigma\sum_i c(X_i,Y_{\sigma(i)})-n\OT_c(\mu,\nu)
       ={}&\frac{A_T}{2\pi}\log n+A_T\beta
       +2h_{g_T,\mu}+\frac1{2\pi}\int_M\log\rho_T\,dA_T\\
       &+O(n^{-1/180}\log(en)).
 \end{split}
\]
For transport from $\mu$ to $\nu_n=n^{-1}\sum_j\delta_{Y_j}$,
$n\{\E\OT_c(\mu,\nu_n)-\OT_c(\mu,\nu)\}$ satisfies the same formula with
$\beta$ replaced by $\beta_{\mathrm{sd}}$ and with the
logarithmic term, spectral correction, and density correction divided by two.
The constants $\beta,\beta_{\mathrm{sd}}$ are those of
Corollary~\ref{cor:surface-mean}; the error constant depends only on
$M,g_T,\mu,k$.
\end{corollary}
\begin{proof}
Apply Corollary~\ref{cor:surface-mean} on $(M,g_T)$; its error dominates
the reduction error $O(n^{-1/4}\log^{5/4}(en))$.
\end{proof}

\begin{example}[Power costs for translated distributions]\label{ex:power-translation}
Let $\mu$ have a smooth positive density on a compact convex domain
$\Omega\subset\R^2$ with smooth boundary. Fix a nonzero vector $\xi$
and let $\nu$ be the law of $X+\xi$ when $X\sim\mu$.
For two independent $n$-point samples $X_i\sim\mu$ and $Y_j\sim\nu$,
we study the minimum total cost
\[
 M_n^{(r)}(\mu,\nu)=\min_\sigma\sum_i|X_i-Y_{\sigma(i)}|^r,
 \qquad r>1.
\]
Every population coupling has mean displacement $\xi$, so Jensen's
inequality shows that $T(x)=x+\xi$ is optimal, with population cost
$|\xi|^r$. Writing $h(u)=|u|^r$, the population dual gap and induced metric are
\[
 k(x,z)=h(\xi+z-x)-h(\xi)-Dh(\xi)\cdot(z-x),\qquad
 g_T=\tfrac12D^2h(\xi).
\]
Since $\xi\ne0$, the gap is smooth near $x=z$ and $g_T$ is positive
definite; strict convexity of $h$ gives $k(x,z)>0$ for $x\ne z$.
The induced area is
$A_T=\tfrac r2\sqrt{r-1}\,|\xi|^{r-2}|\Omega|$, so
Corollary~\ref{cor:general-cost-mean} determines
$\E M_n^{(r)}(\mu,\nu)-n|\xi|^r$ up to $o(1)$.
When $\xi=0$, the distributions coincide, and
Corollary~\ref{cor:general-cost-mean} does not apply for $r\ne2$ because
$k(x,z)=|x-z|^r$ lacks the required smooth positive-definite quadratic part.
\end{example}

\subsection{Specialization to quadratic cost}\label{subsec:torus}
We obtain Theorem~\ref{thm:source-target} by specializing
Corollary~\ref{cor:general-cost-mean} to squared Euclidean distance on
convex domains and squared geodesic distance on flat tori.
Write $\mu=p\,dx$, $\nu=q\,dx$, and $T=\nabla\Phi$, with convex
potential $\Phi$. On a convex domain, and near $x=z$ on a torus after
choosing nearby lifts, the population dual gap satisfies
\[
 k(x,z)=2\{\Phi(x)-\Phi(z)-\nabla\Phi(z)\cdot(x-z)\}.
\]
Thus the induced metric is $g_T=H=D^2\Phi=DT$.
On a convex domain, we assume that $T$ is a smooth diffeomorphism up
to the boundary and $\lambda I\preceq H\preceq\Lambda I$, as in
Theorem~\ref{thm:source-target}. For quadratic cost on a flat torus
$\Omega=\R^2/\Gamma$, smooth positive densities ensure this
regularity~\cite{CE};
here $\Phi(x)=|x|^2/2+u(x)$ with $u$ periodic.
Appendix~\ref{app:flat-source-target} verifies the global quadratic
bounds, the boundary assumptions, and the comparison with $d_H^2$.

\paragraph{Area, density, and spectrum.}
It remains to express the quantities in the mean-cost formula using $H$.
In both settings, change of variables gives
$q(Tx)\det H(x)=p(x)$. The area measure, area, and density are
\[
 dA_H=\sqrt{\det H}\,dx,\qquad A_H=\int_{\Omega}dA_H,\qquad
 \rho_H=\frac{p}{\sqrt{\det H}}.
\]
Thus $A_H$ measures area using the quadratic part of the population dual gap,
while $\rho_H$ measures sampling probability per unit of that area.
When $p=q$, the population optimal transport map is the identity, so $H=I$, $A_H=|\Omega|$, and
$\rho_H=p$.

The weighted Laplacian from Section~\ref{subsec:surface-reduction} becomes
\[
 \cL_Hf=-\rho_H^{-1}\diver_H(\rho_H\nabla_Hf)
         =-p^{-1}\diver(pH^{-1}\nabla f).
\]
On a domain, it has the Neumann condition
$(pH^{-1}\nabla f)\cdot n=0$, where $n$ is the Euclidean outward
normal; on a torus, it has periodic boundary conditions.
Write $h_H=h_{(\Omega,H),\mu}$ for its spectral finite part, defined
by \eqref{eq:h-general}.
Applying Corollary~\ref{cor:general-cost-mean} with these inputs proves
Theorem~\ref{thm:source-target}. Although expressed in source coordinates,
the formula is unchanged when the two distributions are reversed
(Appendix~\ref{app:flat-source-target}).

The area formula also proves Corollary~\ref{cor:effective-area}:
distinct distributions strictly reduce the logarithmic coefficient,
since $A_H<|\Omega|$ unless $\mu=\nu$.

\begin{proof}[Proof of Corollary~\ref{cor:effective-area}]
Since $T$ maps $\Omega$ onto itself, $\int_\Omega\det H\,dx=|\Omega|$.
Cauchy--Schwarz gives $A_H\le|\Omega|$, and equality forces $\det H\equiv1$.
In that case $T$ also transports uniform probability to itself. Its
quadratic dual potentials certify optimality for these uniform
marginals as well, so uniqueness and the zero-cost identity coupling
force $T=\operatorname{Id}$. Hence $\mu=\nu$. The converse is immediate.
\end{proof}

\begin{example}[A varying logarithmic coefficient]\label{ex:source-target-cosine}
To quantify this decrease, let $\mu$ be uniform on $Q$ or $\T$ and
define $\nu=T_\#\mu$ by the map below, with $|\varepsilon|<1$:
\[
 T(x)=\left(x_1+\frac{\varepsilon}{2\pi}\sin(2\pi x_1),x_2\right),
 \qquad
 H(x)=\diag\bigl(1+\varepsilon\cos(2\pi x_1),1\bigr).
\]
The map $T$ is optimal for quadratic cost. The target density $q$ satisfies
$q(Tx)=[1+\varepsilon\cos(2\pi x_1)]^{-1}$, and
\[
 W_2^2(\mu,\nu)=\frac{\varepsilon^2}{8\pi^2},\qquad
 A_H=\int_0^1\sqrt{1+\varepsilon\cos(2\pi s)}\,ds
     =1-\frac{\varepsilon^2}{16}+O(\varepsilon^4).
\]
Thus a small change in the target distribution decreases the logarithmic
coefficient, while the universal constant $\beta$ remains the same.
\end{example}

\subsection{Unsquared distance cost: a change in leading order}\label{subsec:linear}
Theorem~\ref{thm:distinct-to-same} permits general costs but requires
the population dual gap $k$ to have a positive-definite quadratic part.
For unsquared distance cost $c(x,y)=|x-y|$, this condition can fail:
in the example below, population mass moves to the right and $k$
vanishes along that direction. The issue goes beyond the scope of our
reduction: distinct distributions can change the leading order of the
excess matching cost, rather than only its coefficient.

We construct two distinct smooth positive densities on the unit square for which
the excess matching cost is $O(\log n)$, much smaller
than the classical $\sqrt{n\log n}$ order for two independent uniform
samples~\cite{AKT}.\footnote{For
general distinct densities in two dimensions, Treleaven, Pavone, and
Frazzoli~\cite[Lemma~5.9]{TPF} give an $O(\sqrt{n\log n})$ upper bound
on the excess matching cost with high probability.}

\begin{proposition}[Unsquared distance cost for ordered product distributions]\label{prop:linear-product}
Let $\mu=p_1(x_1)\,dx_1dx_2$ and $\nu=q_1(x_1)\,dx_1dx_2$ on $[0,1]^2$,
where $0<\lambda\le p_1,q_1\le\Lambda<\infty$. Write $F,G$ for their
first-coordinate distribution functions and suppose that, for some
$\gamma>0$,
\begin{equation}\label{eq:ordered-marginals}
 F(x)-G(x)\ge\gamma x(1-x)\qquad(0\le x\le1).
\end{equation}
For two independent $n$-point samples $X_i\sim \mu$ and $Y_j\sim \nu$, set
$M_n^{(1)}=\min_\sigma\sum_i|X_i-Y_{\sigma(i)}|$. Then
\[
 0\le\E M_n^{(1)}-nW_1(\mu,\nu)\le C\log(en),\qquad
 W_1(\mu,\nu)=\int_0^1(F-G)(x)\,dx,
\]
where $C$ depends only on $\lambda,\Lambda,\gamma$.
\end{proposition}

For example, $p_1=1$ and $q_1(x)=1+a(2x-1)$, with fixed $0<a<1$,
give smooth positive densities on the same square and $W_1(\mu,\nu)=a/6$.
The bound keeps $a$ fixed; its constant can grow as $a$ decreases to zero.
The proof is given in Appendix~\ref{app:linear-proof}.

This contrast suggests a wider range of leading orders as the sampling
distributions vary. Determining which orders occur under unsquared distance
cost and how they depend on the two distributions is an open problem
that we leave for future work.

\section*{Acknowledgments}
The author thanks Yunbei Xu for discussions and comments on the feasibility
and value of several ideas during the summer of 2026, and for encouragement
to continue pursuing this work.
\section*{Statement of AI Use}
The author formulated the research questions, developed the main conceptual
approach, and wrote the main text. ChatGPT was used under the author's
detailed direction to explore extensions, develop and refine proof arguments, and
provide editorial assistance. The appendix proofs were completed with
substantial assistance from ChatGPT and were reorganized, revised, and checked by the
author. The author takes responsibility for all claims, proofs, and
exposition.

\appendix
\section{The averaging identity}
The bijection below completes the proof of the averaging identity in
Section~\ref{sec:averaging}. We then apply that identity to the
Barthe--Bordenave conjecture for quadratic cost.

\subsection{The bijection behind the Dirichlet law}\label{app:averaging-proof}
Fix $n\ge2$, an integer $k\ge0$, and $N=n-1+k$ sources. Let $c_{ij}$
be the cost of assigning source $i$ to target $j$. Assume the cost matrix
is \emph{generic}: no alternating sum of costs around a simple cycle in
the bipartite source--target graph is zero. For independent sources,
this holds almost surely under the atomless-difference assumption of
Section~\ref{sec:averaging}: condition on all but one source in each cycle.

For a base subset $S$ of $n-1$ sources, let $D_j(S)$ be the minimum cost
of matching $S$ to all targets except $j$. Each source $i\notin S$
chooses a target minimizing $c_{ij}+D_j(S)$, as in~\eqref{eq:insertion}.
Let $K_j(S)$ count these sources choosing target $j$. The choices are
made separately, so several sources may choose the same target.

\begin{lemma}[Every count vector occurs once]\label{lem:counts}
For a generic cost matrix $c_{ij}$, the map
$S\mapsto(K_1(S),\ldots,K_n(S))$ is a bijection from the
$(n-1)$-element base subsets to the nonnegative integer vectors with sum
$k$. All these target choices are unique.
\end{lemma}
\begin{proof}
Fix a count vector $r=(r_1,\ldots,r_n)$ with sum $k$. We construct its
base as the sources whose mass is split in an auxiliary transport
problem. Give each source mass one and target $j$ mass $r_j+a$, where
$a=(n-1)/n$. The total target mass is $k+na=N$.
Choose an optimal plan and dual variables $u_i,q_j$, normalized by
$q_n=0$, so that
$c_{ij}+q_j\ge u_i$, with equality on every edge carrying positive mass.

The \emph{support graph} consists of edges carrying positive mass; the
\emph{equality graph} consists of edges satisfying $c_{ij}+q_j=u_i$.
A cycle in the equality graph would contradict genericity. The support
graph is connected: any component with $b$ targets has integer source
mass and target mass equal to an integer plus $ba$, so $n$ divides $b$
and hence $b=n$. Thus the support is a spanning tree. The equality graph
contains this tree and has no cycles, so the two graphs coincide and
have $N+n-1$ edges.

Let $s$ count the sources split between targets, and let $l_j$ count
the unsplit sources sent to $j$. Since each source uses at least one
edge and each split source at least two, $s\le n-1$.
After the unsplit sources are removed, target $j$ still demands $r_j-l_j+a$,
a nonnegative number with fractional part $a$, hence at least $a$.
The split sources supply these demands, so $s\ge na=n-1$. Equality
forces exactly $n-1$ sources to split, each between two targets, and
$l_j=r_j$ for every $j$.

Let $S$ be the split sources. To identify their costs $D_j(S)$, remove
the unsplit sources and replace each split source by an edge joining
its two targets. Root the resulting tree at target $j$ and assign each
edge's source to the endpoint farther from the root. This matches $S$
to every target except $j$ exactly once. All chosen edges attain dual
equality, so this matching attains the dual lower bound
\[
 D_j(S)=\sum_{i\in S}u_i-\sum_{l\ne j}q_l.
\]
Hence $q_j=D_j(S)-D_n(S)$ and
$c_{ij}+D_j(S)=(c_{ij}+q_j)+D_n(S)$.
Each unsplit source has a unique equality edge, so its target also
uniquely minimizes $c_{ij}+D_j(S)$. Consequently, $K_j(S)=l_j=r_j$.

We have constructed a base for every count vector. Different count
vectors give different bases, and both sets have
$\binom{n-1+k}{n-1}$ elements. Thus every base occurs exactly once,
with unique choices for all sources outside it.
\end{proof}

\subsection{Application to the Barthe--Bordenave conjecture}\label{app:highdim}
The bounds for general densities in Barthe and Bordenave~\cite{BB}
involve matching constants with and without a boundary relaxation.
They conjectured that these constants coincide. For quadratic cost,
our averaging identity (Theorem~\ref{thm:averaging}) transfers this
question to the semidiscrete boundary comparison of Huesmann, Goldman,
and Trevisan~\cite{GHT}.
\Needspace{10\baselineskip}
\begin{corollary}[Quadratic matching for general densities in higher dimensions]\label{cor:highdim}
For every $d\ge5$ there is $\beta_d\in(0,\infty)$ such that, for any compactly supported probability density $f$ on $\R^d$ and two independent $n$-point samples with that density,
\[
 n^{2/d-1}\min_{\sigma\in S_n}\sum_i|X_i-Y_{\sigma(i)}|^2
 \longrightarrow\beta_d\int_{\R^d}f^{1-2/d}\,dx
\]
almost surely and in $L^q$ for every finite $q\ge1$.
\end{corollary}
\begin{proof}
We first work on expanding cubes $Q_R=(0,R)^d$ of side length $R$
and volume $V=R^d$, with $R\to\infty$. Draw a common Poisson sample
size $N\sim\Pois(V)$ and, conditional on $N=n$, two independent
uniform $n$-point samples $X,Y$ in $Q_R$.
The averaging identity relates their expected matching costs, with and
without the boundary relaxation, to semidiscrete costs. We prove that
the two limiting costs agree, then identify their common value with
the constants in Barthe and Bordenave's fixed-size bounds.

\paragraph{Applying the averaging identity.}
Write $r_R(x)=\dist(x,\partial Q_R)$.
For finite point multisets $X,Y\subset Q_R$, let $M_R(X,Y)$ be the
minimum quadratic cost of a matching of cardinality
$\min\{|X|,|Y|\}$. Let $B_R(X,Y)$ be the minimum cost of a partial
matching, charging $r_R(x)^2$ for every unmatched point.
To apply the averaging identity to $B_R$, express it as an assignment
cost: when $|X|=|Y|$, it equals the complete assignment cost for
\[
 b_R(x,y)^2=\min\{|x-y|^2,r_R(x)^2+r_R(y)^2\}.
\]
Indeed, remove every pair using the boundary alternative from a complete
assignment; conversely, complete a partial matching by pairing the
unmatched points. Neither operation increases the corresponding cost.
This is the symmetric boundary transport cost used in~\cite{GHT}; the
coefficient in Barthe--Bordenave's boundary functional equals one for
quadratic cost.

Realize the target sample $Y$ as the points of a unit-intensity
Poisson process on $\R^d$ that fall in $Q_R$. To obtain the Dirichlet
masses required by Theorem~\ref{thm:averaging}, give each point an
independent exponential mass of mean one, and write $\nu$ for this
random measure. Conditional on $N=n>0$,
$\nu|_{Q_R}=\sum_{i=1}^nE_i\delta_{Y_i}$. Put $S=\sum_iE_i$.
The proportions $W_i=E_i/S$ are Dirichlet independently of $S$ and
the positions, and $\E[S\mid N=n]=n$.
For finite measures of equal mass, use the same transport objective
$\OT_c(\eta,\zeta)=\inf_{\pi\in\Pi(\eta,\zeta)}\int c\,d\pi$,
without normalizing their mass.
For either symmetric cost $c=|x-y|^2$ or $b_R^2$, homogeneity and the
averaging identity give
\[
 \begin{aligned}
 \E\!\left[\OT_c\left(\nu|_{Q_R},\frac SVdx\right)\,\middle|\,N=n\right]
 &=n\E\!\left[\OT_c\left(\frac{dx}V,\sum_iW_i\delta_{Y_i}\right)\,\middle|\,N=n\right]\\
 &=\E\!\left[\min_{\sigma\in S_n}\sum_i c(X_i,Y_{\sigma(i)})\,\middle|\,N=n\right],
 \end{aligned}
\]
Here Lebesgue measure is restricted to $Q_R$, and all costs are zero
when $n=0$. Averaging over $N\sim\Pois(V)$ gives
the same equality for the matching model introduced above.

\paragraph{Equality of the semidiscrete limits.}
We check the concentration assumption of~\cite[Assumption~4.1]{GHT}
for $\nu$. For a bounded nonnegative function $h$ supported in a cube
$Q$, put $Z=\int h\,d\nu-\int h\,dx$ and $a=\|h\|_\infty$.
For $|\theta|a\le1/2$, the Poisson exponential formula gives
\[
 \log\E e^{\theta Z}
 =\int_Q\frac{\theta^2h(x)^2}{1-\theta h(x)}\,dx
 \le2\theta^2a^2|Q|.
\]
Thus
\[
 \PP(|Z|>u)\le2\exp\!\left[-c\min\left\{\frac{u^2}{a^2|Q|},\frac ua\right\}\right].
\]
The zero function is immediate. Integration of the tail gives
$\|Z\|_q\le C_qa(|Q|^{1/2}+1)\le C_{d,q}a\operatorname{diam}(Q)^{d/2}$
when $\operatorname{diam}(Q)\ge1$. The measure is stationary with intensity one, so all the assumptions hold with the concentration exponent $\alpha=0$.

The boundary comparison~\cite[Theorem~4.3]{GHT} now gives equal finite limits, for $d\ge3$,
\begin{equation}\label{eq:marked-highdim}
 \lim_{R\to\infty}\frac1V\E\OT_{|x-y|^2}
 \left(\nu|_{Q_R},\frac{\nu(Q_R)}Vdx\right)
 =\lim_{R\to\infty}\frac1V\E\OT_{b_R^2}
 \left(\nu|_{Q_R},\frac{\nu(Q_R)}Vdx\right)=:\kappa_d.
\end{equation}
The averaging identity transfers these equal limits to $M_R$ and $B_R$
for two conditionally independent uniform samples with a common Poisson
sample size, in every $d\ge3$.

\paragraph{Identifying the Barthe--Bordenave constants.}
To obtain the constants in their fixed-size bounds, Barthe and Bordenave
study two samples whose numbers of points are independent Poisson
variables with mean $n$. For quadratic cost in $d\ge5$, this change
preserves the limiting constants of the problem with exactly $n$ points
in each sample~\cite[Section~2.1.1]{BB}.
To identify $\kappa_d$ with their constants, we therefore compare common
and independent Poisson sample sizes.
Adding or deleting one point changes either $M_R$ or $B_R$ by at most
$dR^2$. For $B_R$, leave a new point or the partner of a deleted point
unmatched. For $M_R$, either add a pair using an unmatched point or
replace a deleted matched point by an unmatched point from the same
sample. Apply these constructions in both directions. Take independent
$N_1,N_2\sim\Pois(V)$ and independent uniform sequences $(X_i),(Y_i)$
in $Q_R$, independent of the counts. For $F_R=M_R$ or $B_R$,
\[
 \frac1V\left|\E F_R(X_{1:N_1},Y_{1:N_1})
                 -\E F_R(X_{1:N_1},Y_{1:N_2})\right|
 \le\frac{dR^2}V\E|N_1-N_2|
 \le d\sqrt2\,R^{2-d/2}.
\]
This tends to zero for $d\ge5$. Rescaling to the unit cube therefore identifies the quadratic matching constant and the boundary constant $\beta_2(d),\beta'_2(d)$ of \cite[Theorem~4 and Lemma~11]{BB}:
\[
 \beta_2(d)=\beta'_2(d)=\kappa_d=:\beta_d>0.
\]
Positivity follows from their lower bound for matching under uniform sampling.

With these constants equal, the upper and lower almost-sure bounds in
\cite[Theorem~2]{BB} give the limit in Corollary~\ref{cor:highdim} for
every compactly supported density $f$. The bounded-support concentration
estimate in \cite[Proposition~4]{BB} gives convergence in every finite
$L^q$ as well.
\end{proof}
\section{The semidiscrete-to-smoothing reduction}\label{app:smoothing-proof}
The reduction rests on a deterministic comparison of transport costs
when the source distribution changes. We prove this comparison in
Appendix~\ref{app:source-comparison}, bound its error for random samples in
Appendix~\ref{app:concentration-cost-bounds}, and extend it to surfaces with
boundaries and corners in Appendix~\ref{app:boundary-corner-reduction}.
Appendix~\ref{app:spectral-term} evaluates the spectral term.

\subsection{Comparing costs under a change of source distribution}\label{app:source-comparison}
The main step, Lemma~\ref{s:lem:split}, compares optimal costs by moving
source points while keeping their targets fixed. The leading change
depends only on the source and target distributions. Constructing the
corresponding plans also requires a small adjustment to the source
density; we first recall the standard transport bound used for this adjustment.

\paragraph{Transport from a positive density}
We use the following form of Peyre's transport bound
\cite[Theorem~2.1 and Lemma~2.2]{Peyre}, expressed in a negative Sobolev norm.
Let $(M,g)$ be a compact connected smooth manifold without boundary,
of any dimension, or a surface satisfying the geometric assumptions of
Section~\ref{subsec:surface-reduction}. Let $\mu=\rho\,d\mathrm{vol}_g$,
with $\rho$ smooth and strictly positive.
Let $(\lambda_j,\psi_j)_{j\ge1}$ be the positive eigenpairs of
$-\rho^{-1}\diver_g(\rho\nabla_g)$, orthonormal in $L^2(\mu)$,
with the Neumann condition at the boundary.
For a signed measure $\zeta$ of total mass zero, set
\[
 \|\zeta\|_{-1,\mu}^2
 =\sup_{\substack{\mu(u)=0\\\int|\nabla u|^2\,d\mu\le1}}|\zeta(u)|^2
 =\sum_{j\ge1}\frac{|\zeta(\psi_j)|^2}{\lambda_j}.
\]
The supremum runs over smooth functions $u$ satisfying the displayed
constraints, with no boundary condition imposed on $u$.
The norm may be infinite.
For measures $\eta,\zeta$ of common mass $m>0$, we write
$\OT(\eta,\zeta)=mW_2^2(\eta/m,\zeta/m)$ for their squared-distance
transport cost, and set $\OT(0,0)=0$.

\begin{lemma}[Transport bound with one positive density]\label{c:lem:energy}
Let $(M,g)$ and $\mu$ be as specified at the start of
Appendix~\ref{app:source-comparison}. Let $\eta$ and $q\mu$ be finite
measures of equal mass, with $q\ge b>0$.
Then
\[
 \OT(\eta,q\mu)\le4b^{-1}\|\eta-q\mu\|_{-1,\mu}^2.
\]
The bound also holds within a cell, using the unnormalized restriction
of $\mu$ and its Neumann eigenfunctions. No derivative of $q$ is required.
\end{lemma}

\begin{proof}
For smooth densities, solve
$-\diver(\rho\nabla v)=\eta-q\mu$, with zero boundary flux at every physical or cell boundary.
Interpolate from $q\mu$ to $\eta$ by
$\eta_s=(1-\theta(s))q\mu+\theta(s)\eta$, where
$\theta(s)=1-(1-s)^2$, and use flux $\theta'(s)\rho\nabla v$.
The density along this path is at least $(1-\theta(s))b\rho$.
The Benamou--Brenier action is therefore at most
\[
 b^{-1}\int_0^1\frac{\theta'(s)^2}{1-\theta(s)}\,ds
       \int|\nabla v|^2\,d\mu
 =4b^{-1}\|\eta-q\mu\|_{-1,\mu}^2.
\]
Approximation gives the assertion for finite-energy measure endpoints;
when the norm is infinite the inequality is immediate.
\end{proof}

\paragraph{Changing the source distribution}
Assume now that $\mu$ is a probability measure. All derivatives use $g$,
and put
\[
 \diver_\mu V=\rho^{-1}\diver_g(\rho V),\qquad
 \cL=-\diver_\mu\nabla_g.
\]
We compare the costs of transport from $\mu$ and from $(1+\cL f)\mu$
to a fixed target distribution. On a closed manifold, let $k$ satisfy the cost
assumptions of Section~\ref{subsec:general-smoothing},
including~\eqref{eq:smoothing-cost-assumptions}.
Then $a d_g^2\le k\le b d_g^2$ for some $a,b>0$.
For a surface with boundary, we use only $k=d_g^2$.

\begin{lemma}[Comparison under a change of source distribution]\label{s:lem:split}
For the geometry, probability measure $\mu$, and cost $k$ specified in
Appendix~\ref{app:source-comparison}, let $f\in C^2(M)$, with $\mu(f)=0$ and
$\gamma=\|\nabla^2f\|_{\mathrm{op},\infty}$ sufficiently small.
When a boundary is present, assume $\partial_\nu f=0$ on each smooth face.
For any probability measure $\nu$ on $M$, put
$q=1+\cL f$, $E_f=\int|\nabla f|^2\dd \mu$,
$L=\OT_k(\mu,\nu)$, and $D=\OT_k(q\mu,\nu)$. Then $q$ is a positive
probability density relative to $\mu$ and
\[
|L-D-2\nu(f)+E_f|\le C\gamma(L+E_f).
\]
\end{lemma}
\begin{proof}
We move each source point $z$ to $S_1(z)$, where $S_s$ is the flow
generated by $-\nabla f$. We compute the change in cost, take infima
over the corresponding transport plans, and then correct the difference
between the source density produced by the flow and $q$.

Put $v=\nabla f$ and define the flow
$S_s$ by $\partial_sS_s=-v\circ S_s$, $S_0=\Id$, for $0\le s\le1$.
When there is a boundary, the Neumann condition makes this flow
preserve every face. At a maximum of $f$ its gradient vanishes,
including at a boundary maximum; integration along paths gives
$\|v\|_\infty\le C\gamma$. The flow satisfies
\[
 e^{-2\gamma s}|w|^2\le |dS_sw|^2\le e^{2\gamma s}|w|^2.
\]
Define $J_s$ by $S_{s\#}(J_s\mu)=\mu$.
For this flow, $\partial_sJ_s=-\diver_\mu(J_sv)$.
Since $\|\diver_\mu v\|_\infty\le C\gamma$, we have
$J_s=1+O(\gamma)$ and
\begin{equation}\label{s:eq:flux}
 \widehat q:=J_1=1+\cL f+\diver_\mu R,
 \qquad R=-v\int_0^1(J_s-1)\,ds,
 \qquad |R|\le C\gamma|v|.
\end{equation}
The flux $R$ has zero normal component at any boundary.
Both $q$ and $\widehat q$ have mass one and equal $1+O(\gamma)$.
Thus $S_1$ transports $\widehat q\mu$ exactly to $\mu$, while the
source distribution we need is $q\mu$. Their density difference is
$\widehat q-q=\diver_\mu R$.

We next compute how $S_1$ changes the transport cost.
To compare squared distances without restrictions on cut loci, take a
path $\alpha$ from $y$ to $z$ and put $\beta(s)=S_s\alpha(s)$.
Write $E(\alpha)=\int_0^1|\alpha'(s)|^2\,ds$ for its path energy.
Since $\beta'=dS_s(\alpha'-v(\alpha))$, the flow distortion and Hessian
bound give
\[
 E(\beta)=E(\alpha)+2f(y)-2f(z)+|v(z)|^2
       +O(\gamma\{E(\alpha)+|v(z)|^2\}).
\]
The cross term integrates to $2f(y)-2f(z)$, and replacing
$|v(\alpha(s))|^2$ by $|v(z)|^2$ has the displayed error.
The path correspondence is bijective, so taking energy infima gives
\[
 d_g(S_1z,y)^2-d_g(z,y)^2
 =2f(y)-2f(z)+|v(z)|^2
       +O(\gamma\{d_g(z,y)^2+|v(z)|^2\}).
\]

For general $k$, put $r(z,y)=k(z,y)-d_g(z,y)^2$.
Its vanishing quadratic part gives $|D_zr(z,y)|\le C d_g(z,y)^2$
near the diagonal. Away from it the same bound holds in the Lipschitz
sense, since $d_g(z,y)$ is bounded below. Integrating along the flow gives
\[
 |r(S_1z,y)-r(z,y)|
 \le C|v(z)|\{d_g(z,y)^2+|v(z)|^2\}
 \le C\gamma\{d_g(z,y)^2+|v(z)|^2\}.
\]
Thus the same comparison holds for $k$. We estimate only the change
in its cubic remainder; both transport problems retain the original cost $k$.

The map $(z,y)\mapsto(S_1z,y)$ bijects
couplings of $(\widehat q\mu,\nu)$ and $(\mu,\nu)$.
The leading term $2f(y)-2f(z)+|v(z)|^2$ separates into a function of
the target $y$ and a function of the source $z$. Its integral is
therefore determined by $\nu$ and $\widehat q\mu$, independently of
the coupling. Equation~\eqref{s:eq:flux} and integration by parts give
\[
 \int f\widehat q\,d\mu=E_f+O(\gamma E_f),\qquad
 \int|\nabla f|^2\widehat q\,d\mu=E_f+O(\gamma E_f).
\]
Integrating the pointwise comparison and taking both infima therefore
gives, with $\widehat D=\OT_k(\widehat q\mu,\nu)$,
\[
 L=\widehat D+2\nu(f)-E_f+O(\gamma(\widehat D+E_f)).
\]

Finally, correct the density discrepancy between $\widehat q\mu$ and
$q\mu$. The cost assumptions give
\[
 |k(z,y)-k(z',y)|
 \le C d_g(z,z')\{d_g(z,y)+d_g(z,z')\}.
\]
Gluing a coupling of source measures $\eta,\eta'$ to a transport plan
from $\eta$ to $\nu$, applying Cauchy--Schwarz, and reversing the roles
of $\eta,\eta'$ yields
\[
 |\OT_k(\eta',\nu)-\OT_k(\eta,\nu)|
 \le C\{w\sqrt{\OT_k(\eta,\nu)}+w^2\},
 \qquad w=W_{2,g}(\eta,\eta').
\]
Apply this bound to $\eta=q\mu$ and $\eta'=\widehat q\mu$.
Lemma~\ref{c:lem:energy} and~\eqref{s:eq:flux} give
\[
 w\le C\|(\widehat q-q)\mu\|_{-1,\mu}
   \le C\|R\|_{L^2(\mu)}\le C\gamma\sqrt{E_f}.
\]
The factor $\gamma$ in this bound on $w$ makes the correction small
enough to preserve the preceding cost comparison:
\[
 |\widehat D-D|\le C\gamma\sqrt{E_fD}+C\gamma^2E_f
                 \le C\gamma(D+E_f).
\]
Quadratic comparability, the triangle inequality for $W_{2,g}$,
and Lemma~\ref{c:lem:energy} give $D\le C(L+E_f)$.
Substitution proves the claimed estimate.
\end{proof}

For quadratic cost, choosing $\nu=q\mu$ gives $D=0$ and
$(q\mu)(f)=E_f$. Absorbing $C\gamma L$ in the source comparison recovers
the classical approximation~\cite{AST,AG}:
\[
 W_{2,g}^2(\mu,q\mu)=E_f+O(\gamma E_f).
\]

\subsection{Concentration and cost bounds}\label{app:concentration-cost-bounds}
To apply the source comparison, we need a small Hessian with high
probability and a rough bound on the cost multiplying it. For the Hessian,
we adapt the heat-kernel and concentration argument of Ambrosio and
Glaudo~\cite[Theorem~3.3]{AG} to general dimension and to Dirichlet as
well as equal masses. We first record the concentration
estimate used for both mass laws.

\begin{lemma}[Concentration for equal and Dirichlet masses]\label{lem:weighted-concentration}
Let $\nu_n=\sum_iw_i\delta_{Y_i}$, where the $Y_i$ are independent with
law $\mu$ and the weights are equal or Dirichlet, independently of the sites.
Suppose a family of mean-zero functions $h_x$ satisfies
$\|h_x\|_\infty\le B$ and $\mu(h_x^2)\le V$.
For $\delta>0$, suppose $N_\delta$ members of this family approximate
every member within $\delta/2$ in supremum norm. Then
\[
 \PP\left\{\sup_x|\nu_n(h_x)|>\delta\right\}
 \le C N_\delta\exp\left\{-\frac{cn\delta^2}{V+B\delta}\right\}.
\]
\end{lemma}
\begin{proof}
For equal weights, apply Bernstein's inequality at the selected members
and take a union bound. For Dirichlet weights write $w_i=E_i/\sum_jE_j$
with independent mean-one exponentials. For a fixed member $h$,
\[
 \E|E_i h(Y_i)|^p\le p!B^{p-2}V\qquad(p\ge2),
\]
so Bernstein gives the same bound for $n^{-1}\sum_iE_i h(Y_i)$.
Write
\[
 n^{-1}\sum_iE_i h(Y_i)=Z\,\nu_n(h),\qquad Z=n^{-1}\sum_iE_i.
\]
Since $Z$ is independent of $\nu_n(h)$ and $\PP(Z\ge1/2)\ge1/8$,
the tail of $\nu_n(h)$ is at most eight times that of the left-hand
side at half the threshold.
Finally, $\sum_iw_i=1$ transfers the supremum-norm approximation to each
weighted sum. A union bound proves the assertion.
\end{proof}

The expected spectral term uses the second-moment identity
\begin{equation}\label{f:eq:cov}
 n\E\nu_n(h)^2=\chi_n\mu(h^2)
 \qquad\text{when }h\text{ is deterministic and }\mu(h)=0.
\end{equation}

\subsubsection{Cost and concentration bounds in general dimension}\label{app:general-heat}
Use the assumptions of Theorem~\ref{thm:smoothing-reduction}; in
particular, the manifold has no boundary. Constants may depend on the
fixed geometry, density, and cost. The square is treated in
Appendix~\ref{app:flat-heat}.

\paragraph{Rough cost bound.}
The cost assumptions and compactness give $a d_g^2\le k\le b d_g^2$
for fixed $a,b>0$. Coupling each atom with its heat smoothing costs at
most $Cs$. The triangle inequality for $W_{2,g}$ and the
transport bound of Lemma~\ref{c:lem:energy} therefore give
\[
 n\E\OT_k(\mu,\nu_n)\le Cns+C\chi_nF(2s).
\]
The compact heat-kernel bound gives $F(s)\le C$ for $d=1$,
$F(s)\le C\log(e/s)$ for $d=2$, and $F(s)\le Cs^{1-d/2}$ for $d\ge3$.
Set $s=n^{-2/d}$. Then $n\E\OT_k(\mu,\nu_n)\le CR_d(n)$ and, for $t\ge s$,
$n\E E_{f_t}=\chi_nF(2t)\le CR_d(n)$, as required.

\paragraph{Hessian concentration.}
We show that $\|\nabla_g^2f_t\|_\infty\le\delta$ except on an event
of probability $O(n^{-12})$, where
$\delta=K\sqrt{\log n/(nt^{d/2})}$ and $K$ is a sufficiently large
constant. After choosing $K$, increase the constant in the heat-time
condition of Theorem~\ref{thm:smoothing-reduction} so that $\delta$ is
small enough for the source comparison, and in particular $\delta\le1$.
Let $p_s(x,y)$ be the heat kernel relative to $\mu$ and put
$G_t(x,y)=\int_t^\infty[p_s(x,y)-1]\,ds$.
Then $f_t(x)=\nu_n(G_t(x,\cdot))$ and
$\int G_t(x,y)\,d\mu(y)=0$. The derivative bounds for a smooth compact
heat kernel~\cite{BGV}, integrated over $s\ge t$, imply
\[
 \|\nabla_x^2G_t\|_\infty\le Ct^{-d/2},\qquad
 \sup_x\int|\nabla_x^2G_t(x,y)|^2\,d\mu(y)\le Ct^{-d/2},\qquad
 \|\nabla_x^3G_t\|_\infty\le Ct^{-(d+1)/2}.
\]
For the middle estimate, integrate the $L^2$ bound
$\|\nabla_x^2p_s(x,\cdot)\|_2\le Cs^{-1-d/4}$;
the large-time contribution decays exponentially. Apply
Lemma~\ref{lem:weighted-concentration} in finitely many coordinate
charts to the components of $\nabla_x^2G_t$.
A net of mesh comparable to $\delta t^{(d+1)/2}$ controls the supremum,
so
\[
 \PP(\|\nabla_g^2f_t\|_\infty>\delta)
 \le C\delta^{-d}t^{-d(d+1)/2}
               \exp(-cnt^{d/2}\delta^2),\qquad 0<\delta\le1.
\]
For the chosen $\delta$, we have $t\ge n^{-2/d}$ and $\delta\ge n^{-1}$,
so the prefactor is at most $Cn^{2d+1}$. Taking $K$ large enough gives
the asserted probability bound. The heat-kernel bound also gives
$\|G_t\|_\infty\le Ct^{-d/2}\le Cn$, hence
$\|f_t\|_\infty\le Cn$ and
$E_{f_t}=\int f_t(q_t-1)\,d\mu\le2\|f_t\|_\infty$.
Both transport costs are bounded on the compact domain, so
$n|\OT_k(\mu,\nu_n)-\OT_k(P_t\nu_n,\nu_n)-2\nu_n(f_t)+E_{f_t}|
\le Cn^2$.
Thus the event on which the Hessian bound fails contributes
$O(n^{-10})$ to the expected error.

\subsubsection{The square and torus}\label{app:flat-heat}\label{app:heat}
On the square and torus, positions are independent and uniform, with
equal masses or Dirichlet masses independent of the positions. Write $\mu=\mu_\Omega$,
$P_t\nu_n=q_t\mu$ and $f_t=(-\Delta_\Omega)^{-1}(q_t-1)$.

\paragraph{Heat estimates.}
Let $p_t$ be the reflecting or periodic heat kernel. On the square,
even reflection across the sides gives a smooth periodic extension,
as in~\cite[Appendix~B]{AG}.
Thus the bounds on $D_x^2G_t$ and $D_x^3G_t$ in
Appendix~\ref{app:general-heat} apply with $d=2$.
The Gaussian kernel also gives
\[
 \|p_t\|_\infty+\sup_x\int p_t(x,y)^2\,dy\le C/t,
 \qquad\|\nabla_xp_t\|_\infty\le Ct^{-3/2}.
\]
Apply Lemma~\ref{lem:weighted-concentration} to $p_t(x,\cdot)-1$
and the components of $D_x^2G_t(x,\cdot)$. Their bounds satisfy
$B,V\le C/t$. A grid of mesh $c\delta t^{3/2}$ has at most
$C\delta^{-2}t^{-3}$ points and gives
\[
 \PP(\|q_t-1\|_\infty+\|D^2f_t\|_\infty>\delta)
 \le C\delta^{-2}t^{-3}e^{-cnt\delta^2},\qquad 0<\delta\le1.
\]

\paragraph{Cost bounds.}
The rough bound in Appendix~\ref{app:general-heat}, with $d=2$, gives
\begin{equation}\label{f:eq:rough}
 n\E W_{2,\Omega}^2(\mu,\nu_n)\le C\log(en),\qquad
 d_n^\Omega(t)\le4n\E W_{2,\Omega}^2(\mu,\nu_n).
\end{equation}
The second inequality follows from the triangle inequality and heat
contraction, obtained by coupling endpoints with the same reflected or
periodic Gaussian increment.
Together with the source comparison in Lemma~\ref{s:lem:split}, these
bounds prove the square and torus cases of
Proposition~\ref{prop:surface-reduction}, without the corner error term.

\subsubsection{A lower bound for the smoothing cost}\label{app:smoothing-small-time}
Take uniform sampling on $\Omega\in\{Q,\T\}$ with either equal or
Dirichlet masses. The triangle-inequality argument at the start of
Section~\ref{sec:smoothing-reduction} requires
$n\E W_{2,\Omega}^2(P_t\nu_n,\nu_n)=o(1/\log n)$
to make its cross term $o(1)$.
We show that this condition forces $t=o(1/n)$ by proving
\[
 n\E W_{2,\Omega}^2(P_t\nu_n,\nu_n)\ge c\min\{nt,1\},
 \qquad 0<t\le1/2.
\]
Indeed, transport to $S=\{Y_1,\ldots,Y_n\}$ costs at least
$\int d_\Omega(x,S)^2\,dP_t\nu_n(x)$.
Let $X$ be drawn from the heat kernel started at $Y_1$.
Since the masses are independent of the positions and $\E w_i=1/n$,
the expectation of this integral is $\E d_\Omega(X,S)^2$.
Take $r^2=\varepsilon\min\{t,1/n\}$. The heat-kernel bound $p_t\le C/t$
and independence of $X$ from $Y_2,\ldots,Y_n$ give
$\PP(d_\Omega(X,S)\le r)\le Cr^2/t+Cnr^2\le1/2$
for a fixed small $\varepsilon>0$. Thus $\E d_\Omega(X,S)^2\ge r^2/2$.
The required smallness condition therefore forces
$nt=o(1/\log n)$, and in particular $t=o(1/n)$.

\subsection{Reduction on surfaces with boundaries and corners}\label{app:boundary-corner-reduction}
To prove Proposition~\ref{prop:surface-reduction} on surfaces with boundary,
we use the same source comparison. The additional obstacle is that
the Poisson potential may have unbounded Hessian at a corner.
We make it constant near each vertex and show that this modification
changes the expected costs by a vanishing amount. The heat and Poisson
estimates below provide the bounds needed for this modification.

\subsubsection{Heat and Poisson estimates}
\paragraph{Reflected heat smoothing.}\label{app:reflected-heat}
These heat bounds also serve the subdivision argument, so we state
them for a fixed surface and for small cells.

Let $B$ be a fixed compact connected surface whose boundary is empty,
smooth, or has convex corners, with scale $\ell=1$, or a domain of
diameter comparable to $\ell$ which, after rescaling, is uniformly
bi-Lipschitz to a reference square or triangle.
Assume the metric and density are uniformly comparable to Euclidean
metric and area in the rescaled charts. Let $P_t^B$ be the Neumann
semigroup for $\int_B|\nabla u|^2\,d\mu$, and let $p_t^B$ be its kernel
relative to the \emph{unnormalized} measure $\mu|_B$.
Thus $p_B=\mu(B)\asymp\ell^2$ and the invariant density is $1/p_B$.
Write $d_B$ for intrinsic distance in $B$.

\begin{lemma}[Neumann heat-kernel bounds]\label{lem:reflected-heat}\label{s:lem:cellheat}
For the domains $B$ and reference measures specified in
Appendix~\ref{app:reflected-heat}, there are $C,c>0$ and $\alpha\in(0,1)$,
uniform over the domains, such that for $0<t\le\ell^2$,
\begin{align}
 0\le p_t^B(x,y)&\le Ct^{-1}\exp\{-d_B(x,y)^2/(Ct)\},\label{s:eq:cellgaussian}\\
 |p_t^B(x,y)-p_t^B(x',y)|
   &\le Ct^{-1}\min\{1,(d_B(x,x')/\sqrt t)^\alpha\},\label{s:eq:cellholder}\\
 \int_Bd_B(x,y)^2p_t^B(x,y)\,d\mu(y)&\le Ct.\label{s:eq:cellmovement}
\end{align}
The first nonconstant eigenvalue is at least $c\ell^{-2}$. For eigenfunctions normalized in $L^2(\mu|_B)$,
\begin{equation}\label{s:eq:celldiag}
 \sup_{B,x\in\overline B}\sum_{j\ge1}
 \frac{e^{-2s\lambda_{B,j}}}{\lambda_{B,j}}\psi_{B,j}(x)^2
 \le C\log(e\ell^2/s),\qquad 0<s\le\ell^2.
\end{equation}
\end{lemma}
\begin{proof}
For a closed surface use the smooth heat-kernel estimates~\cite{BGV}.
Otherwise use finitely many interior, face, and corner charts. The
rescaled domains are uniform: points can be joined by short curves
staying quantitatively away from the boundary except near their
endpoints. The metric and measure bounds make the Neumann form
uniformly elliptic. Extending its coefficients and density preserves
these bounds, giving ambient Gaussian estimates.
The reflected-domain theorem~\cite[Theorem~2.8]{Murugan} then gives a
conservative diffusion and \eqref{s:eq:cellgaussian}, since intrinsic
and coordinate distances are comparable. The upper bound also follows
from the Neumann form~\cite{Daners}.

H\"older regularity follows separately. Volume doubling and local
Poincar\'e inequalities survive the coordinate changes. The parabolic
Harnack principle~\cite{SaloffHarnack} and oscillation reduction give
a common exponent $\alpha$. Apply these to $p_u^B(\cdot,y)$ for
$u\in[t/2,t]$, using \eqref{s:eq:cellgaussian}, to obtain
\eqref{s:eq:cellholder}, including at the boundary.

The rescaled Poincar\'e inequality
\[
 \int_B|u-\mu(u\mid B)|^2\,d\mu
 \le C\ell^2\int_B|\nabla u|^2\,d\mu
\]
gives the spectral gap. Integrating the Gaussian tail in annuli gives
\eqref{s:eq:cellmovement}, including for starting points on a cell face.
For \eqref{s:eq:celldiag}, integrate the nonconstant heat diagonal:
it is at most $C/u$ for $2s\le u\le\ell^2$, and the spectral gap
bounds the remaining integral by $C$.
\end{proof}

\paragraph{Poisson estimates near the physical boundary.}\label{app:boundary}
We need gradient bounds on the whole surface and Hessian bounds away
from the vertices. At a corner of angle $\omega$, a Neumann harmonic
function can contain $r^{\pi/\omega}\cos(\pi\theta/\omega)$; its gradient
is bounded when $\omega<\pi$, whereas its Hessian can be unbounded when
$\omega>\pi/2$. We also bound the second derivatives of the Poisson
solution for use in surface subdivision. Here $H^j(M)$, for $j=1,2$,
denotes the Sobolev space of functions whose weak derivatives through
order $j$ are square-integrable in local coordinates.

Let $(M,g)$ be a fixed compact connected surface satisfying the geometric
assumptions of Section~\ref{subsec:surface-reduction}. Its boundary may be empty;
otherwise it consists of smooth closed curves and finitely many smooth
arcs meeting at angles in $(0,\pi)$. The metric
and the strictly positive density $\rho$ are smooth up to each corner
in ambient coordinates. Put $\mu=\rho\,dA$, let $\mathcal V$ be the finite
set of vertices, and let
\[
 \cL=-\rho^{-1}\diver(\rho\nabla),\qquad P_t=e^{-t\cL},
\]
with the Neumann boundary condition. Write $p_t(x,y)$ for its kernel
relative to $\mu$ and
\[
 G_t(x,y)=\int_t^\infty\{p_s(x,y)-1\}\,ds.
\]
If $\mathcal V$ is empty, conditions on the distance to $\mathcal V$
are omitted.

\begin{lemma}[Poisson estimates with boundary]\label{b:lem:heat}
There are constants $C,R_0>0$, depending on the fixed surface and
density, such that, for $0<t\le1/2$,
\begin{align*}
 \sup_{x,y}|G_t(x,y)|&\le C\log(e/t),&
 \sup_{x,y}|\nabla_xG_t(x,y)|&\le Ct^{-1/2},\\
 \sup_x\int_M|\nabla_xG_t(x,y)|^2\,d\mu(y)
     &\le C\log(e/t).
\end{align*}
For $0<R\le R_0$ and $d(x,\mathcal V)\ge R$,
\begin{align*}
 \sup_y|\nabla_x^2G_t(x,y)|
     &\le C\{t^{-1}+R^{-2}\log(e/R)\},\\
 \|\nabla_x^2G_t(x,\cdot)\|_2
     &\le C\{t^{-1/2}+R^{-2}\}.
\end{align*}
The kernels in these estimates have H\"older seminorms in $x$ bounded
by a fixed polynomial in $t^{-1}$ and $R^{-1}$ on their respective
regions. The constants and the H\"older exponent are independent of
$t$ and $R$.

For mean-zero $h\in L^2(\mu)$, the Neumann solution $u=\cL^{-1}h$ satisfies
\[
 \|u\|_{H^2(M)}\le C\|h\|_{L^2(\mu)}.
\]
It has an extension to a fixed smooth ambient surface with the same
$H^2$ bound.
\end{lemma}

\begin{proof}
We first obtain a gradient estimate from elliptic regularity at a
corner. The Hessian estimate is then needed only in ordinary interior
or smooth-boundary neighborhoods.

If there are vertices, choose $p>2$ sufficiently close to $2$ that
\[
 2-2/p<\min_{v\in\mathcal V}\pi/\omega_v,
\]
where $\omega_v$ is the angle at $v$; otherwise choose any fixed $p>2$.
The Neumann corner estimate gives
local $W^{2,p}$ regularity. Its threshold follows from the homogeneous
wedge modes $1$ and $r^{k\pi/\omega_v}\cos(k\pi\theta/\omega_v)$:
below this threshold its second derivatives belong to $L^p$.
Integer modes that are polynomials can have better regularity. The variable-coefficient estimate on curvilinear
domains follows by flattening the faces, freezing the principal
coefficients, and absorbing their variation on small neighborhoods;
see~\cite[Theorem~3.2]{DaugeMixed}.
The 2-dimensional estimate is also obtained from the edge estimate
there by taking a product with an interval and considering functions
independent of the additional coordinate. The angle is measured after
normalizing the frozen principal form, so it is the angle of $g$.

Rescaling this estimate and applying the Sobolev embedding
$W^{2,p}\subset C^{1,1-2/p}$ in dimension two gives
\begin{equation}\label{b:eq:localgradient}
 |\nabla u(x)|
 \le C\{a^{-1}\|u\|_{L^\infty(U_{2a})}
                    +a\|\cL u\|_{L^\infty(U_{2a})}\}.
\end{equation}
Here $U_{2a}$ is an interior, face, or corner neighborhood of radius
comparable to $a$, and $u$ satisfies the Neumann condition on its
physical boundary. A finite collection of charts makes the estimate
uniform for small $a$. The corresponding scaled H\"older estimate for
$\nabla u$ follows from the same embedding. Lower-order terms have
bounded coefficients and are absorbed in the displayed estimate.

Lemma~\ref{lem:reflected-heat}, applied to the fixed surface with
$\ell=1$, gives Gaussian heat bounds and a positive spectral gap.
In particular,
\[
 \|P_s\|_{L^2(\mu)\to L^\infty}\le Cs^{-1/2},
 \qquad
 \|\cL P_s\|_{L^2(\mu)\to L^\infty}
 \le \|P_{s/2}\|_{2\to\infty}\,
       \|\cL P_{s/2}\|_{2\to2}\le Cs^{-3/2}
\]
for $s\le1$. Apply \eqref{b:eq:localgradient} to $u=P_sh$ with
$a$ comparable to $\sqrt s$. We obtain
\[
 \|\nabla P_s\|_{2\to\infty}\le Cs^{-1}.
\]
By duality at each $x$, and then by the semigroup identity,
\begin{equation}\label{b:eq:gradientkernel}
 \|\nabla_xp_s(x,\cdot)\|_2\le Cs^{-1},
 \qquad
 |\nabla_xp_s(x,y)|\le Cs^{-3/2}.
\end{equation}
The spectral gap supplies exponential decay for the differentiated
kernels when $s\ge1$. Integration gives the asserted bounds on
$G_t$ and its gradient.

To obtain the logarithmic bound on
$\int|\nabla_xG_t(x,y)|^2\,d\mu(y)$, combine
\eqref{b:eq:gradientkernel} with the spectral identity for the
orthonormal eigenfunctions $(\psi_j)$:
\[
 \sum_j e^{-s\lambda_j}|\nabla\psi_j(x)|^2
   =\|\nabla_xp_{s/2}(x,\cdot)\|_2^2\le Cs^{-2}
 \qquad(s\le1).
\]
Tonelli's theorem and exponential decay at large $s$ therefore give
\begin{align*}
 \int|\nabla_xG_t(x,y)|^2\,d\mu(y)
 &=\sum_j\frac{e^{-2t\lambda_j}}{\lambda_j^2}
                 |\nabla\psi_j(x)|^2\\
 &=\int_{2t}^\infty(s-2t)
           \sum_j e^{-s\lambda_j}|\nabla\psi_j(x)|^2\,ds
 \le C\log(e/t).
\end{align*}

Now suppose $d(x,\mathcal V)\ge R$. A neighborhood of $x$ of radius
$c_0R$ meets at most one smooth boundary face. Interior and Neumann
boundary parabolic estimates~\cite{Lieberman}, rescaled to radius
$\sqrt s$, imply, for $s\le c_1R^2$,
\[
 |\nabla_x^2p_s(x,y)|\le Cs^{-2},\qquad
 \|\nabla_x^2p_s(x,\cdot)\|_2\le Cs^{-3/2}.
\]
For the second inequality apply the local estimate to $P_sh$ and use
$\|P_uh\|_\infty\le Cs^{-1/2}\|h\|_2$ for $u$ comparable to $s$.
For $c_1R^2\le s\le1$, use a neighborhood of radius $c_0R$ instead.
The same estimates give
\[
 |\nabla_x^2p_s(x,y)|\le CR^{-2}s^{-1},\qquad
 \|\nabla_x^2p_s(x,\cdot)\|_2\le CR^{-2}s^{-1/2}.
\]
For $s\ge1$ replace the time factors by exponential decay, with the
same factor $CR^{-2}$. Integrating over these three ranges proves the
asserted Hessian bounds.
The scaled elliptic H\"older estimate used above and the local
parabolic H\"older estimates give the asserted polynomial seminorm
bounds. Only finitely many coordinate charts are involved.

The Hilbert-space corner estimate gives the global $H^2$ bound for
the Neumann resolvent; the spectral gap removes the constant mode.
A bounded $H^2$ extension
operator exists in each Lipschitz corner chart~\cite{Stein};
a partition of unity gives an extension to a fixed ambient surface.
\end{proof}

\subsubsection{Proof of the surface reduction}\label{app:surface-reduction}
Use the surface, density, and mass laws of
Section~\ref{subsec:surface-reduction}. The following cutoff preserves
the Neumann condition when we make the potential constant near each vertex.

\begin{lemma}\label{s:lem:cornercut}
For each vertex $v$ and every sufficiently small $R>0$, there is a smooth
function $\xi_{v,R}$ equal to zero near $v$ and one outside a neighborhood
of radius $CR$, with
\[
 \partial_\nu\xi_{v,R}=0,\qquad
 |\nabla\xi_{v,R}|\le C/R,\qquad
 |\nabla^2\xi_{v,R}|\le C/R^2.
\]
These neighborhoods are disjoint for small $R$ and have total area $O(R^2)$.
\end{lemma}
\begin{proof}
Extend the metric smoothly across the vertex and put
$u(x)=d_{\mathrm{ext}}(v,x)^2$ in a fixed normal neighborhood.
On each of the two boundary faces, parametrized by arclength $s$ from
$v$, the function $a_i(s)=\partial_{\nu_i}u$ satisfies
$a_i(s)=O(s^2)$ and $a_i'(s)=O(s)$.
Let $d_i$ be inward signed distance to the extended face and $s_i$ its
nearest-point arclength. Choose angular cutoffs $\eta_i$, equal to one
near the respective face and supported in disjoint thin sectors.
Their derivatives of order $k$ are $O(r^{-k})$, where
$r=d_{\mathrm{ext}}(v,x)$. Then
\[
 h=u+\sum_i\eta_i d_i a_i(s_i)
\]
has zero normal derivative on both faces. It satisfies
$h=u+O(r^3)$, $|\nabla h|\le Cr$, and $|\nabla^2h|\le C$.
For a fixed smooth function $\vartheta$ equal to zero near zero and one
above a fixed positive threshold, take
$\xi_{v,R}=\vartheta(h/R^2)$.
It is constant near the vertex; the derivative bounds and area estimate follow.
\end{proof}

\begin{proof}[Proof of Proposition~\ref{prop:surface-reduction}]
Write $P_t\nu_n=q_t\mu$ and $f_t=\cL^{-1}(q_t-1)$.
For the square and torus with uniform sampling, the proof of
Theorem~\ref{thm:smoothing-reduction} applies using
Appendix~\ref{app:flat-heat}. We prove the general surface estimate.
For a small radius $R>0$ and each vertex $v$, put
$e_{v,R}=(1-\xi_{v,R})\{f_t(v)-f_t\}$, and define
\[
 e_R=\sum_v e_{v,R},\qquad
 f_R=f_t+e_R-\mu(e_R),\qquad q_R=1+\cL f_R.
\]
When there are no vertices, set $e_R=0$. The function $f_R$ is constant
near each vertex and has zero normal derivative on every face.
Choose $R=r_0n^{-1/100}$ and
$\delta=K\sqrt{(t^{-1}+R^{-4})\log n/n}$,
with $r_0>0$ sufficiently small and $K$ sufficiently large.
Increasing the constant in $nt\ge C\log n$ makes $\delta$ small enough
for the source comparison. Write $E_f=\int|\nabla f|^2\,d\mu$.
We bound the changes in $2\nu_n(f_t)-E_{f_t}$ and in the smoothing cost,
and verify the Hessian bound needed to apply Lemma~\ref{s:lem:split}.

\emph{Change in $2\nu_n(f_t)-E_{f_t}$.} Lemma~\ref{b:lem:heat} gives
$\sup_x\|\nabla_xG_t(x,\cdot)\|_2^2\le C\log(e/t)$.
The covariance identity~\eqref{f:eq:cov} and integration along paths
from the vertex yield
\[
 n\E|\nabla f_t(x)|^2\le C\log(en),\qquad
 n\E|f_t(x)-f_t(v)|^2\le Cr^2\log(en)
 \quad(d_M(x,v)\le Cr).
\]
Such paths have length at most $Cr$ in a fixed corner chart.
The cutoff bounds and the area of their supports now give
\[
 n\E\int|\nabla e_R|^2\,d\mu\le CR^2\log(en),\qquad
 n\E|E_{f_R}-E_{f_t}|\le CR^2\log(en).
\]
For the signed term, use covariance directly. Its kernel is supported
near the vertices, so
\[
 n\E\nu_n(e_R)
 =\chi_n\sum_v\int(1-\xi_{v,R}(x))
       \{G_t(v,x)-G_t(x,x)\}\,d\mu(x)
 =O(R^2\log(en)).
\]
The last bound uses $\|G_t\|_\infty\le C\log(e/t)$ from
Lemma~\ref{b:lem:heat}. Moreover $\E \mu(e_R)=0$.
The covariance calculation in Theorem~\ref{thm:smoothing-reduction}
identifies $n\E[2\nu_n(f_t)-E_{f_t}]$ with the spectral term in
\eqref{eq:surface-decomposition}. Replacing $f_t$ by $f_R$ therefore changes that
term by $O(R^2\log(en))$.

\emph{The Hessian bound.}
Lemma~\ref{b:lem:heat} and the cutoff derivative bounds give the following
supremum and squared $L^2$ bounds for the kernel defining $\nabla^2f_R$:
\[
 B\le C\{t^{-1}+R^{-2}\log(e/R)+R^{-1}t^{-1/2}\},\qquad
 V\le C\{t^{-1}+R^{-4}+R^{-2}\log(e/t)\}.
\]
For the cutoff terms, integrate the gradient kernel along a path of
length $O(R)$ from the vertex. Since $t\ge1/n$ and $R=r_0n^{-1/100}$,
both $B$ and $V$ are at most $C(t^{-1}+R^{-4})$.
The same bound holds for the density kernel $p_t(x,\cdot)-1$.
Polynomial H\"older bounds give spatial nets of polynomial size in $n$.
Lemma~\ref{lem:weighted-concentration}, with $K$ large enough, therefore gives
\[
 \|q_t-1\|_\infty+\|\nabla^2f_R\|_\infty\le C\delta
\]
outside an event of probability $O(n^{-12})$.
In particular, both $q_R$ and $q_t$ are at
least $1/2$ when these bounds hold. We use $q_R\mu$ only on this good event.

\emph{Change in the smoothing cost.}
On the good event, Lemma~\ref{c:lem:energy} and
$q_R-q_t=\cL e_R$ give
\[
 W_2^2(q_R\mu,P_t\nu_n)\le C\int|\nabla e_R|^2\,d\mu.
\]
The rough cost bound at time $1/n$ and
$W_{2,M}^2(P_t\nu_n,\mu)\le4E_{f_t}$ from Lemma~\ref{c:lem:energy} give
\[
 n\E\bigl[W_{2,M}^2(\mu,\nu_n)+W_{2,M}^2(P_t\nu_n,\nu_n)\bigr]
 \le C\log(en).
\]
The triangle inequality and Cauchy--Schwarz therefore bound the
expectation of
$n|W_{2,M}^2(q_R\mu,\nu_n)-W_{2,M}^2(P_t\nu_n,\nu_n)|$
over the good event by $CR\log(en)$. Applying
Lemma~\ref{s:lem:split} to $f_R$ contributes $O(\delta\log(en))$.
On the exceptional event, the kernel and cutoff bounds give
$\|f_R\|_\infty\le C\log(en)$ and $\|\nabla f_R\|_\infty\le Ct^{-1/2}$.
Since $t\ge1/n$ and the transport costs are bounded,
\[
 n\bigl|W_{2,M}^2(\mu,\nu_n)-W_{2,M}^2(P_t\nu_n,\nu_n)
                   -2\nu_n(f_R)+E_{f_R}\bigr|\le Cn^2.
\]
This contributes $O(n^{-10})$ in expectation. Combining the bounds
proves~\eqref{eq:surface-decomposition} with error
$O((\delta+R)\log(en)+n^{-10})$. Since
$\delta+R\le C\sqrt{\log n/(nt)}+Cn^{-1/100}$,
the stated estimate follows.
\end{proof}

\subsection{The spectral term}\label{app:spectral-term}
\subsubsection{The spectral expansion on surfaces}\label{app:surface-trace}
We prove the expansion in \eqref{eq:h-general}, which separates the
logarithmic and constant parts of the spectral term.
Use the surface and sampling distribution of
Section~\ref{subsec:surface-reduction}.
The \emph{heat trace} of $P_t=e^{-t\cL}$ is
$1+\sum_{j\ge1}e^{-t\lambda_j}$, where the constant eigenfunction
contributes the term $1$. We show that
\[
 \sum_{j\ge1}e^{-t\lambda_j}=\frac A{4\pi t}+O(t^{-1/2}).
\]
The ordinary Neumann Laplacian on $(M,g)$ has heat trace
$A/(4\pi t)+O(t^{-1/2})$ by~\cite[Theorem~1.2]{NRS}; without boundary,
the usual heat expansion~\cite{BGV} gives the stronger remainder $O(1)$.
Conjugation by $\sqrt\rho$ turns $\cL$
into this Laplacian plus a bounded potential and a bounded Robin
boundary term. The boundary coefficient can have either sign.
For normalized $u$, the trace inequality
\[
 \|u\|_{L^2(\partial M)}^2
       \le C\|u\|_{H^1(M)}\|u\|_{L^2(M)}+C\|u\|_{L^2(M)}^2
\]
and the min--max principle show that its eigenvalues differ from the
ordinary Neumann eigenvalues $\lambda_j^0$ by at most
$C(\sqrt{\lambda_j^0}+1)$. Since $\lambda_j^0\asymp j$, the difference
of the two heat traces is bounded by
\[
 Ct\sum_{j\ge1}(\sqrt j+1)e^{-ctj}=O(t^{-1/2}).
\]
Removing the constant mode changes the trace by one.
Integrating over $[t,\infty)$ gives
$F_{M,\mu}(t)=A(4\pi)^{-1}\log(1/t)+h_{M,\mu}+O(\sqrt t)$.
Combining the expansions at $t$ and $2t$ proves \eqref{eq:h-general}.

\subsubsection{The square--torus difference}\label{app:flat-spectrum}
We compute $h_Q-h_{\T}$ for the mean-cost comparison in
Section~\ref{sec:source-target}.
The positive eigenvalues are $\pi^2(k^2+l^2)$ on $Q$, with
$k,l\ge0$, and $4\pi^2(k^2+l^2)$ on $\T$, with $k,l\in\mathbb Z$,
in both cases excluding $(k,l)=(0,0)$. Separating the terms on the
coordinate axes from those with $kl\ne0$ gives
\[
 F_Q(t)=F_{\T}(t/4)+\sum_{j\ge1}\frac{e^{-\pi^2j^2t}}{\pi^2j^2}.
\]
The series equals $1/6+O(\sqrt t)$. Applying \eqref{eq:h-general}
with $A=1$ therefore yields
\[
 h_Q-h_{\T}=\frac16+\frac{\log2}{2\pi}.
\]
\section{Subdivision of the smoothing cost}\label{app:flat-proof}\label{app:local-cost}
We prove the subdivision estimate in Lemma~\ref{g:lem:grid} and the
surface formula in Theorem~\ref{thm:surface-smoothing}. The main issue is
the cost of reassigning mass near cell boundaries so that transport stays
within each cell. We bound this added cost through the optimal transport
potential. Cancellation between the sample and its own smoothing makes
the potential small in $L^2$, allowing an expected added cost of $o(1)/n$.
Section~\ref{app:transport-bounds} first bounds the longest transport
distance, so that allocations far from the boundaries can be kept.
Section~\ref{app:localization} contains the cancellation and repair
argument. We apply it to the square and torus in
Section~\ref{app:flat-local}, complete the sample-size comparison in
Section~\ref{app:flat-counts}, and treat surfaces in
Section~\ref{app:surface-proof}.

\subsection{Bounding the longest transport distance}\label{app:transport-bounds}
We first bound the longest distance travelled in an optimal plan.
An allocation whose target is farther than this bound from the cell
boundary stays inside the cell. The following lemma extends the displacement
estimate of Ambrosio, Glaudo, and Trevisan~\cite{Map} to costs comparable
with squared distance, using a lower bound on the source mass in balls.

\begin{lemma}\label{lem:local-cost-localization}
Let $(M,d_0)$ be a compact connected geodesic space of diameter $D>0$.
Suppose that a probability measure $\mu$ satisfies
\[
 \mu(B(x,r))\ge c_\mu r^s
 \qquad(x\in M,\ 0<r\le D)
\]
for constants $c_\mu,s>0$. Let $k$ be continuous and satisfy
$a_0d_0^2\le k\le b_0d_0^2$ with $a_0,b_0>0$.
For any $k$-optimal coupling $\pi$ from $\mu$ to a probability measure,
\begin{equation}\label{eq:local-cost-localization}
 \sup_{(x,z)\in\operatorname{supp}\pi}d_0(x,z)
 \le C\left(\int k\,d\pi\right)^{1/(s+2)},
\end{equation}
where $C$ depends only on $a_0,b_0,c_\mu,s$.
\end{lemma}

\begin{proof}
Write $E=\int k\,d\pi$ and fix a support pair $(x_0,z_0)$ with
$r=d_0(x_0,z_0)>0$. Choose an integer $m>8b_0/a_0$ and set
$\epsilon=1/(10m)$. On a minimizing geodesic from $z_0$ to $x_0$,
place $m-1$ interior points $s_i$ at distances $ir/m$ from $z_0$.
The balls $B_i=B(s_i,\epsilon r)$ are disjoint and exclude $x_0$.

Suppose
\begin{equation}\label{eq:local-cost-small-energy}
 E<a_0c_\mu\epsilon^{s+2}r^{s+2}.
\end{equation}
Each $B_i$ then contains a point $x_i$ in a support pair $(x_i,z_i)$ with
$d_0(x_i,z_i)<\epsilon r$. Otherwise the whole source mass in $B_i$
would travel at least $\epsilon r$, giving
\[
 E\ge a_0\epsilon^2r^2 \mu(B_i)
   \ge a_0c_\mu\epsilon^{s+2}r^{s+2}.
\]
Optimality implies that cyclically reassigning the targets among finitely
many support pairs cannot reduce their total cost. This property is
called \emph{cyclical monotonicity}~\cite[Chapter~5]{Villani}.
Apply it to the $m$ chosen pairs, assigning
$x_1$ to $z_0$, $x_i$ to $z_{i-1}$ for $2\le i\le m-1$, and
$x_0$ to $z_{m-1}$:
\[
 k(x_0,z_0)+\sum_{i=1}^{m-1}k(x_i,z_i)
 \le k(x_1,z_0)+\sum_{i=2}^{m-1}k(x_i,z_{i-1})
       +k(x_0,z_{m-1}).
\]
Every reassigned edge has length at most $r/m+3\epsilon r$.
The left side is at least $a_0r^2$, whereas the right side is at most
\[
 b_0m(r/m+3\epsilon r)^2
 =\frac{169b_0}{100m}r^2<a_0r^2.
\]
This contradiction rules out \eqref{eq:local-cost-small-energy} and proves
\eqref{eq:local-cost-localization}. The cyclic inequality remains valid
if some chosen pairs have the same target.
\end{proof}

On a surface as in Section~\ref{subsec:surface-reduction} with smooth positive source
density, the ball condition holds with $s=2$. The displacement bound also
controls the effect of changing the cost by $O(d_0^3)$: along either
optimal plan, the cubic error is bounded by a constant times the longest
distance times the squared-distance cost. We use this comparison to handle boundaries in
Section~\ref{app:surface-confinement} and distinct distributions in
Appendix~\ref{app:cost-comparison}.

\begin{proposition}[Cubic perturbations of quadratic cost]\label{prop:local-cost-deterministic}
Let $(M,d_0)$ and $\mu$ satisfy the assumptions of
Lemma~\ref{lem:local-cost-localization} with $s=2$. Suppose $k$ is continuous,
$a_0d_0^2\le k\le b_0d_0^2$, and $|k-d_0^2|\le Kd_0^3$.
Then every probability measure $\eta$ satisfies
\begin{equation}\label{eq:local-cost-deterministic}
 |\OT_k(\mu,\eta)-W_{2,d_0}^2(\mu,\eta)|
 \le C\bigl(W_{2,d_0}^2(\mu,\eta)\bigr)^{5/4}.
\end{equation}
The constant depends only on $a_0,b_0,K,c_\mu$. It is uniform when $\mu$
is replaced by $q\mu$ with $1/2\le q\le2$ and $\mu(q)=1$.
\end{proposition}

\begin{proof}
Put $Q=W_{2,d_0}^2(\mu,\eta)$ and $E=\OT_k(\mu,\eta)$.
The quadratic bounds give $a_0Q\le E\le b_0Q$.
Let $\pi_Q$ and $\pi_k$ be optimal plans for the two costs.
Lemma~\ref{lem:local-cost-localization} gives
\[
 \int d_0^3\,d\pi_Q\le C Q^{1/4}\int d_0^2\,d\pi_Q
                         =CQ^{5/4},
 \qquad
 \int d_0^3\,d\pi_k\le C E^{1/4}a_0^{-1}E
                         \le CQ^{5/4}.
\]
Testing each optimum in the other objective and using
$|k-d_0^2|\le Kd_0^3$ proves the claim.
\end{proof}

\subsection{Cancellation and repair near cell boundaries}\label{app:localization}
We now bound the cost of reassigning mass near cell boundaries.
Integration by parts will express the error in terms of
$\int\phi^2\,d\mu$, where $\phi$ is the optimal transport potential.
This is useful because cancellation between the sample and its own
smoothing gives a much smaller bound for $\int\phi^2\,d\mu$ than for
the transport cost, which involves $\int|\nabla\phi|^2\,d\mu$.
We first bound the potential in terms of the difference between the source
and target distributions, then use their shared sampling fluctuations to
bound that difference. Here $\phi$ generates the optimal allocation; it is distinct
from the Poisson potential $f_t$ in Section~\ref{sec:smoothing-reduction}.

Let $\mu=\rho\,dA$ on a fixed surface $M$ as in
Section~\ref{subsec:surface-reduction}. Write $\cL=-\rho^{-1}\diver(\rho\nabla)$,
with the Neumann condition when there is a boundary. Use a smooth closed
extension of $M$ to define the ambient distance $\bar d$; when $M$ is
closed, use its own distance. All constants below may depend on this
fixed geometry and density. For a signed measure $\zeta$ of total mass
zero, define the following norm through its Poisson potential:
\[
 \|\zeta\|_{-2,\mu}^2
 =\|\cL^{-1}\zeta\|_{L^2(\mu)}^2
 =\sum_{j\ge1}\frac{|\zeta(\psi_j)|^2}{\lambda_j^2},
\]
where $(\lambda_j,\psi_j)$ are the eigenpairs of $\cL$ with
$\lambda_j>0$ and $(\psi_j)_{j\ge1}$ orthonormal in $L^2(\mu)$.

\begin{lemma}[An $L^2$ bound for the optimal potential]\label{lem:potential-l2}
Suppose the Neumann inverse satisfies
$\|\cL^{-1}h\|_{H^2(M)}\le C\|h\|_{L^2(\mu)}$ for $\mu(h)=0$, and its
solutions have bounded $H^2$ extensions to the ambient surface.
Let $\mu(q)=1$, $\|q-1\|_\infty\le\delta\le1/2$, and let $\nu$ be a
finite atomic probability measure. Write its optimal map from $q\mu$ for
cost $\bar d^2$ as $T(x)=\exp_x(\nabla\phi(x))$, with $\mu(\phi)=0$, and put
\[
 D=\int q|\nabla\phi|^2\,d\mu,\qquad
 L=\|\nabla\phi\|_\infty,\qquad V=\int\phi^2\,d\mu.
\]
If $L$ is smaller than a fixed fraction of the ambient injectivity radius, then
\[
 V\le C\{\|\nu-q\mu\|_{-2,\mu}^2+(\delta^2+L^2)D\}.
\]
\end{lemma}
\begin{proof}
Set $h=\cL^{-1}\phi$ and extend it with its $H^2$ bound. Expanding
$(\nu-q\mu)(h)$ along the transport paths will recover
$V=\int\phi^2\,d\mu$. To control the remainder, put $v=\nabla\phi$ and
$F_u(x)=\exp_x(uv(x))$, $0\le u<1$: this moves each point a fraction
$u$ of the way to its assigned target.
On a region assigned to a target $y$,
$F_u(x)=\exp_y((1-u)\log_yx)$. Its Jacobian is
$(1-u)^2$ times a bounded ratio of smooth volume densities.
The images of distinct pieces are disjoint. Indeed, if edges of lengths
$a,b$ meet at time $u\in(0,1)$, cyclical monotonicity gives
\[
 a^2+b^2\le\{ua+(1-u)b\}^2+\{ub+(1-u)a\}^2
 =a^2+b^2-2u(1-u)(a-b)^2.
\]
Thus $a=b$, and the exchanged paths are minimizing geodesics without a
corner at the meeting point. Uniqueness for the geodesic equation makes
the original endpoints equal. Consequently $F_{u\#}(q\mu)$ has density
at most $C(1-u)^{-2}$ relative to ambient area.

Geodesic Taylor expansion and Cauchy--Schwarz give
\[
 (\nu-q\mu)(h)=\int q\nabla\phi\cdot\nabla h\,d\mu+R(h),\qquad
 |R(h)|\le CL\sqrt D\,\|h\|_{H^2(M)}.
\]
The factor $1-u$ in Taylor's remainder cancels the factor $(1-u)^{-1}$
from the square root of the interpolation density bound. Approximation
in $H^2$ justifies the identity at the atomic targets, since it gives
uniform convergence in dimension two. The weak Neumann equation gives
$\int\nabla\phi\cdot\nabla h\,d\mu=V$; it imposes no boundary condition
on $\phi$. Separating $q=1+(q-1)$ therefore yields
\[
 V\le\sqrt V\{\|\nu-q\mu\|_{-2,\mu}+C(\delta+L)\sqrt D\},
\]
which proves the estimate.
\end{proof}

\paragraph{Cancellation between the sample and its smoothing.}
We apply the potential bound with $q\mu=S_t\nu_n$, so the signed
difference to estimate is $\nu_n-S_t\nu_n$. Let
$\nu_n=\sum_iW_i\delta_{Y_i}$, where the sites are independent with law
$\mu$ and the independent mass vector is equal or Dirichlet. Let
$S_t=P_t^{\cB}$ be the direct sum of weighted Neumann heat operators in a
partition $\cB$. The covariance identity~\eqref{f:eq:cov} gives
$\E\nu_n=\mu$ and $n\E\nu_n(h)^2=\chi_n\|h\|_2^2$ for
$\mu(h)=0$, with $\chi_n\le2$.
For every deterministic $h\in L^2(\mu)$, self-adjointness of $S_t$ gives
\begin{equation}\label{eq:smoothing-discrepancy}
 (\nu_n-S_t\nu_n)(h)=\nu_n((I-S_t)h).
\end{equation}
Since $S_t$ preserves $\mu$, the function $(I-S_t)h$ has zero $\mu$-mean.
The covariance formula therefore gives the exact identity
\begin{equation}\label{eq:localization-covariance}
 n\E\bigl[((\nu_n-S_t\nu_n)(h))^2\bigr]
   =\chi_n\|(I-S_t)h\|_2^2.
\end{equation}
The variance therefore depends on $h-S_th$: a function that changes
little under heat smoothing has nearly the same integral against the
sample and its smoothing. This is the cancellation we need.
For $h\in H^1(M)$, the direct-sum form agrees with the global energy, so
$\|(I-S_t)h\|_2^2\le\min\{\|h\|_2^2,t\int|\nabla h|^2\,d\mu\}$.
Applying \eqref{eq:localization-covariance} to the global eigenfunctions
with positive eigenvalues and summing gives
\begin{equation}\label{eq:localization-negative-two}
 n\E\|\nu_n-S_t\nu_n\|_{-2,\mu}^2
 =\chi_n\sum_j\frac{\|(I-S_t)\psi_j\|_2^2}{\lambda_j^2}
 \le C\sum_j\frac{\min(1,t\lambda_j)}{\lambda_j^2}
 \le C\sqrt t.
\end{equation}
The heat trace gives $\lambda_j\asymp j$; the last bound follows from
$\min(1,t\lambda_j)\le\sqrt{t\lambda_j}$. Lemma~\ref{lem:potential-l2}
transfers this discrepancy estimate to the random optimal potential.
The expected squared $L^2$ norm of the potential can be $o(1/n)$
even when the expected smoothing cost is of order $(\log n)/n$.

\paragraph{Confining transport to cells.}
We keep allocations away from cell boundaries and reassign the remaining
mass within each cell. The potential estimate controls the error caused
by cutting off the original allocation. To turn this control into a
feasible transport plan, we smooth the remaining atomic mass at an
auxiliary time $s=1/n$ and apply the transport bound of
Lemma~\ref{c:lem:energy}. The original smoothing time remains $t$.

\begin{lemma}[Cost of restricting transport to cells]\label{lem:cell-repair}
Use the partition $\cB$ and smoothing $S_t$ defined above, with
intrinsic cell distances. Put $q\mu=S_t\nu_n$ and define
\[
 D=\OT_{\bar d^2}(q\mu,\nu_n),\qquad
 Z=n\sum_{B\in\cB}\OT_B(q\mu,\nu_n).
\]
Here $\OT_B$ transports the restrictions to $B$ using squared intrinsic
distance as cost, so $Z-nD$ is the added cost of restricting transport to cells,
multiplied by $n$.

Assume the cell diameters and Poincar\'e
constants are bounded uniformly. The heat kernels are taken
relative to $\mu|_B$, without normalizing its mass. For their positive
eigenpairs $(\lambda_{B,j},\psi_{B,j})$, normalized in $L^2(\mu|_B)$,
assume, at $s=1/n$, that
\[
 \int_Bd_B(x,y)^2p_s^B(x,y)\,d\mu(y)\le Cs,\qquad
 \sum_{j\ge1}\frac{e^{-2s\lambda_{B,j}}}{\lambda_{B,j}}
          \psi_{B,j}(x)^2\le C\log(en)
\]
uniformly over $B$ and $x\in B$.

Let $\eta$ be a smooth cutoff in the ambient
surface, equal to one within distance $r$ of every cell boundary,
including the physical boundary, with
\[
 0\le\eta\le1,\qquad \mu(\operatorname{supp}\eta)\le\kappa,\qquad
 |\nabla\eta|\le C/r,\qquad |\nabla^2\eta|\le C/r^2,
 \qquad 0<r\le1.
\]

Let $E=\{\|q-1\|_\infty\le\delta\}$, with
$0<\delta\le1/2$. On $E$, let $T=\exp(\nabla\phi)$ be an optimal
ambient map with $\mu(\phi)=0$, and use $L,V$ from
Lemma~\ref{lem:potential-l2}. Suppose
\[
 n\E D\le C\log(en),\qquad n\E[V\1_E]\le C\delta^2\log(en),
 \qquad L^2\le C\sqrt t,\quad L<r/2\quad\hbox{on }E.
\]
Set
\[
 \epsilon=\kappa+r^{-4}(\delta^2+\sqrt t).
\]
If $\epsilon\le1$, then
\[
 0\le\E(Z-nD)\le C\sqrt\epsilon\log(en)+Cn\PP(E^c).
\]
\end{lemma}
\begin{proof}
\emph{Retaining interior allocations and smoothing the remaining mass.}
By symmetry, we construct cellwise transport from $\nu_n$ back to $q\mu$.
Write $\xi=1-\eta$ and $\omega=\xi\circ T$. At each target $y$, return the
fraction $\xi(y)$ along its original allocation. Whenever $\xi(y)>0$,
$y$ is more than $r$ from the boundary; since $L<r/2$, these paths stay
in its cell. The retained allocations leave the atomic mass $\eta\nu_n$.
Before and after smoothing this mass, the intermediate measures are
\[
 \sigma_0=\eta\nu_n+\omega q\mu,\qquad
 \sigma_s=S_s(\eta\nu_n)+\omega q\mu.
\]
Both have mass $\nu_n(B)=\int_Bq\,d\mu$ in every cell.
The first two transports take $\nu_n$ to $\sigma_0$ along retained
allocations and then to $\sigma_s$ by heat smoothing.
Lemma~\ref{c:lem:energy} bounds the final transport from $\sigma_s$ to
$q\mu$ on $E$:
\begin{equation}\label{g:eq:repair-action}
 \sqrt{Z/n}\le\sqrt D+C\sqrt{s\nu_n(\eta)}
                       +C\|\sigma_s-q\mu\|_{-1,\cB}.
\end{equation}
For a signed measure $\zeta$ with zero mass in every cell, the norm is
\[
 \|\zeta\|_{-1,\cB}^2
 =\sum_{B\in\cB}\sum_{j\ge1}
   \frac{|\int_B\psi_{B,j}\,d\zeta|^2}{\lambda_{B,j}}.
\]

\emph{Bounding the remaining discrepancy.}
Center a signed measure in each cell by setting
\[
 (\Z_{\cB}\zeta)|_B=\zeta|_B-\frac{\zeta(B)}{\mu(B)}\mu|_B.
\]
Decompose
\[
 \sigma_s-q\mu=\{S_s(\eta\nu_n)-\eta q\mu\}+q(\omega-\xi)\mu.
\]
The cell masses of the two terms are opposite, so centering them separately
does not change the sum. The first centered term splits into the sampling
discrepancy $U$ and the additional smoothing error $W$:
\[
 U=\Z_{\cB}S_s[\eta(\nu_n-q\mu)],\qquad
 W=(S_s-I)(\eta q)\mu.
\]
The bounds for $U$ and $W$ use the small boundary mass $\kappa$;
the remaining term $q(\omega-\xi)$ requires the potential estimate.
For positive cell eigenpairs normalized in $L^2(\mu|_B)$, covariance and
self-adjointness give
\begin{align*}
 n\E\|U\|_{-1,\cB}^2
 &=\chi_n\sum_{B,j}\frac{e^{-2s\lambda_{B,j}}}{\lambda_{B,j}}
       \|(I-S_t)(\eta\psi_{B,j})\|_2^2\\
 &\le C\sum_B\int_B\eta^2\sum_{j\ge1}
       \frac{e^{-2s\lambda_{B,j}}}{\lambda_{B,j}}
                    \psi_{B,j}(x)^2\,d\mu(x)
 \le C\kappa\log(en).
\end{align*}
Also $\|(S_s-I)h\|_{-1,\cB}^2\le s\|h\|_2^2$, so
$n\E[\|W\|_{-1,\cB}^2\1_E]\le C\kappa$.

For the remaining term, put $v=\nabla\phi$ and
$\diver_\mu=\rho^{-1}\diver(\rho\,\cdot)$.
Taylor expansion gives
\[
 q(\omega-\xi)=\diver_\mu(\phi\nabla\xi)
 -\phi\diver_\mu\nabla\xi+(q-1)\nabla\xi\cdot v+qR,
 \qquad |R|\le Cr^{-2}|v|^2.
\]
The flux vanishes near every cell boundary. Integration by parts,
the cutoff bounds, and the cell Poincar\'e inequality therefore yield
\[
 \|\Z_{\cB}[q(\omega-\xi)\mu]\|_{-1,\cB}^2
 \le C\{r^{-4}V+\delta^2r^{-2}D+r^{-4}L^2D\}.
\]
\emph{Combining the costs.}
The preceding bounds and $r\le1$ give
\[
 n\E[\|\sigma_s-q\mu\|_{-1,\cB}^2\1_E]
 \le C\{\kappa+r^{-4}(\delta^2+\sqrt t)\}\log(en).
\]
Since $ns\E\nu_n(\eta)\le\kappa$, squaring
\eqref{g:eq:repair-action} and applying Cauchy--Schwarz gives
$\E[(Z-nD)\1_E]\le C(\sqrt\epsilon+\epsilon)\log(en)$.
Every cellwise coupling is feasible for the ambient cost, so $Z\ge nD$.
The diameter bounds give $Z\le Cn$ on $E^c$, completing the proof.
\end{proof}

\subsection{Subdivision on the square and torus}\label{app:flat-local}
In this subsection and the next, positions are independent with law $\mu=\mu_\Omega$
on $\Omega\in\{Q,\T\}$, with equal masses or a Dirichlet mass vector
independent of the positions, as in Section~\ref{subsec:flat-heat}.
Recall that $t=n^{-5/6}$, $m\in\{1,2,3\}$, and each cell has side
$\ell=1/m$. Let $\cB$ be the grid and let
$P_t^{\cB}=\bigoplus_{B\in\cB}P_t^B$ smooth within each cell with
reflection at its boundary. Define $q\mu=P_t^{\cB}\nu_n$ and
$Z_{n,m}=n\sum_{B\in\cB}\OT_B(q\mu,\nu_n)$, the sum of the cell transport costs multiplied by $n$, using intrinsic
squared-distance cost.
Reflection preserves each cell's mass:
\[
 \int_Bq\,d\mu=\nu_n(B)\qquad(B\in\cB).
\]
We prove the following stronger estimates, which together imply
Lemma~\ref{g:lem:grid}:
\begin{align}
 \E\bigl|nW_{2,\Omega}^2(P_t\nu_n,\nu_n)-Z_{n,m}\bigr|&\le Cn^{-1/120}\log n,
 \qquad \Omega\in\{Q,\T\},\label{g:eq:grid}\\
 \E Z_{n,m}&=\E d_K^Q(m^2t),
 \qquad K\sim\Bin(n,m^{-2}).\label{g:eq:localexact}
\end{align}
The second identity follows from the sampling law and rescaling.
For the first estimate, put
$\mathcal R_{n,m}=Z_{n,m}-nW_{2,\T}^2(q\mu,\nu_n)\ge0$.
Every cellwise allocation is feasible on the torus; the unrestricted square
cost lies between these two costs. We retain interior edges and repair the
remaining mass in its own cells.

\begin{proof}[Proof of Lemma~\ref{g:lem:grid}]
It suffices to consider large $n$: the finitely many remaining integers
$n\ge2$ are covered by increasing $C$, since both costs are bounded by $2n$.
\emph{Sampling law and rescaling.}
For independent sites with any law $\mu$, let $K_B$ and $S_B$ be their
count and total mass in a fixed cell $B$. Conditional on $K_B=k>0$,
the sites are independent with law $\mu(\cdot\mid B)$; the normalized
masses have the original equal or Dirichlet law, independently of the
sites and $S_B$. In both cases,
\begin{equation}\label{eq:mass-law}
 \E[nS_B\mid K_B=k]=k.
\end{equation}
For Dirichlet masses, write $w_i=E_i/\sum_jE_j$ with independent
rate-one exponential variables. Within a cell, the normalized
exponentials are independent of their sum and of the sites, giving
this restriction law. Symmetry gives
$\E[S_B\mid K_B=j]=j/n$, proving \eqref{eq:mass-law}; for equal
masses, $S_B=j/n$ deterministically.
The local mass law and squared-distance rescaling prove
\eqref{g:eq:localexact}: a nonempty cell with $j$ targets contributes, in
conditional expectation,
$n\ell^2\E[S_B\mid j]d_j(m^2t)/j=\ell^2d_j(m^2t)$.
The rough bound~\eqref{f:eq:rough} therefore gives
$\E Z_{n,m}\le C\log n$.

\paragraph{The potential estimate.}
Let $E=\{\|q-1\|_\infty\le\delta\}$, with
$nt\delta^2\gg\log n$ and $\sqrt t\le\delta^2$ as chosen below.
For $t\le\ell^2$, each reflecting cell kernel $p_t^B(x,\cdot)$,
extended by zero outside $B$, has mean one under the full sampling
law. Rescaling gives supremum norm and variance bounded by $C/t$
for the centered kernels, and $O(\delta^{-2}t^{-3})$ net points over
all cells. Lemma~\ref{lem:weighted-concentration} therefore gives
$\PP(E^c)\le C\delta^{-2}t^{-3}e^{-cnt\delta^2}$.

Regard $q$ as a density on $\T$. Let $T=\Id+\nabla\phi$ be its optimal
periodic map to $\nu_n$, with $\int\phi=0$. Put
$D=\int q|\nabla\phi|^2$, $V=\int\phi^2$, and
$L=\|\nabla\phi\|_\infty$. The heat coupling gives $D\le4t$, and restriction
to cells gives $n\E D\le\E Z_{n,m}\le C\log n$.
On $E$, the source density is at least $1/2$, so
Lemma~\ref{lem:local-cost-localization}, with $s=2$ and cost
$d_{\T}^2$, gives $L^4\le CD$.

The torus inverse Laplacian has the $H^2$ bound required in
Lemma~\ref{lem:potential-l2}, and $L\le Ct^{1/4}$ is eventually below
its injectivity radius. That lemma and the discrepancy estimate
\eqref{eq:localization-negative-two} give
\[
 n\E[V\1_E]\le C\{\sqrt t+(\delta^2+\sqrt t)\log n\}
                  \le C\delta^2\log n.
\]
This bound controls the cutoff derivatives in the repair.
Choose $-\phi$ as a Kantorovich potential for $d_{\T}^2/2$.
Its $c$-transform representation gives a uniform Lipschitz bound;
normalizing $\int\phi=0$ therefore gives $\|\phi\|_\infty\le C$.
Thus $n\E[V\1_{E^c}]\le Cn\PP(E^c)$ is negligible, and the same
potential estimate holds without the indicator.

\paragraph{Repair near the grid.}
Take a smooth periodic cutoff equal to one within distance $r$ of the
grid and zero beyond distance $2r$, with first and second derivatives
bounded by $C/r$ and $C/r^2$. Its support has area $\kappa\le Cmr$.
The cell Poincar\'e constants are uniformly bounded, and the cosine
series gives an integrated positive heat diagonal at time $s$ bounded
by $C\log(e\ell^2/s)$. The reflected Gaussian coupling has mean squared
displacement at most $4s$. These verify the analytic inputs of
Lemma~\ref{lem:cell-repair}.

We choose $r$ larger than the transport distances while keeping both the
boundary mass $mr$ and the cutoff error $r^{-4}(\delta^2+\sqrt t)$ small.
Set $r=n^{-1/60}$, $\delta=n^{-1/24}$, and $s=1/n$. Then
$t=r^{50}$ and $\delta^2=r^5$, so
\[
 \kappa+r^{-4}(\delta^2+\sqrt t)
 \le C(mr+r+r^{21})\le Cmr.
\]
Also $nt\delta^2=n^{1/12}$, and $L<r/2<\ell/16$ for all large $n$.
Lemma~\ref{lem:cell-repair} therefore gives
\begin{equation}\label{g:eq:gap}
 0\le\E\mathcal R_{n,m}\le C\sqrt{mr}\log n
                         \le Cn^{-1/120}\log n.
\end{equation}
The exceptional event contributes negligibly, since both costs are
bounded by $2n$.

\paragraph{Replacing global smoothing by cellwise smoothing.}
Couple global and cellwise smoothing by the same Gaussian increment.
The folded or periodic endpoints agree unless the increment crosses a
cell boundary. At distance $b$ from that boundary their expected squared
separation is at most $Ct e^{-cb^2/t}$. Integrating over uniform initial
positions gives
\[
 \Delta=n\E W_{2,\Omega}^2(P_t^\Omega\nu_n,q\mu)
                \le Cmnt^{3/2}=Cm n^{-1/4}.
\]
For $\Omega=Q$ use reflection globally and within the cells; for $\T$ use
periodic reduction globally. The triangle inequality and
$nW_{2,\Omega}^2(q\mu,\nu_n)\le Z_{n,m}$ give
\[
 \begin{aligned}
 n\E\bigl|W_{2,\Omega}^2(P_t\nu_n,\nu_n)
          -W_{2,\Omega}^2(q\mu,\nu_n)\bigr|
 &\le\Delta+2\sqrt{\Delta\,\E Z_{n,m}}\\
 &\le C\sqrt{\Delta\log n}.
 \end{aligned}
\]
Here $\Delta\log n\to0$, so the cross term is $o(1)$. Finally,
\[
 W_{2,\T}^2(q\mu,\nu_n)
 \le W_{2,Q}^2(q\mu,\nu_n)
 \le Z_{n,m}/n.
\]
This sandwich and \eqref{g:eq:gap} imply \eqref{g:eq:grid} for both metrics.
\end{proof}

\subsection{Sample-size comparisons}\label{app:flat-counts}
We derive the sample-size recurrence~\eqref{g:eq:twosided} and transfer
the Poisson-size asymptotics~\eqref{g:eq:poissonrate} back to fixed sample
sizes. Set $t_n=n^{-5/6}$, $u_n=d_n^Q(t_n)$ for $n\ge1$, and $u_0=0$,
as in Section~\ref{sec:flat}. Write $\chi=2$ for Dirichlet masses and
$\chi=1$ for equal masses.

\paragraph{Rescaled heat times.}
Fix $m\in\{2,3\}$, put $k=n/m^2$ and $\tau=m^2t_n$, and let
$K\sim\Bin(n,m^{-2})$. On $k/2\le K\le2k$, the times $\tau$ and $t_K$
both tend to zero and satisfy $K\tau\asymp Kt_K\asymp n^{1/6}$.
The change-of-time formula~\eqref{f:eq:time-change} therefore gives
\[
 d_K^Q(\tau)-u_K
 =\frac{\chi_K}{4\pi}
    \left\{\log m^2+\frac56\log(K/n)\right\}
       +O(n^{-1/24}\log n).
\]
On this event, replacing $\chi_K$ by $\chi$ changes the expression by
$O(n^{-1}\log n)$. Binomial tails and the rough
bound~\eqref{f:eq:rough} control the complementary event, including
$K=0$. Taylor expansion of the logarithm on the central event gives
\[
 \E[\1_{\{K>0\}}\log K]=\log k+O(k^{-1}).
\]
Combining these estimates with Lemma~\ref{g:lem:grid} proves
\eqref{g:eq:twosided}.

\paragraph{Averaging over the sample size.}
The thinning calculation in Section~\ref{sec:flat} gives
\eqref{g:eq:two-scales} by averaging \eqref{g:eq:twosided} over
$J\sim\Pois(m^2N)$. On $m^2N/2\le J\le2m^2N$, the expected error
is $O(N^{-1/120}\log(eN))$; Poisson tails and $u_j\le C\log(ej)$
control the complement.

\paragraph{Monotonicity at a fixed heat time.}
For a fixed $t>0$, define
$F_t(\eta)=W_{2,Q}^2(P_t\eta,\eta)$ for probability measures $\eta$.
Linearity of $P_t$ and joint convexity of squared transport cost make
$F_t$ convex. Deleting target $i$ from $\nu_{n+1}$ and renormalizing
gives $\nu_n^{(-i)}$ and the exact mixture
\[
 \nu_{n+1}=\frac1n\sum_{i=1}^{n+1}(1-w_i)\nu_n^{(-i)}.
\]
For Dirichlet masses, the normalized remaining masses are independent
of $w_i$ and have the $n$-target Dirichlet law. The remaining positions
also have the required independent uniform law. For equal masses the
coefficients are deterministic. Taking expectations in the convexity
inequality consequently gives
\begin{equation}\label{g:eq:smoothing-monotonicity}
 \frac{d_{n+1}^Q(t)}{n+1}\le\frac{d_n^Q(t)}n.
\end{equation}
We use this inequality at fixed $t$ and then apply the change-of-time
formula to account for the varying times $t_n$.

\paragraph{Returning to a fixed sample size.}
Write $\bar d(N,t)=\E d_{\Pois(N)}^Q(t)$.
Fix a large integer $n$, hold $t=t_n$ fixed, and put
$N_\pm=n\pm n^{2/3}$. On $J\ge n$, monotonicity gives
$d_J^Q(t_n)\le Ju_n/n$; on $1\le J\le n$ it gives the reverse
inequality. Poisson tails and the rough cost bound yield
\[
 \frac n{N_+}
   \{\bar d(N_+,t_n)-O(e^{-cn^{1/3}}\log n)\}
 \le u_n\le
 \frac{n\bar d(N_-,t_n)}{N_--O(ne^{-cn^{1/3}})}.
\]
The zero count contributes zero.

For $|J-n|\le2n^{2/3}$, both $Jt_n$ and $Jt_J$ have order $n^{1/6}$,
so the change-of-time formula gives
\[
 d_J^Q(t_n)-u_J
 =\frac{5\chi}{24\pi}\log(J/n)+O(n^{-1/24}\log n).
\]
The logarithm is $O(n^{-1/3})$ on this event. Poisson tails and
\eqref{f:eq:rough} control the complement for either mean $N_\pm$,
so
\[
 \bar d(N_\pm,t_n)=\bar u(N_\pm)
                   +O(n^{-1/3}+n^{-1/24}\log n).
\]
Insert this into the preceding sandwich and use
\eqref{g:eq:poissonrate}. The factors $n/N_\pm=1+O(n^{-1/3})$
and the change from $\log N_\pm$ to $\log n$ contribute
$O(n^{-1/3}\log n)$. We obtain
\[
 u_n=\frac{\chi}{24\pi}\log n+\kappa_\star
 +O(n^{-1/120}\log(en)),
\]
which proves~\eqref{g:eq:canonical-smoothing}.

\subsection{Subdivision on surfaces}\label{app:surface-proof}
We prove Theorem~\ref{thm:surface-smoothing} under the assumptions of
Section~\ref{subsec:surface-reduction}. First take $t=t_n=n^{-5/6}$:
we construct suitable cells, confine transport to them, and compare their
smoothing costs with the square. The final change-of-time comparison
extends the formula to the stated range of $t$.
Constants may depend on the fixed metric and positive density, but are
uniform over the cells and their sizes.

\subsubsection{Cells and local estimates}\label{app:cells}
The repair estimate requires uniform geometry in every cell; comparison
with the square also requires nearly Euclidean coordinates in most cells.
We call a neighborhood of the boundaries a \emph{collar}, and the
additional cells filling collars along chart and physical boundaries
\emph{seam cells}.
\begin{lemma}[A nearly square partition]\label{s:lem:cells}
For all small $\ell>0$, a surface satisfying the assumptions of Section~\ref{subsec:surface-reduction} has a partition $\cB_\ell$ into $O(\ell^{-2})$ connected cells. Each has diameter at most $C\ell$, area $A_B\asymp\ell^2$, and probability $p_B=\mu(B)\asymp\ell^2$. After rescaling by $\ell$, each coordinate domain is uniformly bi-Lipschitz to a square or a reference triangle. The number of boundary pieces, their Lipschitz chart constants, and the bi-Lipschitz constants are bounded uniformly. In particular, with $d_B$ the intrinsic cell distance,
\[
 d_M(x,y)\le d_B(x,y)\le C d_M(x,y)\quad(x,y\in\overline B),
 \qquad
 \mu\bigl(\{y\in B:d_B(x,y)<r\}\bigr)\asymp r^2
 \quad(0<r\le\ell).
\]
The cell boundaries have total length $O(\ell^{-1})$. Let $\Sigma_\ell$ include
all cell boundaries, including the physical boundary of the surface.
For a surface with boundary, take distances to $\Sigma_\ell$ in a fixed
smooth extension of the metric and denote this distance by $\bar d$;
for a closed surface put $\bar d=d_M$. Then
\begin{equation}\label{s:eq:collar}
 \mu\{x:\bar d(x,\Sigma_\ell)<r\}\le Cr/\ell\qquad(0<r\le\ell).
\end{equation}
There is a smooth function $\eta_r$ equal to one at distance at most $r$ from $\Sigma_\ell$ and zero at distance at least $2r$, with
$|\nabla\eta_r|\le C/r$ and $|\nabla^2\eta_r|\le C/r^2$.

Except for cells of total area $O(\ell)$, every cell is \emph{good}: there is a map $\Phi_B:Q\to B$ such that
$\Phi_{B\#}dx=\mu(\cdot\mid B)$ and the pulled-back metric satisfies
\begin{equation}\label{s:eq:chart}
 (1-C\ell)A_B I\preceq \Phi_B^*g\preceq(1+C\ell)A_B I.
\end{equation}
\end{lemma}
\begin{proof}
We adapt the domain decompositions of Ambrosio, Goldman, and
Trevisan~\cite{Domain}, keeping track of the boundary collars and the
uniform geometry of the cells. Choose a smooth triangulation subordinate to finitely many isothermal
charts, using charts on a smooth extension when $M$ has boundary.
Include its physical boundary arcs and corner vertices in the
triangulation. The finitely many sides have bounded curvature, finite
total length, and angles bounded away from zero.

Retain lattice squares of side length $\ell$ at distance at least
$20\ell$ from these boundaries. The omitted collar has width $C\ell$
and area $O(\ell)$. Away from vertices, use arclength and normal
coordinates, whose Jacobians and inverses are uniformly bounded. Divide
the sides into intervals of lengths between $\ell$ and $2\ell$, and the
collar into strips of comparable width, lying inside $M$ at a physical
boundary.

Absorb partial lattice squares into adjacent strip elements. Group pieces
narrower than a fixed multiple of $\ell$ before triangulating, so the
rescaled elements have bounded aspect ratios and numbers of sides. Each
element meets boundedly many lattice squares; the width $20\ell$ allows
these adjustments within an $O(\ell)$ collar. At every original vertex,
including physical corners, subdivide the fixed-angle sector into such
strips and triangles. The finitely many angle bounds give uniform
constants.

Each completed cell consequently has area comparable to $\ell^2$ and is uniformly bi-Lipschitz to a reference square or triangle after rescaling. The metric and positive density are uniformly bounded in the fixed charts. This proves the distance and ball-volume comparisons, including for seam cells.

Each cell has perimeter $O(\ell)$. Summing the collar bound
$C(\ell r+r^2)$ over $O(\ell^{-2})$ cells proves \eqref{s:eq:collar}.
Construct $\eta_r$ by mollifying a Lipschitz cutoff of
$\bar d(\cdot,\Sigma_\ell)/r$ in ambient charts at scale $cr$, with
transition margins strictly between $r$ and $2r$. A fixed smooth
partition of unity combines these cutoffs and preserves their values
zero and one. Differentiation gives the bounds, including at corners.

It remains to choose coordinates that send $\mu(\cdot\mid B)$ to the
uniform law on $Q$ for each good cell $B$. Rescale its
coordinate square to $Q$. The metric is $A_B(I+O(\ell))$: in isothermal
coordinates its scalar coefficient has relative variation $O(\ell)$ and
integral $A_B$. The density $p$ of $\mu(\cdot\mid B)$ with respect to
$dx$ satisfies $\|p-1\|_{C^1(Q)}=O(\ell)$. We use the triangular
change of variables of the Knothe--Rosenblatt
rearrangement~\cite[Chapter~1]{Villani}:
\[
 p_1(s)=\int_0^1p(s,t)\,dt,\qquad
 T(s,t)=\left(\int_0^sp_1(u)\,du,
             \frac{\int_0^tp(s,u)\,du}{p_1(s)}\right).
\]
Then $\det DT=p$, $DT=I+O(\ell)$, and $T$ maps each face to itself,
so $T_\#(p\,dx)=dx$. Composing the coordinate map with $T^{-1}$ gives
$\Phi_B$ and preserves the metric estimate, proving \eqref{s:eq:chart}.
\end{proof}

\paragraph{Verifying the local heat estimates.}
Take $P_t^B$ and $p_t^B$ relative to the unnormalized measure $\mu|_B$.
Lemma~\ref{s:lem:cells}, together with smoothness and positivity of
the metric and density, verifies the assumptions of
Lemma~\ref{lem:reflected-heat} uniformly, including on seam cells.
The invariant density is $1/p_B$.

We need concentration and a rough local cost bound. Write $P_t^{\cB}\nu_n=q_t\mu$. For $t\le\ell^2$, Lemma~\ref{lem:weighted-concentration} and \eqref{s:eq:cellholder} give
\begin{equation}\label{s:eq:localconc}
 \PP(\|q_t-1\|_\infty>\delta)
 \le C t^{-1}(\delta t)^{-2/\alpha}\exp(-cnt\delta^2),
 \qquad 0<\delta\le1,
\end{equation}
whenever the net scale below is smaller than $c\ell$.
Indeed, at a fixed point the kernel minus its expectation has size at most $C/t$ and variance at most $C/t$; the latter also follows from
$\int_Bp_t^B(x,y)^2\,d\mu(y)=p_{2t}^B(x,x)$.
Its mean against the full sampling law is one, even though the semigroup has a separate constant mode in every cell. A net of mesh
$c\sqrt t(\delta t)^{1/\alpha}$ makes the spatial oscillation of a probability-weighted kernel sum at most $\delta/2$. The union of cellwise nets has at most
$Ct^{-1}(\delta t)^{-2/\alpha}$ points.

For the cost bound, let $\nu_j$ have $j$ independent targets with law $\mu(\cdot\mid B)$
and either equal or Dirichlet masses. Write $d_j^B(t)$ for $j$ times the
expected squared transport cost between $\nu_j$ and $P_t^B\nu_j$.
Smoothing at time $A_B/j$, the energy bound of Lemma~\ref{c:lem:energy},
and the integrated heat diagonal give
\[
 j\E W_{2,B}^2(\mu(\cdot\mid B),\nu_j)\le CA_B\log(ej).
\]
The same energy estimate for the smoothed endpoint gives
$j\E W_{2,B}^2(\mu(\cdot\mid B),P_t^B\nu_j)
\le CA_B\log(e+A_B/t)$, with the logarithm replaced by a constant when $t\ge A_B$.
Comparing the two endpoints through $\mu(\cdot\mid B)$ proves, for $t\ge1/n$ and $j\le n$,
\begin{equation}\label{s:eq:localrough}
 d_j^B(t)\le CA_B\log(en).
\end{equation}

\subsubsection{Confining transport to cells}\label{app:surface-confinement}
We compare the smoothing cost on $M$ with the sum of the cell costs by
changing global smoothing to cellwise smoothing and restricting transport
to each cell. Put $\cB=\cB_\ell$, write $P_t^{\cB}\nu_n=q\mu$, and set
\[
 D_n(t)=nW_{2,M}^2(P_t\nu_n,\nu_n),\qquad
 Z_n=n\sum_{B\in\cB}\OT_B(q\mu,\nu_n).
\]
Here $\OT_B$ uses squared intrinsic distance within $B$ as cost, and cellwise heat
smoothing preserves the mass in each cell. The conditional mass
law \eqref{eq:mass-law} gives
\begin{equation}\label{s:eq:localexpect}
 \E Z_n=\sum_B\E d_{K_B}^B(t),\qquad K_B\sim\Bin(n,p_B).
\end{equation}
Indeed, with $S_B=\nu_n(B)$ and conditional on $K_B=j>0$, the factor
$nS_B/j$ has expectation one and is independent of the normalized local
problem.

Keep the flat repair scales
\[
r=n^{-1/60},\qquad t=r^{50},\qquad s=r^{60},\qquad\delta^2=r^5.
\]
The cell size $\ell$ is the only new choice. The errors from metric
distortion and from cells near chart boundaries are $O(\ell\log n)$,
whereas the repair costs $O(\sqrt{r/\ell}\log n)$. Balancing them gives
\[
 \ell=\sqrt{r/\ell}=r^{1/3}=n^{-1/180}.
\]
In particular, the repair width $r$ is much smaller than a cell.

\begin{lemma}\label{s:lem:local}
For the partition $\cB_\ell$ of Lemma~\ref{s:lem:cells}, with
$\ell=n^{-1/180}$ and $t=n^{-5/6}$,
\[
\E|D_n(t)-Z_n|\le C_{M,\rho}n^{-1/180}\log(en).
\]
\end{lemma}
The following coupling controls the change in smoothing: the two
processes agree until they reach a cell boundary.

\Needspace{5\baselineskip}
\begin{lemma}\label{s:lem:smoothingcoupling}
For the partition $\cB_\ell$ of Lemma~\ref{s:lem:cells}, let
$P_t^{\cB}$ be the cellwise Neumann semigroup. For either equal or
Dirichlet masses and $0<t\le c\ell^2$,
\begin{equation}\label{s:eq:smoothingcoupling}
 n\E W_{2,M}^2(P_t\nu_n,P_t^{\cB}\nu_n)
 \le Cnt^{3/2}/\ell.
\end{equation}
\end{lemma}
\begin{proof}
Start both diffusions at $x\in B$. Their generators agree inside $B$,
so couple them to coincide until their first hit $\tau$ of $\partial B$.
When $M$ has boundary, use its Neumann diffusion globally and stop at
the whole cell boundary, including its physical part. After $\tau$,
continue the two corresponding processes from their common boundary point.

Let $b=\bar d(x,\partial B)$. Before this hit, neither process has reached
a physical boundary, so both are the smooth ambient diffusion. In fixed
coordinate charts its martingale part has bounded quadratic variation
and its drift is bounded. The exponential maximal inequality, with the
process stopped on leaving a fixed chart, gives
\[
 \PP_x(\tau\le t)\le
 \PP_x\!\left(\sup_{u\le t}\bar d(Z_u,x)\ge b\right)
 \le C\exp(-cb^2/t).
\]
For $b$ of order $t$ or smaller, enlarge $C$ and use the bound one; a finite atlas makes the constants uniform elsewhere. Conditional on the hit, the global heat displacement bound and \eqref{s:eq:cellmovement}, applied from the common boundary point, bound each expected squared remaining displacement by $C(t-\tau)$. For a surface with boundary, the global bound follows from the same reflected heat-kernel estimate used in Lemma~\ref{s:lem:cellheat}, on the fixed domain $M$. The strong Markov property and the triangle inequality consequently give a coupling of the endpoint kernels with
\[
 \E_x d_M(Z_t,Z_t^B)^2\le Ct\,\PP_x(\tau\le t)
 \le Ct\exp(-cb^2/t).
\]

Mix these couplings with the masses of $\nu_n$. Since $\E\nu_n=\mu$,
the expected cost of this coupling, multiplied by $n$, is at most
$Cnt\int_M\exp\{-c\,\bar d(x,\Sigma_\ell)^2/t\}\,d\mu(x)$.
The collar estimate \eqref{s:eq:collar}, integrated in distance, bounds this integral by $C\sqrt t/\ell$; distances greater than $\ell$ have an exponentially smaller contribution. This proves \eqref{s:eq:smoothingcoupling}.
\end{proof}

To control the restriction on transport, we use the potential and repair
estimates in a smooth extension of $M$. The next lemma bounds the error
when we return to intrinsic distance.

\begin{lemma}[Intrinsic and ambient distances]\label{s:lem:ambient}
Extend $g$ smoothly to a closed surface containing $M$, and write $\bar d$
for the distance in this extension and $d$ for the intrinsic distance in
$M$. If $M$ is closed, take $\bar d=d$. Suppose any corners of $M$ have
interior angle less than $\pi$. There are constants $c,C>0$ such that
\[
 c d(x,y)\le\bar d(x,y)\le d(x,y),\qquad
 0\le d(x,y)^2-\bar d(x,y)^2\le C\bar d(x,y)^3.
\]
For $1/2\le q\le2$, $\mu(q)=1$, and a probability measure $\nu$, set
$\bar D=\OT_{\bar d^2}(q\mu,\nu)$. Then
\[
 0\le W_{2,M}^2(q\mu,\nu)-\bar D\le C\bar D^{5/4}.
\]
The constants are uniform over $q$ and $\nu$.
\end{lemma}
\begin{proof}
Local boundary charts compare intrinsic and ambient distances; compactness
then gives the first assertion. To prove the cubic bound, let
$a=\bar d(x,y)$ be small and take an ambient minimizing geodesic.
Near a smooth boundary piece, its depth outside $M$ is at most $Ca^2$,
by the bounded Hessian of the signed distance to that piece. Retraction
along the boundary normals has differential norm at most $1+Ca^2$
on this part of the geodesic. Retraction therefore produces a path in
$M$ of length at most $a(1+Ca^2)$.

At a corner, flatten the two boundary arcs to their tangent rays and
freeze the metric at the vertex. The resulting sector is convex because
its angle is less than $\pi$. If either endpoint lies within $C_0a$ of
the vertex, both endpoints lie in a chart of radius $Ca$; the coordinate
change and the varying metric distort lengths by at most $1+Ca$.
The straight segment in the sector gives $d(x,y)\le a(1+Ca)$.
Choose $C_0$ large enough in terms of the finitely many corner angles.
In the remaining case the short geodesic can meet at most one boundary
arc, and the smooth-boundary argument applies. This proves the cubic
bound for small $a$; compactness extends it to all pairs.

Intrinsic balls have $q\mu$ mass at least a constant times the square of
their radius, uniformly for $1/2\le q\le2$. Apply
Proposition~\ref{prop:local-cost-deterministic} to $k=\bar d^2$ and the
intrinsic metric $d$. Quadratic comparability turns its error into
$C\bar D^{5/4}$, proving the cost estimate.
\end{proof}

\begin{proof}[Proof of Lemma~\ref{s:lem:local}]
We first apply the cell repair estimate to transport from $q\mu$ to
$\nu_n$ in the ambient metric.
The local concentration bound \eqref{s:eq:localconc} gives
$E=\{\|q-1\|_\infty\le\delta\}$ with $\PP(E^c)=O(n^{-N})$ for every fixed $N$, since $nt\delta^2=n^{1/12}$. Let $T(x)=\exp_x(\nabla\phi(x))$ be the
optimal map from $q\mu$ to $\nu_n$ for quadratic cost in the ambient metric, with
$\mu(\phi)=0$.
Write $v=\nabla\phi$, $D=\int q|v|^2\,d\mu$, $V=\int\phi^2\,d\mu$, and
$L=\|v\|_\infty$. The local heat coupling gives $D\le Ct$.
The rough cell estimate \eqref{s:eq:localrough}, restored with each cell's mass, gives $n\E D\le C\log n$.
Lemma~\ref{lem:local-cost-localization} and the distance comparison in
Lemma~\ref{s:lem:ambient} give $L^4\le CD$ on $E$. Thus
$L^2\le C\sqrt t$.

The transport distances are shorter than the fixed ambient injectivity radius.
Lemma~\ref{b:lem:heat} supplies the Neumann $H^2$ inverse and ambient
extension required by Lemma~\ref{lem:potential-l2}. That lemma and
\eqref{eq:localization-negative-two}, with $q\mu=P_t^{\cB}\nu_n$, give
\[
 n\E[V\1_E]\le C\{\sqrt t+(\delta^2+\sqrt t)\log n\}
                  \le C\delta^2\log n.
\]

The cutoff from Lemma~\ref{s:lem:cells} includes the physical boundary
and has collar mass $\kappa=Cr/\ell$. Lemma~\ref{s:lem:cellheat} supplies
the displacement, integrated heat diagonal, and Poincar\'e bounds used
in Lemma~\ref{lem:cell-repair}. Also $L<r/2$ for large $n$, and
\[
 \kappa+r^{-4}(\delta^2+\sqrt t)
 \le C(r/\ell+r+r^{21})\le Cr/\ell.
\]
The repair lemma consequently gives
$0\le\E(Z_n-nD)\le C\sqrt{r/\ell}\log n$.

We next restore intrinsic distance. Lemma~\ref{s:lem:ambient} gives, on $E$,
\[
 0\le W_{2,M}^2(q\mu,\nu_n)-D\le CD^{5/4}\le Ct^{1/4}D.
\]
Since $n\E D\le C\log n$, this contributes at most $Ct^{1/4}\log n$
to $n\E W_{2,M}^2(q\mu,\nu_n)$. The exceptional event contributes less than any
required inverse power of $n$ by the diameter bound.

Finally, restore global heat smoothing. Lemma~\ref{s:lem:smoothingcoupling} gives
\[
 n\E W_2^2(P_t\nu_n,P_t^{\cB}\nu_n)\le Cnt^{3/2}/\ell.
\]
The triangle inequality bounds the resulting change in $\E D_n(t)$ by
$C\sqrt{nt^{3/2}\ell^{-1}\log n}$, which is smaller than
$\sqrt{r/\ell}\log n$ at the stated scales.
\end{proof}

\subsubsection{Comparing and summing cell costs}\label{app:surface-cell-comparison}
For a good cell $B$, the coordinates $\Phi_B^{-1}$ from
Lemma~\ref{s:lem:cells} send $\mu(\cdot\mid B)$ to the uniform law on $Q$.
To prove \eqref{eq:surface-cell-comparison}, we
compare distance and heat smoothing in these coordinates.
Put $\Phi=\Phi_B$ and $g_B=\Phi_B^*g$. By \eqref{s:eq:chart}, the metric $g_B$ lies between
$(1-\epsilon)A_B I$ and $(1+\epsilon)A_B I$, with $\epsilon=C\ell$.
Thus, for any probability measures $\eta,\zeta$ on $Q$,
\begin{equation}\label{s:eq:coordinate-cost}
 (1-\epsilon)A_B W_{2,Q}^2(\eta,\zeta)
 \le W_{2,B}^2(\Phi_\#\eta,\Phi_\#\zeta)
 \le(1+\epsilon)A_B W_{2,Q}^2(\eta,\zeta).
\end{equation}
Let $e^{-t\mathsf A}$ be the cell's heat smoothing in these coordinates,
with Neumann quadratic form on $L^2(Q,dx)$
$\mathfrak a(h,h)=\int g_B^{-1}[\nabla h,\nabla h]\,dx$.
Compare it with $\mathsf B=A_B^{-1}(-\Delta_Q)$, whose form is
$\mathfrak b(h,h)=A_B^{-1}\int|\nabla h|^2dx$. Their forms differ by
relative error $O(\ell)$.

The negative norm below will bound transport between the two smoothings
through Lemma~\ref{c:lem:energy}. Set
$T_t=e^{-t\mathsf A}-e^{-t\mathsf B}$. For $\nu_j$ supported on $j$
independent uniform points in $Q$, with either mass law,
\[
 j\E\|T_t\nu_j\|_{-1,\mathsf A}^2
   =\chi_j\|\mathsf A^{-1/2}T_t\|_{\HS}^2,
 \qquad \|\zeta\|_{-1,\mathsf A}=\|\mathsf A^{-1/2}\zeta\|_2.
\]
Here $\|\cdot\|_{\HS}$ is the Hilbert--Schmidt norm on mean-zero functions.
The identity follows by covariance in an orthonormal basis; $T_t$
annihilates constants.

\begin{lemma}\label{s:lem:duhamel}
Let $\mathsf A,\mathsf B$ be nonnegative self-adjoint Neumann divergence-form
operators on $L^2(Q,dx)$ with common kernel consisting of constants.
Suppose $\mathsf B=a^{-1}(-\Delta_Q)$ for some $a>0$. Write $\mathfrak a,\mathfrak b$ for their closed quadratic forms, both with domain $H^1(Q)$, and assume
\[
 |\mathfrak a(h,h)-\mathfrak b(h,h)|\le\epsilon\,\mathfrak b(h,h)
 \qquad(h\in H^1(Q)),\qquad 0\le\epsilon\le\tfrac12.
\]
Then, for $0<t\le a$,
\begin{equation}\label{s:eq:HS}
 \|\mathsf A^{-1/2}(e^{-t\mathsf A}-e^{-t\mathsf B})\|_{\HS}^2
 \le C\epsilon^2a.
\end{equation}
The norm is on the mean-zero subspace. The constant is uniform over
uniformly elliptic coefficients satisfying the stated form comparison.
\end{lemma}
\begin{proof}
After removing constants, the form assumption and polarization give a bounded operator
$T=\mathsf A^{-1/2}(\mathsf A-\mathsf B)\mathsf B^{-1/2}$ with
$\|T\|\le\epsilon(1-\epsilon)^{-1/2}$. Duhamel's identity, interpreted first in finite spectral
approximations and then in the form sense, gives
\[
 \mathsf A^{-1/2}(e^{-t\mathsf B}-e^{-t\mathsf A})
 =\int_0^t e^{-(t-s)\mathsf A}T\mathsf B^{1/2}e^{-s\mathsf B}\dd s.
\]
Form comparison and the square spectrum give
$\|e^{-u\mathsf A}\|_{\HS}\le C\sqrt{a/u}$ and
$\|\mathsf B^{1/2}e^{-u\mathsf B}\|_{\HS}\le C\sqrt a/u$ for $u\le a$.
Also $\|\mathsf B^{1/2}e^{-u\mathsf B}\|\le Cu^{-1/2}$.
For $s\le t/2$ use the first and third bounds; for $s\ge t/2$ use the
second. The integral of Hilbert--Schmidt norms is bounded by
\[
 C\epsilon\sqrt a\left(
 \int_0^{t/2}\frac{\dd s}{\sqrt{s(t-s)}}+
 \int_{t/2}^t\frac{\dd s}{s}\right)\le C\epsilon\sqrt a.
\]
Squaring proves the claim.
\end{proof}
\paragraph{Proof of the cell comparison.}
For $j$ between $np_B/2$ and $2np_B$, we have
$jt/A_B\asymp nt=n^{1/6}$. Flat heat concentration therefore makes the
flat smoothed density at least $1/2$ outside events of probability
$O(n^{-N})$ for every fixed $N$. The covariance identity and
Lemma~\ref{s:lem:duhamel} bound the expected energy discrepancy by
$C\ell^2A_B/j$. Restricted to the event that the flat density is at least
$1/2$, the expected squared transport distance satisfies the same bound
by Lemma~\ref{c:lem:energy}.
Together with \eqref{s:eq:localrough}, the triangle inequality then
bounds the change in $d_j^B(t)$ by
$C\ell A_B\sqrt{\log n}$. Comparing metrics through
\eqref{s:eq:coordinate-cost} adds $C\ell A_B\log n$, while the diameter
bound controls the exceptional event. Hence
\begin{equation}\label{s:eq:freeze}
 d_j^B(t)=A_B d_j^Q(t/A_B)+O(\ell A_B\log n)+O(n^{-N}A_B).
\end{equation}

\paragraph{Completing the smoothing-cost formula.}
Apply Theorem~\ref{thm:smoothing-cost} at $\tau=t/A_B$.
Its conditions hold uniformly for $np_B/2\le j\le2np_B$:
$j\tau\asymp n^{1/6}$, $j\asymp n\ell^2$, and
$\tau\asymp n^{-5/6}\ell^{-2}$. Write $\chi=2$ for Dirichlet masses
and $\chi=1$ for equal masses, and use the corresponding $\kappa_\star$.
We now justify the summation in Section~\ref{subsec:surface-smoothing}. For $K_B\sim\Bin(n,p_B)$, the
probability of leaving $[np_B/2,2np_B]$ is exponentially small.
Taylor expansion and binomial variance give
$\E[\log K_B;K_B>0]=\log(np_B)+O((np_B)^{-1})$;
\eqref{s:eq:localrough} controls the omitted tails.
The cells without probability coordinates have total area $O(\ell)$,
so their actual and candidate contributions are both $O(\ell\log n)$.

Using \eqref{s:eq:localexpect}, the errors from the cell comparison
\eqref{s:eq:freeze}, Theorem~\ref{thm:smoothing-cost}, and the random
cell counts sum to at most
\[
 C\left\{\ell\log n+(n\ell^2)^{-1/120}\log n
       +\sqrt{\frac{(\log(en))^3}{nt}}
       +\frac{\sqrt t}{\ell}+(n\ell^2)^{-1}\right\}
       +O(n^{-N}).
\]
This is $O(n^{-1/180}\log n)$: besides $\ell=n^{-1/180}$, the
relevant powers are $(n\ell^2)^{-1/120}=n^{-89/10800}$,
$(nt)^{-1/2}=n^{-1/12}$, and $\sqrt t/\ell=n^{-37/90}$.
The density Riemann sum in Section~\ref{subsec:surface-smoothing} has
error $O(\ell)$ because $\rho$ is smooth and positive. Together with
Lemma~\ref{s:lem:local}, these bounds prove~\eqref{eq:surface-smoothing}
at $t_n=n^{-5/6}$ with error $O(n^{-1/180}\log(en))$.

For general $t$ in the stated range, subtract
Proposition~\ref{prop:surface-reduction} at $t_n$ and $t$ and use
\eqref{eq:h-general}. The spectral finite part cancels, giving
\[
 \begin{aligned}
 d_n^{M,\mu}(t)-d_n^{M,\mu}(t_n)
 &=\frac{\chi A}{4\pi}\log(t/t_n)\\
 &\quad+O_{M,g,\rho}\!\left(n^{-1/180}\log(en)
       +\sqrt{\frac{(\log(en))^3}{nt}}+\sqrt t\right).
 \end{aligned}
\]
The errors at $t_n$ and the $O(\log(en)/n)$ change from replacing
$\chi_n$ by $\chi$ are absorbed into $O(n^{-1/180}\log(en))$.
This proves Theorem~\ref{thm:surface-smoothing}.
\section{The distinct-to-same-distribution reduction}
Appendix~\ref{app:cost-comparison} bounds the change in expected cost
when the population dual gap is replaced by squared-distance cost in its
induced metric, as required by Theorem~\ref{thm:distinct-to-same}.
Appendix~\ref{app:flat-source-target} verifies the assumptions for the
Euclidean convex domains and flat tori of Theorem~\ref{thm:source-target}.

\subsection{Replacing the population dual gap by squared-distance cost}\label{app:cost-comparison}
The following comparison supplies the error term in
Theorem~\ref{thm:distinct-to-same}. The cubic difference between the two
costs gains a factor of the transport distance over the quadratic cost.
Proposition~\ref{prop:local-cost-deterministic} quantifies this gain;
we use a moment bound for semidiscrete transport to show that the expected
total error tends to zero.

\begin{lemma}\label{lem:cost-comparison}
Let $(M,g)$ and $\mu=\rho\,dA_g$ satisfy the assumptions of
Section~\ref{subsec:surface-reduction}. Suppose a continuous cost
$k:M\times M\to\R$ satisfies, for some $a_0,b_0,K>0$,
\begin{equation}\label{eq:local-cost-assumptions}
 a_0d_g^2\le k\le b_0d_g^2,
 \qquad |k-d_g^2|\le Kd_g^3.
\end{equation}
For two independent $n$-point samples $X_i,Z_j$ with law $\mu$,
\begin{equation}\label{eq:local-cost-discrete}
 \left|\E\min_\sigma\sum_i k(X_i,Z_{\sigma(i)})
       -\E\min_\sigma\sum_i d_g(X_i,Z_{\sigma(i)})^2\right|
 \le Cn^{-1/4}\log^{5/4}(en).
\end{equation}
For $\eta_n=n^{-1}\sum_j\delta_{Z_j}$, the same bound holds for
$n|\E\OT_k(\mu,\eta_n)-\E W_{2,g}^2(\mu,\eta_n)|$.
The constant $C$ depends only on $M,g,\rho,a_0,b_0,K$.
\end{lemma}

\begin{proof}
Write $\eta_n^W=\sum_jW_j\delta_{Z_j}$, where $W$ is independent of
the target positions and either $W=(1/n,\ldots,1/n)$ or
$W\sim\Dir(1,\ldots,1)$, and put
$Q=W_{2,g}^2(\mu,\eta_n^W)$.
The positive source density gives $\mu(B(x,r))\ge cr^2$, so
Proposition~\ref{prop:local-cost-deterministic} yields
\[
 n\E\bigl|\OT_k(\mu,\eta_n^W)-W_{2,g}^2(\mu,\eta_n^W)\bigr|
 \le Cn\E Q^{5/4}.
\]
We will bound $\E Q^2$ by $Cn^{-2}\log^2(en)$, which also controls
$\E Q^{5/4}$ by H\"older's inequality. Only an upper bound is needed here,
so the elementary smoothing estimate at $t=1/n$ suffices.
Use $\cL,P_t$ from
Section~\ref{subsec:surface-reduction}. Lemma~\ref{c:lem:energy}, applied
with reference density $q=1$, bounds $W_{2,g}^2(\mu,P_t\eta_n^W)$ by
$4\|P_t\eta_n^W-\mu\|_{-1,\mu}^2$. Coupling each sampled point to its
heat-smoothed mass costs at most $Ct$ by~\eqref{s:eq:cellmovement}.
The triangle inequality therefore gives
\[
 Q\le Ct+C\|P_t\eta_n^W-\mu\|_{-1,\mu}^2.
\]
To express the second term as the squared norm of a sum, let
$(\lambda_j,\psi_j)$ be the positive eigenpairs
of $\cL$, with $(\psi_j)$ orthonormal in $L^2(\mu)$, and write
\[
 \begin{gathered}
 F_t(z)=\left(\frac{e^{-t\lambda_j}}{\sqrt{\lambda_j}}
                    \psi_j(z)\right)_{j\ge1}\in\ell^2,\\
 \|P_t\eta_n^W-\mu\|_{-1,\mu}^2
       =\left\|\sum_iW_iF_t(Z_i)\right\|_{\ell^2}^2.
 \end{gathered}
\]
The heat-kernel estimate~\eqref{s:eq:celldiag} gives
$\sup_z\|F_t(z)\|_{\ell^2}^2\le C\log(e/t)$ for $0<t\le1$.
Conditional on $W$, the vectors $F_t(Z_i)$ are independent and centered.
Expanding the fourth power gives
\[
 \E\!\left[\left\|\sum_iW_iF_t(Z_i)\right\|_{\ell^2}^4\,\middle|\,W\right]
 \le3\sup_z\|F_t(z)\|_{\ell^2}^4\left(\sum_iW_i^2\right)^2.
\]
For equal masses, $(\sum_iW_i^2)^2=n^{-2}$; the Dirichlet moment
formula gives $\E(\sum_iW_i^2)^2\le Cn^{-2}$ as well. Thus
$\E Q^2\le Ct^2+Cn^{-2}\log^2(e/t)$. Taking $t=1/n$ yields
\[
 n\E Q^{5/4}\le n(\E Q^2)^{5/8}
 \le Cn^{-1/4}\log^{5/4}(en).
\]
For Dirichlet masses, the averaging identity (Theorem~\ref{thm:averaging})
identifies the two expected matching costs with their weighted
semidiscrete mean costs, proving~\eqref{eq:local-cost-discrete}.
Equal masses give the stated semidiscrete bound.
\end{proof}

For Theorem~\ref{thm:distinct-to-same}, the expansion
$k=d_{g_T}^2+O(d_{g_T}^3)$ gives~\eqref{eq:local-cost-assumptions} near
the diagonal. Since $k>0$ off the diagonal, compactness extends these
bounds to all of $M\times M$.

\subsection{Verification for Euclidean convex domains and flat tori}\label{app:flat-source-target}
To apply Theorem~\ref{thm:distinct-to-same} in the two settings of
Theorem~\ref{thm:source-target}, we verify the global cost bounds
in~\eqref{eq:local-cost-assumptions} and the assumptions on the induced
metric at the boundary. Use $T=\nabla\Phi$ and
$\lambda I\preceq H=D^2\Phi\preceq\Lambda I$ from
Section~\ref{subsec:torus}. This is the Hessian metric studied in
optimal transport geometry by Kolesnikov~\cite{Kolesnikov}.
Write $\Phi^*$ for the convex conjugate of $\Phi$.

\paragraph{Convex domains.}
On the convex domain $\Omega$, the quadratic dual potentials
$a(x)=|x|^2-2\Phi(x)$ and $b(y)=|y|^2-2\Phi^*(y)$ give
\[
 k(x,z)=2\{\Phi(x)-\Phi(z)-\nabla\Phi(z)\cdot(x-z)\}.
\]
Integrating $H$ along the segment from $z$ to $x$ gives
$\lambda|x-z|^2\le k(x,z)\le\Lambda|x-z|^2$ and
$k(x,z)=(x-z)^\top H(z)(x-z)+O(|x-z|^3)$.

The squared intrinsic distance $d_H(x,z)^2$ is the minimum of the energy
$\int_0^1\dot\gamma(s)^\top H(\gamma(s))\dot\gamma(s)\,ds$
over paths $\gamma$ from $z$ to $x$ in $\Omega$.
Write $r=|x-z|$. The straight segment gives the upper bound
$(x-z)^\top H(z)(x-z)+O(r^3)$.
For the lower bound, take a minimizing path. Its energy is at most
$\Lambda r^2$, so ellipticity gives $\int_0^1|\dot\gamma|^2\,ds\le Cr^2$;
the path stays within distance $Cr$ of $z$.
Replacing $H(\gamma(s))$ by $H(z)$ therefore changes its energy by
$O(r^3)$. Cauchy--Schwarz bounds the resulting energy below by
$(x-z)^\top H(z)(x-z)$. Thus, uniformly up to the boundary,
\[
 d_H(x,z)^2=(x-z)^\top H(z)(x-z)+O(|x-z|^3)
           =k(x,z)+O(|x-z|^3).
\]
Ellipticity also gives
$\lambda|x-z|^2\le d_H(x,z)^2\le\Lambda|x-z|^2$, hence the
required global comparison of $k$ and $d_H^2$. At a corner, a linear
change of coordinates that makes $H$ at the vertex the identity preserves the
convexity of the tangent cone; the transformed cone still has angle
in $(0,\pi)$.
The metric therefore satisfies the surface assumptions, with the same
piecewise smooth boundary and convex corners.

\paragraph{Flat tori.}
On a torus, the additional issue is that each point has many lattice
representatives. We identify periodic dual potentials and minimize over
these representatives to obtain the population dual gap.
For smooth positive periodic densities, periodic transport
regularity~\cite[Theorem~1]{CE} gives uniform ellipticity of $D^2\Phi$.
Differentiating $q(\nabla\Phi)\det D^2\Phi=p$ and applying elliptic
regularity then gives smoothness of $T=\nabla\Phi$.
Let $\Omega=\R^2/\Gamma$, write $[x]$ for the class of $x$ modulo
$\Gamma$, and set
$d_0([x],[y])=\min_{m\in\Gamma}|x-y+m|$.
Write $\Phi(x)=|x|^2/2+u(x)$ with $u$ $\Gamma$-periodic, as in
Section~\ref{subsec:torus}. The translation laws
\[
 \begin{aligned}
 \Phi(x+m)&=\Phi(x)+m\cdot x+\tfrac12|m|^2,\\
 \Phi^*(y+m)&=\Phi^*(y)+m\cdot y+\tfrac12|m|^2
 \qquad(m\in\Gamma)
 \end{aligned}
\]
show that $v(y)=\Phi^*(y)-|y|^2/2$ is periodic.
Fenchel's inequality for every representative $y+m$, $m\in\Gamma$, gives
$-2u(x)-2v(y)\le d_0([x],[y])^2$, with equality when
$y=T(x)$. Thus $a=-2u$ and $b=-2v$ are dual potentials on the torus.

Set $D_\Phi(x,z)=\Phi(x)-\Phi(z)-\nabla\Phi(z)\cdot(x-z)$.
The exact population dual gap is
\[
 k([x],[z])=2\min_{m\in\Gamma}D_\Phi(x+m,z).
\]
The bounds on $H$ give
$\lambda d_0^2\le k\le\Lambda d_0^2$; the same bounds hold for
$d_H^2$. Near the diagonal, choose representatives with $|x-z|$ small.
The $m=0$ term is $O(|x-z|^2)$, while uniform convexity bounds all
terms with $m\ne0$ away from zero. The minimum is therefore attained
at $m=0$ there, and Taylor expansion gives
\[
 k([z+\delta],[z])
   =\delta^\top H(z)\delta+O(|\delta|^3)
   =d_H([z+\delta],[z])^2+O(|\delta|^3).
\]
Consequently $k$ is smooth near the diagonal, vanishes only on it,
and has the positive quadratic part $H$ required by
Theorem~\ref{thm:distinct-to-same}.

\paragraph{Reversing the distributions.}
The formula in Theorem~\ref{thm:source-target} is unchanged when $\mu$
and $\nu$ are exchanged. The reverse map $T^{-1}$ induces the metric
$H_{\mathrm{rev}}(y)=D(T^{-1})(y)$ and density
$\rho_{\mathrm{rev}}(y)=q(y)/\sqrt{\det H_{\mathrm{rev}}(y)}$.
Since $H_{\mathrm{rev}}(Tx)=H(x)^{-1}$,
\[
 DT(x)^{\mathsf T}H_{\mathrm{rev}}(Tx)DT(x)=H(x),\qquad
 \rho_{\mathrm{rev}}(Tx)
 =q(Tx)\sqrt{\det H(x)}=\rho_H(x).
\]
Thus $T$ is a measure-preserving isometry between the two induced
weighted surfaces. Their areas, spectra, and density integrals agree.
\section{Unsquared distance cost for ordered product distributions}
\label{app:linear-proof}
\begin{proof}[Proof of Proposition~\ref{prop:linear-product}]
We construct a matching with expected cost $nW_1(\mu,\nu)+O(\log n)$
by matching horizontal ranks within blocks of about $\sqrt n$ consecutive
vertical ranks. This choice bounds the expected squared vertical
displacement of each pair by $C/n$.
Positive horizontal displacements make the additional cost of vertical
mismatch quadratic; small or negative horizontal displacements require
a separate bound.
Write $X_i=(X_{i1},X_{i2})$ and $Y_i=(Y_{i1},Y_{i2})$.
All constants below depend only on $\lambda,\Lambda,\gamma$.

Every pair satisfies $|Y-X|\ge Y_1-X_1$. Since $F\ge G$, the
increasing horizontal quantile coupling attains equality while keeping
the vertical coordinate fixed. Hence
\[
 W_1(\mu,\nu)=\E Y_{11}-\E X_{11}
             =\int_0^1(F-G)(x)\,dx.
\]
The signed horizontal cost is independent of the assignment. Subtracting it gives
\begin{equation}\label{eq:linear-residual}
 \begin{aligned}
 R_n&:=M_n^{(1)}-\sum_jY_{j1}+\sum_iX_{i1}\\
    &=\min_\sigma\sum_i
      \bigl\{|Y_{\sigma(i)}-X_i|-(Y_{\sigma(i),1}-X_{i1})\bigr\}\ge0.
 \end{aligned}
\end{equation}
Its expectation is the excess matching cost,
$\E R_n=\E M_n^{(1)}-nW_1(\mu,\nu)$.
For horizontal and vertical displacements $d,v$,
\[
 0\le\sqrt{d^2+v^2}-d
 \le2(-d)_++\min\left\{|v|,\frac{v^2}{2|d|}\right\},
\]
where the second entry is omitted when $d=0$.
The two terms control backward movement and the additional cost of
vertical displacement, respectively.

\paragraph{Keeping vertical displacements small.}
For $n\ge16$, partition $\{1,\ldots,n\}$ into
$B=\lceil\sqrt n\rceil$ consecutive blocks with sizes $m_b$ differing
by at most one. In particular,
\begin{equation}\label{eq:linear-block-sizes}
 c\sqrt n\le m_b\le C\sqrt n,
 \qquad B\le C\sqrt n,\qquad \sum_bm_b=n.
\end{equation}
Sort each sample by height and divide it into these rank blocks.
Within corresponding blocks, pair points by horizontal rank.
Let $U_{(i)},V_{(i)}$ be the ordered vertical coordinates of the two
samples. The product structure makes them independent of the horizontal
coordinates indexed by vertical rank, which are independent samples
with densities $p_1$ and $q_1$.

For ranks $i,j$ in the same block, the uniform order-statistic moments give
\[
 \E|U_{(i)}-V_{(j)}|^2
 \le \frac{1}{2(n+2)}+\frac{(i-j)^2}{(n+1)^2}
 \le C/n.
\]
The choice $m_b\asymp\sqrt n$ keeps the squared difference in mean
heights at $O(1/n)$, the same bound as for the variances.
Let $\mathcal H$ denote all horizontal coordinates after vertical
sorting. Conditional on $\mathcal H$, the assigned vertical ranks are
fixed. Thus every assigned vertical displacement $v_k$ satisfies
\begin{equation}\label{eq:linear-vertical-moments}
 \E[v_k^2\mid\mathcal H]\le C/n,\qquad
 \E[|v_k|\mid\mathcal H]\le C/\sqrt n.
\end{equation}

\paragraph{Horizontal gaps.}
Consider a block of size $m$, with source horizontal coordinates $A_1,\ldots,A_m$ of density
$p_1$ and target horizontal coordinates $B_1,\ldots,B_m$ of density $q_1$. Write
$d_k=B_{(k)}-A_{(k)}$ for the signed horizontal displacement and
$r_k=\min(k,m+1-k)$ for the rank counted from the nearer endpoint.
The ordering assumption~\eqref{eq:ordered-marginals} makes the forward
gap typically at least $cr_k/m$. Combined with the vertical moment bound,
this gives an expected additional cost of at most $Cm/(nr_k)$ for such pairs; summing over
ranks produces the logarithm. We now bound the probability of a smaller
gap and the expected total backward movement. The first bound is
\begin{equation}\label{eq:linear-rank-gap}
 \PP\{d_k<c_0r_k/m\}\le Ce^{-c_1r_k}.
\end{equation}
For $0<s<1$, put $x=F^{-1}(s)$. Since $x\ge s/\Lambda$,
$1-x\ge(1-s)/\Lambda$, and $G$ is $\Lambda$-Lipschitz,
assumption~\eqref{eq:ordered-marginals} gives
\[
 G^{-1}(s)-F^{-1}(s)
 \ge\frac{F(x)-G(x)}{\Lambda}
 \ge\frac{\gamma}{\Lambda^3}s(1-s)
 \ge\kappa\min(s,1-s),\qquad
 \kappa=\frac{\gamma}{2\Lambda^3}.
\]
Both inverse distribution functions are $1/\lambda$-Lipschitz.

To pass to sample quantiles, let $Z_{(k)}$ be a uniform order statistic
and $s_k=k/(m+1)$. For each fixed $0<\eta\le1/4$, binomial Chernoff
bounds yield
\begin{equation}\label{eq:linear-order-tail}
 \PP\{|Z_{(k)}-s_k|>\eta\min(s_k,1-s_k)\}
 \le C_\eta e^{-c_\eta r_k}.
\end{equation}
Indeed, for $k\le(m+1)/2$, the two tails are
$\{\Bin(m,s_k(1-\eta))\ge k\}$ and
$\{\Bin(m,s_k(1+\eta))\le k-1\}$, so multiplicative Chernoff bounds
give $C_\eta e^{-c_\eta k}$.
Reflection covers the remaining ranks.

Represent $A_{(k)}$ and $B_{(k)}$ by applying $F^{-1}$ and $G^{-1}$
to independent uniform order statistics. Choose
$\eta\le\min(1/4,\kappa\lambda/4)$. Outside the two exceptional
events in \eqref{eq:linear-order-tail},
\[
 d_k\ge(\kappa-2\eta/\lambda)\frac{r_k}{m+1}
      \ge\frac{\kappa r_k}{4m}.
\]
This proves \eqref{eq:linear-rank-gap}.

To bound the penalty for backward movement, we also use that horizontal
displacements near the endpoint ranks are small.
For $k\le(m+1)/2$, the bounds $F^{-1}(s),G^{-1}(s)\le s/\lambda$
and
\[
 \E Z_{(k)}^2=\frac{k(k+1)}{(m+1)(m+2)}
\]
give $\E d_k^2\le C(k/m)^2$. Applying the same argument to
$1-A_{(k)}$ and $1-B_{(k)}$ covers the remaining ranks. Hence
$\E d_k^2\le C(r_k/m)^2$ for every $k$. Cauchy--Schwarz and
\eqref{eq:linear-rank-gap} now yield
\begin{equation}\label{eq:linear-backward}
 \E\sum_{k=1}^m(-d_k)_+
 \le \sum_{k=1}^m(\E d_k^2)^{1/2}\PP(d_k<0)^{1/2}
 \le \frac{C}{m}\sum_{k=1}^m r_ke^{-cr_k}
 \le C/m.
\end{equation}

\paragraph{Bounding the excess matching cost.}
In a block of size $m$, use the
quadratic bound on vertical cost when $d_k\ge c_0r_k/m$ and the
linear bound otherwise. Combining the conditional vertical moments
\eqref{eq:linear-vertical-moments} with
\eqref{eq:linear-backward} and \eqref{eq:linear-rank-gap} gives
\[
 \begin{split}
 \E\sum_{k=1}^m\big(\sqrt{d_k^2+v_k^2}-d_k\big)
 &\le C/m+\frac{C}{\sqrt n}\sum_k e^{-c_1r_k}
        +\frac{Cm}{n}\sum_k r_k^{-1}\\
 &\le C/m+C/\sqrt n+Cm\log(em)/n.
 \end{split}
\]
The first two terms, from backward movement and vertical cost at small
gaps, sum to $O(1)$ over all blocks. The third contains the harmonic sum
that gives the logarithm.
Summing over blocks and using \eqref{eq:linear-block-sizes} gives
\[
 \E R_n
 \le C\sum_bm_b^{-1}+CB/\sqrt n
       +\frac{C}{n}\sum_bm_b\log(em_b)
 \le C\log(en).
\]
Together with \eqref{eq:linear-residual}, this proves the proposition
for $n\ge16$. Enlarging $C$ covers the remaining values of $n$.
\end{proof}
\par
\end{document}